\documentclass[11pt,reqno]{amsart}
\usepackage{amsthm, amsfonts,  amsmath, amssymb}
\usepackage{etoolbox}

\makeatletter
\patchcmd{\@maketitle}
    {\@setauthors}
    {\customauthors}
    {}{\typeout{Author patch failed}}

\newcommand{\customauthors}{%
    \begin{center}
        \normalsize
        Aman Patel$^{1}$\\
        \footnotesize\texttt{amanpat@iitk.ac.in}\\[2mm]

        \normalsize
        Mrinmay Biswas$^{1}$\\
        \footnotesize\texttt{mbiswas@iitk.ac.in}\\[2mm]

        \normalsize
        Mohmedmunavvar Mubarak Bapu$^{2}$\\
        \footnotesize\texttt{munavvar.bapu87@gmail.com}\\[3mm]

        \small
        $^{1}$Department of Mathematics,\\
        Indian Institute of Technology Kanpur,\\
        Kanpur--208016, India\\[2mm]

        $^{2}$Department of Mathematics,\\
        Indian Institute of Technology Bombay,\\
        Mumbai--400076, India
    \end{center}
}
\makeatother
\usepackage{dsfont}
\allowdisplaybreaks
\usepackage[all,cmtip]{xy}
\usepackage{tkz-euclide}
\usepackage{tikz-cd}
\usetikzlibrary{arrows}
\usepackage{mathtools}  % for paired delimiters\usepackage{comment}
\usepackage{mathrsfs}   % for script capitals
\usepackage[all]{xy}
\usepackage{url,enumerate}
\usepackage{amsmath}
\usepackage{setspace}
\usepackage[a4paper, top = 1.2in, bottom = 1.2in, left = 1.2in, right = 1.2in]{geometry}
\usepackage{cases}
\usepackage{hyperref}
\UseComputerModernTips

\newtheorem{thm}{Theorem}[section]
\newtheorem{lem}[thm]{Lemma}
\newtheorem{prop}[thm]{Proposition}

\theoremstyle{definition}

\newtheorem{dfn}[thm]{Definition}
\usepackage[OT2,T1]{fontenc}
\DeclareSymbolFont{cyrletters}{OT2}{wncyr}{m}{n}
\DeclareMathSymbol{\Sha}{\mathalpha}{cyrletters}{"58}

\newtheorem{remark}[thm]{Remark}
\theoremstyle{plain}

\numberwithin{equation}{section}

\newcommand{\N}{\mathbb{N}}

\newcommand{\R}{\mathbb{R}}

\newcommand{\mcA}{\mathcal{A}}
\newcommand{\mcD}{\mathcal{D}}
\newcommand{\mcH}{\mathcal{H}}

\newcommand{\mcL}{\mathcal{L}}

\newcommand{\mcP}{\mathcal{P}}
\newcommand{\mcR}{\mathcal{R}}
\newcommand{\mcS}{\mathcal{S}}

\newcommand{\cmat}[2]{ \begin{pmatrix} #1 \\ #2 \end{pmatrix}}
\newcommand{\csmat}[2]{\big( \begin{smallmatrix} #1 \\ #2 \end{smallmatrix}\big)}

\newcommand{\mcal}[1]{\mathcal{#1}}

\begin{document}
	\title[Approximate controllability of the GBHE]{Approximate controllability of the generalized Burgers-Huxley equation}    

\author[A. Patel]{Aman Patel}
%\email{amanpat@iitk.ac.in}
%\address{
%	Department of Mathematics \\
%	Indian Institute of Technology Kanpur\\
%	Kanpur-208016\\
%	INDIA
%    }
    \author[M. Biswas]{Mrinmay Biswas}
%\email{mbiswas@iitk.ac.in}
%\address{
%	Department of Mathematics \\
%	Indian Institute of Technology Kanpur\\
%	Kanpur-208016\\
%	INDIA
%}
\author[M. M. Bapu]{Mohmedmunavvar Mubarak Bapu}
%\email{munavvar.bapu87@gmail.com}
%\address{
%	Department of Mathematics \\
%	Indian Institute of Technology Bombay\\
%	Mumbai-400076\\
%	INDIA
%}

	\doublespacing
	
	\begin{abstract}
        In this article, we investigate the controllability of the generalized Burgers–Huxley equation (GBHE) by means of an interior control supported in an arbitrary open subinterval compactly contained in the domain. Using weighted energy estimates, we establish an obstruction to approximate controllability that is independent of the terminal time. Nevertheless, we prove uniform quasi-static tracking of steady states belonging to the same connected component, with a quantitative error bound in $H_0^1(0,1)$, thereby establishing approximate controllability between such states. The proof relies on a weighted spectral analysis of the parameter-dependent linearized operator and the construction of a finite-dimensional feedback control. The existence and uniqueness of strong solutions to the nonlinear perturbed closed-loop system are established using the Leray–Schauder fixed point theorem and suitable a priori estimates. A Lyapunov functional defined at the natural energy level enables us to control the nonlinear convection and reaction terms.
		
		\noindent \textit{Keywords.} Generalized Burgers-Huxley equation; Approximate controllability; Quasi-static deformation;
Finite-dimensional feedback; Leray-Schauder fixed point theorem; Lyapunov functional
		\newline
		\noindent \textit{2020 Mathematics Subject Classification.} 93B05,	93B52, 93C20, 93D15, 93D05,  35Q92, 35K57
	\end{abstract} \maketitle

%\usepackage[notref]{showkeys} %this_is_for_display_of_cite_name_and_equations

%\date{}

%\author{Patel A, Bapu M, Biswas M}
%\onehalfspacing

%\begin{document}

\maketitle

\section{Introduction}
    The generalized Burgers-Huxley equation (GBHE) is a nonlinear partial differential equation that arises in several areas, including fluid mechanics, transport theory, population dynamics, and nerve signal propagation. It takes the form (see~\cite{satsuma1987exact})
    \begin{equation}
    \label{equation_one}
         u_t = u_{xx} - \alpha u^{\delta} u_x + \beta u(1-u^{\delta})(u^{\delta}-\gamma),
    \end{equation}
    where $\alpha, \beta, \gamma >0$ and $\delta\in\mathbb{N}$ are parameters. This equation comprises three main components: the diffusion term $u_{xx}$, which accounts for dissipative effects; the nonlinear convection term $-\alpha u^{\delta} u_x$, which describes transport phenomena; and the nonlinear reaction term $\beta u(1-u^{\delta})(u^{\delta}-\gamma)$, which models growth and damping interactions. The combined presence of these mechanisms makes the mathematical analysis of the equation more involved. The case $\delta=1$ includes several well-known evolution equations:
    \begin{itemize}
        \item[(i)] If $\alpha\not=0$ and $\beta\not=0$ then  \eqref{equation_one} reduces to the classical Burgers-Huxley equation (see~\cite{Wang1990solitary, wazwaz2009pde,Yefimova2004});
        \item[(ii)] If $\alpha=0$ and $\beta\neq 0$, then \eqref{equation_one} becomes the Huxley equation, a reaction-diffusion model arising in the study of liquid-crystal dynamics and signal transmission in nerve fibers.  (see~\cite{wang1985nerve});
        \item[(iii)] If $\beta=0$ and $\alpha=1$ then \eqref{equation_one} reduces to the viscous Burgers equation, which has been extensively used to study shock formation, turbulence, and nonlinear wave propagation (see~\cite{ Bateman1915, Burgers1948, Burgers1974}). The viscous Burgers equation can be transformed into the linear heat equation and solved explicitly via the Hopf-Cole transformation (see~\cite{cole1951, hopf1950}). For further reading on the theoretical aspects of the viscous Burgers equation, we refer to~\cite{Benia2016, Bressan1986}.
    \end{itemize}
    
    Several analytical properties of the generalized Burgers-Huxley equation, including existence of solutions, stability, wave propagation, asymptotic behaviour, solitary and travelling-wave solutions have been investigated by numerous authors; see, for example, \cite{abdou2005, khattak2007, mohan2021gbh, Wazwaz2005} and the references therein.
  
    Over the past few decades, extensive research has been devoted to the controllability and stabilization of semilinear and quasilinear parabolic equations owing to their wide range of applications in physics, biology, and engineering; see, for instance, \cite{coron2007, fernandez2014, fursikov1996, lions1988, zuazua2005}. In particular, null controllability, approximate controllability and stabilization problems for nonlinear parabolic equations have been extensively investigated by using Carleman estimates, observability inequalities, spectral methods and fixed point techniques. Controllability problems for convection-diffusion equations and Burgers type systems have also attracted considerable attention because of the interaction between nonlinear transport and diffusion effects. From the control theoretic point of view, feedback stabilization problems for generalized Burgers-Huxley equations have received attention in recent years. Boundary feedback stabilization results were obtained in \cite{singh2023boundary}, where the authors proved global existence of strong solutions together with exponential stabilization results under suitable boundary controls. Feedback stabilizability around nonconstant stationary solutions and Burgers-Huxley equations with memory were further investigated in \cite{akram2025feedback,bag2026boundary}. Motivated by these developments, in the present article we investigate the approximate controllability properties of the generalized Burgers-Huxley equation by means of controls acting on a fixed proper subset of the spatial domain. 

    In this article, we study the following internally controlled generalized Burgers-Huxley equation with homogeneous Dirichlet boundary conditions
    \begin{equation}\label{cntrl_eqn}
        \begin{cases}
            u_t=u_{xx}-\alpha u^{\delta} u_x+\beta u(1-u^{\delta})(u^{\delta}-\gamma)+\chi_{\omega} f,& \text{in } (0,T)\times(0,1),\\
            u(t,0)=u(t,1)=0,& t\in(0,T),\\
            u(0,x)=u_0(x),& x\in(0,1),
        \end{cases}
    \end{equation}
    where $\alpha,\beta,\gamma,\delta$ are as in \eqref{equation_one}, and  $f\in L^2((0,T)\times(0,1))$ is a distributed control localized in the subdomain $\omega:=(b,c)\subset(0,1)$ with $0<b<c<1$, and $\chi_{\omega}$ denotes the characteristic function of $\omega$.

    The present work extends the quasi-static deformation framework developed by Coron-Trélat \cite{Coron2004} and subsequently adapted to the FitzHugh-Nagumo equation in \cite{Shirshendu20} to the generalized Burgers-Huxley equation. The presence of the nonlinear convection term $-\alpha u^\delta u_x$ gives rise to a parameter-dependent first-order term in the linearized operator and additional derivative nonlinearities in the perturbed system. To investigate the spectral properties of the linearized operator, we introduce a parameter-dependent weight that reduces the associated eigenvalue problem to an equivalent Sturm-Liouville problem; see Subsection~\ref{Spectral_Analysis_A}. This enables us to construct a finite-dimensional feedback control through a suitable spectral decomposition and pole-shifting technique.

    An important part of our analysis concerns the well-posedness of the nonlinear perturbed closed-loop system arising from the quasi-static deformation procedure. The presence of nonlinear convection, polynomial reaction terms, and a time-dependent feedback operator requires suitable a priori estimates. By introducing an associated homotopy problem and applying the Leray-Schauder fixed point theorem, we establish the existence of strong solutions to the perturbed system. Uniqueness is subsequently established by an energy argument.

    Furthermore, our proof does not rely on the uniform boundedness result established in Theorem~3.21 of \cite{Shirshendu20}, and may therefore be useful for systems for which such a uniform bound is unavailable or difficult to verify. Moreover, unlike the approach in \cite{Shirshendu20}, where stronger regularity of the closed-loop solution is established in Theorem~3.19 before introducing the Lyapunov functional, our Lyapunov functional is well defined at the natural energy level; see Remark~\ref{rem:Lyapunov_energy_level}. By estimating the nonlinear remainder terms and deriving a suitable differential inequality for the Lyapunov functional, we establish a uniform quasi-static tracking estimate of order $O(\epsilon)$ in $H_0^1(0,1)$.
   
    To the best of our knowledge, this is the first study addressing the approximate controllability of the GBHE via finite-dimensional feedback control.

    In this article, we are interested in the approximate controllability of the system~\eqref{cntrl_eqn}. We first recall the notion of approximate controllability for system \eqref{cntrl_eqn}. 

\begin{dfn}[Approximate controllability in $L^2(0,1)$]
\label{def:approx_controllability}
    Let $T>0$. The system \eqref{cntrl_eqn} is said to be approximately controllable in time $T$ in $L^2(0,1)$ if, for every $u_0,u_T\in L^2(0,1)$ and every $\epsilon>0$, there exists a control $f\in L^2((0,T)\times(0,1))$ such that the corresponding solution $u$ of \eqref{cntrl_eqn} with initial datum $u_0$ satisfies
    \begin{equation*}
        \|u(T)-u_T\|_{L^2(0,1)}\le \epsilon.
    \end{equation*}
\end{dfn}

    Our first result concerning approximate controllability is the following:
\begin{thm}
\label{thm1}
    Assume that either $1\leq\delta\leq2$ or
    $\beta>2^{2\delta}\alpha^2$. For every $u_0\in L^2(0,1)$,
    there exist $u_1\in C_c^\infty((b/4,b/2))$ and
    $\eta_{0}>0$ such that, for every $T>0$ and every control
    $f\in L^2((0,T)\times(0,1))$, the corresponding solution
    of \eqref{cntrl_eqn} satisfies
    \begin{equation*}
        \|u(T,\cdot)-u_1\|_{L^2(0,1)}>\eta_{0}.
    \end{equation*}
    In particular, system \eqref{cntrl_eqn} is not approximately
    controllable in $L^2(0,1)$, even when the terminal time is
    allowed to depend on the initial state, the target and the
    prescribed accuracy.
\end{thm}
\begin{remark}
    The assumption that either $1\le\delta\le2$ or $\beta>2^{2\delta}\alpha^2$ in Theorem~\ref{thm1} is imposed solely to ensure the uniqueness of weak solutions for initial data in $L^2(0,1)$. If the initial datum satisfies $u_0\in H_0^1(0,1)$, then the weak solution is unique under the standing assumptions on the parameters specified in \eqref{equation_one}. Consequently, the conclusion of Theorem~\ref{thm1} remains valid without these additional assumptions when $u_0\in H_0^1(0,1)$. Moreover, in this case, the system is not approximately controllable in $H_0^1(0,1)$. Here approximate controllability in $H_0^1(0,1)$ is understood by replacing the $L^2(0,1)$-norm in Definition~\ref{def:approx_controllability} with the $H^1_{0}$-norm.
\end{remark}

    The proof of Theorem~\ref{thm1}, given in Section~\ref{LOGAC},
is based on suitable weighted energy estimates inspired by
the approach developed in \cite{Furshikov_notapprox}.

    Before stating our next result, we provide the definition of steady states.
\begin{dfn}[Steady state]
    A function $u\in C^2([0,1])$  is called a steady state of the control system \eqref{cntrl_eqn} if $u$ satisfies
    \begin{equation}
    \label{steady_state}
        \begin{cases}
            u_{xx} -\alpha u^{\delta} u_x+\beta u(1-u^{\delta})(u^{\delta} -\gamma) = 0,\;\;\;& {\text{in}} \;\; (0,b)\cup(c,1),\\
             u(0)=u(1)=0.
        \end{cases}
    \end{equation}
\end{dfn}

    Let $\mcS$ denote the set of all steady states endowed with the $C^2$-topology. We address the following controllability problem: is it possible to drive the system arbitrarily close to a target steady state, provided that the initial and target states belong to the same connected component of $\mcS$? More strongly, we investigate whether the solution can be made to remain uniformly close to a slowly varying path of steady states connecting them. The following theorem establishes uniform quasi-static tracking with arbitrarily small error in $H_0^1(0,1)$, and consequently yields approximate controllability between these steady states in sufficiently large time.

\begin{thm}\label{Main_thm}
    Let $u_0,u_1 \in \mcS$ lie in the same connected component of $\mcS$.
     Then there exist a path $\overline{u}\in C^1([0,1];C^2([0,1]))$ with $\overline{u}(0)=u_0,\, \overline{u}(1)=u_1$ and $\overline{u}(\tau)\in\mcS$ for every $\tau\in[0,1]$, and constants $\epsilon^*_0,\, C>0$ such that, for every $0<\epsilon \leq \epsilon^*_0$, there exists a control $f \in L^2((0,1/\epsilon)\times(0,1))$ for which the corresponding solution $u$ of \eqref{cntrl_eqn} satisfies
    \begin{equation}
    \label{Uniform_tracking}
        \sup_{0\le t\le 1/\epsilon}\|u(t,\cdot)-\overline{u}(\epsilon t,\cdot)\|_{H_0^1(0,1)}\le C\epsilon.
    \end{equation}
    In particular, for any given $\widehat{\delta}>0$, there exists $T_{\widehat{\delta}}>0$ such that for every $T \ge T_{\widehat{\delta}},$ one can find a control $f \in L^2((0,T)\times(0,1))$ such that the solution $u$ of \eqref{cntrl_eqn} satisfies
    \begin{equation}
    \label{u_T_minus_u_1_leq_delta}
        \|u(T, \cdot)-u_1\|_{H^{1}_{0}(0,1)} \le \widehat{\delta}.
    \end{equation}
\end{thm}

    To prove Theorem~\ref{Main_thm}, we adapt the quasi-static 
deformation method introduced by Coron-Tr\'elat 
\cite{Coron2004, Coron06} for semilinear heat and wave equations 
(see also \cite{Coron2002, Schmidt06}). We first linearize the 
system along a path of steady states connecting the initial state 
$u_0$ to the target state $u_1$. By reducing the associated 
parameter-dependent eigenvalue problem to an equivalent weighted 
Sturm-Liouville problem, we obtain a spectral decomposition into 
an infinite-dimensional stable part and a finite-dimensional part 
containing all unstable modes. The latter is stabilized by 
constructing a finite-dimensional feedback law \eqref{Feedback} 
with a constant diagonal gain matrix. The existence and uniqueness 
of strong solutions to the resulting nonlinear perturbed closed-loop 
system \eqref{Close_loop_z_eqn} are then established by combining 
the Leray-Schauder fixed point theorem with suitable a priori 
estimates and an energy argument; see Theorem~\ref{Thm_Soln_Trans_sys}. 
Finally, by constructing a Lyapunov functional 
\eqref{Lyapunov_Functional} on the quadratic-form domain of the 
linearized operator and estimating the nonlinear remainder terms, 
we derive the uniform tracking estimate \eqref{Uniform_tracking} 
of order $O(\epsilon)$ in $H_0^1(0,1)$ over the time interval 
$[0,1/\epsilon]$.

While classical local controllability results primarily concern 
driving the system to a prescribed terminal state, our objective 
is to control the evolution of the solution along a path of steady 
states. More precisely, the control strategy based on the 
finite-dimensional feedback law \eqref{Feedback} ensures that the 
solution remains uniformly close to the reference path over the 
entire time interval $[0,1/\epsilon]$, with an error of order 
$O(\epsilon)$ in $H_0^1(0,1)$; see \eqref{Uniform_tracking}. Thus, 
approximate controllability between steady states is obtained as 
a consequence of a quantitative trajectory-tracking result.

\subsection*{Our contributions}
    The main contributions of this work are as follows:
    \begin{itemize}
        \item First, we show that, for every initial state $u_0\in L^2(0,1)$ under the assumptions of
        Theorem~\ref{thm1}, there exist smooth target states that cannot be reached approximately, irrespective of the time allowed. More precisely, the $L^2$-norm of every reachable state on a suitable subinterval disjoint from the control region satisfies an explicit bound depending on $u_0$ and the fixed parameters, but independent of the control and the terminal time.

        \item Second, we establish the well-posedness of the nonlinear 
    perturbed closed-loop system \eqref{Close_loop_z_eqn} arising 
    from the quasi-static deformation procedure. By deriving 
    suitable a priori estimates for an associated homotopy problem 
    and applying the Leray-Schauder fixed point theorem, we prove 
    the existence of strong solutions in the presence of nonlinear 
    convection and a time-dependent feedback operator. Uniqueness 
    follows from an energy argument; see 
    Theorem~\ref{Thm_Soln_Trans_sys}.
        
        \item Third, we construct a finite-dimensional feedback law 
    \eqref{Feedback} that ensures uniform quasi-static tracking of 
    a path of steady states with an error of order $O(\epsilon)$ 
    in $H_0^1(0,1)$. The analysis relies on the weighted spectral 
    decomposition of the linearized operator and a Lyapunov 
    functional defined at the natural energy level. As a 
    consequence, we establish approximate controllability between 
    steady states belonging to the same connected component 
    of $\mcS$ in sufficiently large time; see 
    Theorem~\ref{Main_thm}.
    \end{itemize}
    
    The weighted estimate does not characterize the set of approximately reachable states, nor does it settle
    controllability between steady states belonging to
    different connected components of $\mcS$.

    This paper is organized as follows. Section~\ref{Notations} introduces the notation and preliminaries used throughout the paper. Section~\ref{WPAR} collects the well-posedness and regularity results for system~\eqref{cntrl_eqn}. Section~\ref{LOGAC} proves an obstruction to approximate controllability that is independent of the terminal time and contains the proof of Theorem~\ref{thm1}. Section~\ref{SSATCC} describes the set of steady states and characterizes its connected components. Section~\ref{QSD} is devoted to the quasi-static deformation argument and contains the proof of Theorem~\ref{Main_thm}. Section~\ref{CR} collects further remarks and open problems. Finally, the Appendix contains the proofs of 
Lemmas~\ref{lem:weak_time_continuity}, \ref{lem:left_ivp}, 
and~\ref{lem:right_ivp}, as well as that of 
Theorem~\ref{Thm_Soln_Trans_sys}.
    
\section{Notations and preliminaries}\label{Notations}
    We denote by $\mathbb{N}$ the set of natural numbers. For $n\in \mathbb{N},$ denote the Euclidean inner product and the norm in $\mathbb{R}^{n}$ by $\langle\cdot,\cdot\rangle_2$ and $||\cdot||_2$, respectively. Let $\Gamma$ be a set and $\Theta:=\{(\epsilon, t):\ 0<\epsilon\leq 1,\ 0\leq t\leq 1/{\epsilon}\}.$ Let $F_1, F_2$ be two real valued functions defined on $\Theta\times \Gamma$. The notation $F_1\lesssim F_2$ means that $F_2\geq 0$ and there exists a constant $C>0$ such that
    \begin{equation*}
        \forall\ (\epsilon,t)\in \Theta,\ \forall\ \zeta \in\Gamma,\quad F_1(\epsilon,t,\zeta)\leq C F_2(\epsilon,t,\zeta)\ .
    \end{equation*}
    We say that $F_1\sim F_2$ if both $F_1\lesssim F_2$ and $F_2\lesssim F_1.$ 
    
    Let $\Omega=(0,1)$. For $1\leq p\leq\infty$, we denote by $L^p(\Omega)$ the Lebesgue space equipped with the norm 
    \begin{align*}
        \|u\|_{L^p(\Omega)}=
        \begin{cases}
             \displaystyle\left(\int_\Omega |u(x)|^p\,dx\right)^{1/p},& 1\leq p<\infty,\\[3mm]
             \displaystyle\operatorname*{ess\,sup}_{x\in\Omega}|u(x)|,& p=\infty.
        \end{cases}
    \end{align*}
    For an integer $k\geq 0$ and $1\leq p\leq\infty$, the Sobolev space $W^{k,p}(\Omega)$ is defined by
    \begin{equation*}
        W^{k,p}(\Omega):= \left\{u \in L^p(\Omega) \; : \;D^{\alpha} u \in L^p(\Omega)\ \text{for all multi-indices } {\alpha} \text{ with } |{\alpha}|\le k\right\},
    \end{equation*}
    equipped with the norm
        \begin{align*}
            \|u\|_{W^{k,p}(\Omega)}:=
            \begin{cases}
               \Bigg( \displaystyle\sum_{|\alpha|\leq k}\|D^\alpha u\|^p_{L^p(\Omega)}\Bigg)^{1/p},& 1\leq p<\infty,\\[3mm]
                \displaystyle\max_{|\alpha|\leq k}\|D^\alpha u\|_{L^\infty(\Omega)},& p=\infty.
            \end{cases}
        \end{align*}
    We denote $H^k(\Omega): = W^{k,2}(\Omega),$ which is a Hilbert space, and the norm is induced by the standard inner product
    \begin{equation*}
        \langle u,v\rangle_{H^k(\Omega)}:= \sum_{|\alpha|\le k} \int_\Omega D^\alpha u(x)\, D^\alpha v(x)\, dx.
    \end{equation*}
   
    Let $(X,\|\cdot\|_X)$ be a Banach space and $J\subseteq[0,\infty)$ be a time interval, either finite or infinite. For $1\leq p\leq\infty$, the Bochner space $L^p(J;X)$ is defined by
    \begin{align*}
        L^p(J;X):=
        \begin{cases}
            \displaystyle\left\{u:J\to X\text{ strongly measurable}:\int_J\|u(t)\|_X^p\,dt<\infty\right\},& 1\leq p<\infty,\\[3mm]
            \displaystyle\left\{u:J\to X\text{ strongly measurable}:\operatorname*{ess\,sup}_{t\in J}\|u(t)\|_X<\infty\right\},& p=\infty,
        \end{cases}
    \end{align*}
    equipped with the norm
    \begin{align*}
        \|u\|_{L^p(J;X)}:=
        \begin{cases}
            \displaystyle\left(\int_J\|u(t)\|_X^p\,dt\right)^{1/p},& 1\leq p<\infty,\\[3mm]
            \displaystyle\operatorname*{ess\,sup}_{t\in J}\|u(t)\|_X,& p=\infty.
        \end{cases}
    \end{align*}
    We use $W^{1,p}(J;X)$ for the usual Bochner-Sobolev space, and write
\begin{equation*}
   H^1(J;X):=W^{1,2}(J;X). 
\end{equation*}
    Let $C_0^\infty(\Omega)$ denote the space of infinitely differentiable functions with compact support in $\Omega$. We define $H_0^1(\Omega)$ as the closure of $C_0^\infty(\Omega)$ in $H^1(\Omega)$, and denote its dual space by $H^{-1}(\Omega)$. The duality pairing between $H^{-1}(\Omega)$ and $H_0^1(\Omega)$ is denoted by $\langle\cdot,\cdot\rangle$. In one dimension, the following continuous embeddings hold:
    \begin{equation*}
        H_0^1(\Omega)\hookrightarrow C(\overline{\Omega})\hookrightarrow L^\infty(\Omega)\hookrightarrow L^p(\Omega),\qquad 1\leq p<\infty.
    \end{equation*}

\section{Well-Posedness and Regularity}\label{WPAR}
    In this section, we study the well-posedness and regularity of solutions to the generalized Burgers-Huxley equation. We first introduce the notions of weak and strong solutions and state the corresponding existence, uniqueness, and regularity results. Finally, in the absence of control, that is, when $(f=0)$, we show that nonnegative initial data give rise to nonnegative solutions.

    As in~\cite{Mohan2021}, we adopt the following standard definitions of weak and strong solutions to~\eqref{cntrl_eqn}.

\begin{dfn}[Weak solution]
\label{Weak_soln}
    Let $u_0\in L^2(0,1)$ and $f\in L^2(0,T;L^2(0,1)).$ A function
    \begin{equation*}
         u\in L^\infty(0,T;L^2(0,1))\cap L^2(0,T;H_0^1(0,1))\cap L^{2(\delta+1)}(0,T;L^{2(\delta+1)}(0,1))
    \end{equation*}
    with $u_t\in L^{1+\frac{1}{2\delta-1}}(0,T;H^{-1}(0,1))$ is called a weak solution of \eqref{cntrl_eqn} provided $u$ satisfies
    \begin{align}
        \begin{cases}
            \langle u_t(t),\varphi\rangle+ \int_0^1 u_x(t,x)\varphi_x(x)\,dx+ \alpha\int_0^1 u^{\delta}(t,x)u_x(t,x)\varphi(x)\,dx\\
            \qquad \qquad = \beta\int_0^1 u(t,x)\bigl(1-u^{\delta}(t,x)\bigr)\bigl(u^{\delta}(t,x)-\gamma\bigr)\varphi(x)\,dx\\
            \qquad \qquad \qquad \, + \int_0^1 \chi_\omega(x)f(t,x)\varphi(x)\,dx, \\
            \int_0^1 u(0,x)\varphi(x)\,dx=\int_0^1 u_{0}(x)\varphi(x)\,dx
        \end{cases}
    \end{align}
    for all $\varphi\in H_0^1(0,1)$ and for a.e. $t\in(0,T)$.
\end{dfn}

\begin{dfn}[Strong solution] 
    Let $u_0\in H_0^1(0,1),f\in L^2(0,T;L^2(0,1)).$ A function
    \begin{equation*}
        u\in C([0,T];H_0^1(0,1))\cap L^2(0,T;H^2(0,1)) \text{ with } u_t\in L^2(0,T;L^2(0,1))
    \end{equation*}
    is called a strong solution of \eqref{cntrl_eqn} provided $u(0)=u_0$ in $H_0^1(0,1)$, and
    \begin{equation*}
        u_t=u_{xx}-\alpha u^{\delta} u_x+\beta u(1-u^{\delta})(u^{\delta}-\gamma)+\chi_\omega f
    \end{equation*}
    holds in $L^2(0,1)$ for a.e. $t\in(0,T)$.
\end{dfn}

    We recall the following well-posedness result.
\begin{thm}[\cite{Mohan2021}]
\label{thm3.4}
    Let $  u_0 \in L^2(0,1)  $ and $  f \in L^2(0,T; L^2(0,1))  $. Then the system \eqref{cntrl_eqn} admits at least one weak solution. Moreover, if $\beta  > 2^{2\delta}\alpha^{2}$, then the weak solution is unique. In particular, for $1\le \delta \le 2$, uniqueness holds for all positive constants $\alpha$ and $\beta$. If, in addition, $  u_0 \in H^1_0(0,1)  $ then the system \eqref{cntrl_eqn} admits a unique strong solution.
\end{thm}
    We next establish the time continuity of weak solutions.
\begin{lem}\label{lem:weak_time_continuity}
    Let $u$ be a weak solution of \eqref{cntrl_eqn} in the sense of Definition~\ref{Weak_soln}. Then 
    \begin{equation*}   
        u\in C([0,T];L^2(0,1)). 
    \end{equation*}
\end{lem}
The proof of Lemma~\ref{lem:weak_time_continuity} is given 
in the Appendix.
    
    In the absence of control (i.e., when $f \equiv 0$), the system \eqref{cntrl_eqn} reduces to
\begin{equation}
\label{without_cntrol_eqn}
    \begin{cases}
        u_t=u_{xx} -\alpha u^{\delta} u_x+\beta u(1-u^{\delta})(u^{\delta} -\gamma) ,\;\;\;& {\text{in}} \;\; (0,T)\times(0,1),\\
        u(t,0)=u(t,1)=0,\;\;\;\;& t\in(0,T),\\
        u(0,x)=u_0(x)\;\; &x\in (0,1).\\
    \end{cases}
\end{equation}
    The following proposition addresses the non-negativity of the solution:
\begin{prop} 
    Let $u_0 \in L^2(0,1)$ with $u_0 \geq 0$. Then the solution $u$ of \eqref{without_cntrol_eqn} with initial data $u_0$ satisfies $u \geq 0$ in $(0,T) \times (0,1)$.
\end{prop}
\begin{proof}
    The existence and uniqueness of the solution to \eqref{without_cntrol_eqn} follow from Theorem~\ref{thm3.4}. Now for the non-negativity, multiplying equation \eqref{without_cntrol_eqn} by $-u^-:= -\max\{-u,0\}$ and performing integration by parts, we obtain for each  $0\leq t \leq T$
    \begin{equation}
    \label{first_21}
        \begin{split}
            \frac{1}{2}\frac{d}{dt}\|u^-(t,\cdot)\|_{L^2(0,1)}^2+\|u_x^-(t,\cdot)\|_{L^2(0,1)}^2&=\int_0^1\alpha(-u^-(t,x))^{\delta+1}(u^-(t,x))_x\;dx \\
            &- \int_0^1 \beta (u^-(t,x))^2(1-u^{\delta}(t,x))(\gamma-u^{\delta}(t,x))\;dx\\
            &:=I_1+I_2
        \end{split}
    \end{equation}
    First, we show that $I_1=0$. For that, we have 
    \begin{equation}
    \label{I_1}
        \begin{split}
            I_1&= \int_0^1\alpha(-u^-(t,x))^{\delta+1}(u^-(t,x))_x\;dx\\
            &= -\frac{\alpha}{\delta+2}\int_0^1 (-u^-(t,x))^{\delta+2}_x\;dx.   
        \end{split}
    \end{equation}
    Performing integration by parts with the help of $u(t,0)=u(t,1)=0$ we obtain $I_1=0$. Now, we find an estimate for $I_2$. Using Cauchy–Schwarz and Young's inequality, we obtain
    \begin{equation}
    \label{estimate_I_2}
        \begin{split}
            I_2&= -\int_0^1 \beta (u^-(t,x))^2(1-u^{\delta}(t,x))(\gamma-u^{\delta}(t,x))\;dx\\
            &= -\int_0^1 \beta (u^-(t,x))^2 \left(\gamma-(\gamma+1) (-u^-(t,x))^\delta)+ (-u^-(t,x))^{2\delta})\right)\;dx\\
            &= -\beta\gamma\int_0^1|u^-(t,x)|^2\;dx +\beta (\gamma+1) \int_0^1 (-u^-(t,x))^{\delta+2}\;dx-\beta\int_0^1 |u^-(t,x)|^{2\delta+2}\;dx\\
            &\leq -\beta\gamma\int_0^1|u^-(t,x)|^2\;dx + \frac{\beta (\gamma+1)^2}{2} \int_0^1 |u^-(t,x)|^{2}\;dx+\frac{\beta}{2}\int_0^1 |u^-(t,x)|^{2\delta+2}\;dx\\
            &\hspace{9cm}-\beta\int_0^1 |u^-(t,x)|^{2\delta+2}\;dx\\
            &\leq\frac{\beta(\gamma^2+1)}{2}\|u^-(t,\cdot)\|^2_{L^2(0,1)},
        \end{split}
    \end{equation}
    for all $0\leq t\leq T$. Substituting this estimate \eqref{estimate_I_2} in \eqref{first_21}, we have 
    \begin{equation*}
        \frac{1}{2}\frac{d}{dt}\|u^-(t,\cdot)\|_{L^2(0,1)}^2\leq \frac{\beta(\gamma^2+1)}{2}\|u^-(t,\cdot)\|^2,
    \end{equation*}
    for all $0\leq t\leq T$. Finally, using Gr\"onwall's inequality, we get 
    \begin{equation}
    \label{last_21}
        \|u^-(t,\cdot)\|_{L^2(0,1)}^2 \leq e^{\beta(\gamma^2+1)t }\|u^-_0(\cdot)\|_{L^2(0,1)}^2,
    \end{equation}
    for all $0\leq t\leq T$. Since $u_0 \geq 0$, it follows that $u^- = 0$. Therefore, from the estimate in \eqref{last_21}, we conclude that $u^-(t,x) = 0$ for almost every $(t,x) \in (0,T) \times (0,1)$. Hence, $u \geq 0$ in $(0,T) \times (0,1)$.
\end{proof}

\section{On the Lack of Global Approximate Controllability}
\label{LOGAC}

    This section is devoted to the proof of Theorem~\ref{thm1}.
    We begin by defining $g:[0,b]\subset[0,1]\rightarrow\R$ by
    \begin{equation*}
        g(x):=x^N(b-x)^N,\qquad x\in[0,b],\qquad N>4.
    \end{equation*}
    We also set
    \begin{equation*}
        J=N(N-1)(b-x)^2-2N^2x(b-x)+N(N-1)x^2.
    \end{equation*}
    We establish the following estimate without imposing
    a sign condition on $u$.

\begin{lem}
    Let $T>0$, and let $u$ be a weak solution of
    \eqref{cntrl_eqn}. Then the function
    \begin{equation*}
        t\longmapsto\int_0^b g(x)u^2(t,x)\,dx
    \end{equation*}
    is absolutely continuous on $[0,T]$ and satisfies,
    for almost every $t\in(0,T)$,
    \begin{equation}
    \label{estimt1}
        \frac12\frac{d}{dt}\int_0^b g(x)u^2(t,x)\,dx
        +\frac{\beta}{2}\int_0^b
        g(x)|u(t,x)|^{2\delta+2}\,dx
        \leq L(b,N),
    \end{equation}
    where the finite positive constant $L(b,N)$ is independent of $u$, $f$, and $T$ and is given by
    {\small
    \begin{align*}
        L(b,N):=\int_0^b\Bigg[
        &\frac{\delta}
        {(1+\delta)\left(
        2^{\frac{\delta-2}{\delta}}
        \beta^{\frac1\delta}
        (1+\delta)^{\frac1\delta}\right)}
        \big(x(b-x)\big)^{\frac{\delta(N-2)-2}{\delta}}
        |J|^{\frac{1+\delta}{\delta}}
        \\
        &+\frac{\delta}{2\delta+2}
        \frac{\left(\frac{N\alpha}{\delta+2}\right)
        ^{\frac{2\delta+2}{\delta}}}
        {\left(\frac\beta8\right)^{\frac{\delta+2}{\delta}}
        \left(\frac{2\delta+2}{\delta+2}\right)
        ^{\frac{\delta+2}{\delta}}}
        |b-2x|^{\frac{2\delta+2}{\delta}}
        \big(x(b-x)\big)^{\frac{N\delta-2\delta-2}{\delta}}
        \\
        &+\frac{\delta}{2\delta+2}
        \frac{\big(\beta(1+\gamma)\big)
        ^{\frac{2\delta+2}{\delta}}}
        {\left(\frac\beta4\right)^{\frac{\delta+2}{\delta}}
        \left(\frac{2\delta+2}{\delta+2}\right)
        ^{\frac{\delta+2}{\delta}}}
        \big(x(b-x)\big)^N
        \Bigg]\,dx.
    \end{align*}}
\end{lem}

\begin{proof}
    We first carry out the calculation for sufficiently
    regular solutions. For weak solutions, testing with the zero extension of $gu$ is justified by time regularization in the weak formulation. Indeed, the regularity
    \begin{equation*}
        u\in L^2(0,T;H_0^1(0,1))
        \cap L^{2(\delta+1)}((0,T)\times(0,1))
    \end{equation*}
    ensures the integrability of the terms below. The time derivative is the sum of an element of $L^2(0,T;H^{-1}(0,1))$ and an element of
    $L^{\frac{2\delta+2}{2\delta+1}}((0,T)\times(0,1))$.
    The corresponding weighted chain rule gives the asserted absolute continuity and the integrated energy identity. Thus the differential inequalities below hold for almost every $t\in(0,T)$.

    We use the notation
    \begin{equation}
    \label{value_Eg}
        E_g(t,x):=\frac12\frac{d}{dt}
        \big\{g(x)u^2(t,x)\big\},
        \qquad (t,x)\in(0,T)\times[0,b].
    \end{equation}
    Using \eqref{cntrl_eqn}, we obtain
    \begin{equation}
    \label{E_g}
        \begin{split}
            E_g&=guu_t\\
            &=g\Big\{
            uu_{xx}-\alpha u^{\delta+1}u_x
            +\beta u^2
            \big[u^\delta-\gamma-u^{2\delta}
            +\gamma u^\delta\big]
            +\chi_\omega fu
            \Big\}\\
            &=g\Big\{
            (uu_x)_x-u_x^2
            -\frac{\alpha}{\delta+2}(u^{\delta+2})_x
            +\beta(1+\gamma)u^{\delta+2}\\
            &\qquad
            -\beta|u|^{2\delta+2}-\beta\gamma u^2
            \Big\}+\chi_\omega guf\\
            &\leq g(uu_x)_x
            -\frac{\alpha}{\delta+2}g(u^{\delta+2})_x
            +\beta(1+\gamma)gu^{\delta+2}\\
            &\qquad
            -\beta g|u|^{2\delta+2}+\chi_\omega guf\\
            &=(guu_x)_x
            -\left(\frac{g'u^2}{2}\right)_x
            +\frac{g''u^2}{2}
            -\frac{\alpha}{\delta+2}(gu^{\delta+2})_x\\
            &\qquad
            +\frac{\alpha}{\delta+2}g'u^{\delta+2}
            +\beta(1+\gamma)gu^{\delta+2}
            -\beta g|u|^{2\delta+2}+\chi_\omega guf.
        \end{split}
    \end{equation}
    Integrating over $[0,b]$, using $\chi_\omega=0$ and
    $g(0)=g(b)=g'(0)=g'(b)=0$, we find that
    \begin{equation*}
        I_g(t):=\frac12\frac{d}{dt}
        \int_0^b g(x)u^2(t,x)\,dx
    \end{equation*}
    satisfies
    \begin{equation}
    \label{I_g}
        \begin{split}
            I_g(t)\leq\int_0^b\Bigg(
            \frac{g''u^2}{2}
            +\frac{\alpha}{\delta+2}g'u^{\delta+2}
            +\beta(1+\gamma)gu^{\delta+2}-\beta g|u|^{2\delta+2}
            \Bigg)\,dx.
        \end{split}
    \end{equation}
    For $(t,x)\in(0,T)\times(0,b)$, set
    \begin{align*}
        L_1&=\frac{g''u^2}{2}
        =\frac12\big(x(b-x)\big)^{N-2}Ju^2,\\
        L_2&=\frac{\alpha}{\delta+2}g'u^{\delta+2}
        =\frac{\alpha N}{\delta+2}
        \big(x(b-x)\big)^{N-1}(b-2x)u^{\delta+2},\\
        L_3&=\beta(1+\gamma)gu^{\delta+2}
        =\beta(1+\gamma)
        \big(x(b-x)\big)^Nu^{\delta+2}.
    \end{align*}
    Applying Young's inequality with conjugate exponents
    $\delta+1$ and $(\delta+1)/\delta$ to $L_1$, and with
    conjugate exponents $(2\delta+2)/(\delta+2)$ and
    $(2\delta+2)/\delta$ to $L_2$ and $L_3$, gives
    {\small
    \begin{equation}
    \label{estimate_Li}
        \begin{aligned}
            |L_1|
            &\leq\frac\beta8
            \big(x(b-x)\big)^N|u|^{2\delta+2}\\
            &\quad+
            \frac{\delta}
            {(1+\delta)\left(
            2^{\frac{\delta-2}{\delta}}
            \beta^{\frac1\delta}
            (1+\delta)^{\frac1\delta}\right)}
            \big(x(b-x)\big)
            ^{\frac{\delta(N-2)-2}{\delta}}
            |J|^{\frac{1+\delta}{\delta}},
            \\[0.5em]
            |L_2|
            &\leq\frac\beta8
            \big(x(b-x)\big)^N|u|^{2\delta+2}\\
            &\quad+
            \frac{\delta}{2\delta+2}
            \frac{\left(\frac{N\alpha}{\delta+2}\right)
            ^{\frac{2\delta+2}{\delta}}}
            {\left(\frac\beta8\right)
            ^{\frac{\delta+2}{\delta}}
            \left(\frac{2\delta+2}{\delta+2}\right)
            ^{\frac{\delta+2}{\delta}}}
            |b-2x|^{\frac{2\delta+2}{\delta}}
            \big(x(b-x)\big)
            ^{\frac{N\delta-2\delta-2}{\delta}},
            \\[0.5em]
            |L_3|
            &\leq\frac\beta4
            \big(x(b-x)\big)^N|u|^{2\delta+2}\\
            &\quad+
            \frac{\delta}{2\delta+2}
            \frac{\big(\beta(1+\gamma)\big)
            ^{\frac{2\delta+2}{\delta}}}
            {\left(\frac\beta4\right)
            ^{\frac{\delta+2}{\delta}}
            \left(\frac{2\delta+2}{\delta+2}\right)
            ^{\frac{\delta+2}{\delta}}}
            \big(x(b-x)\big)^N.
        \end{aligned}
    \end{equation}}
    Substituting \eqref{estimate_Li} into \eqref{I_g},
    we retain one half of the negative reaction term
    and obtain
    \begin{equation}
        I_g(t)+\frac{\beta}{2}
        \int_0^b g(x)|u(t,x)|^{2\delta+2}\,dx
        \leq L(b,N).
    \end{equation}
    Since $N>4$ and $\delta\geq1$,
    \begin{equation*}
        \frac{\delta(N-2)-2}{\delta}
        =\frac{N\delta-2\delta-2}{\delta}
        =N-2-\frac2\delta>0.
    \end{equation*}
    Moreover, $|J|$ and $|b-2x|$ are bounded on $[0,b]$.
    Thus every term defining $L(b,N)$ is integrable, and
    $0<L(b,N)<\infty$. This proves \eqref{estimt1}.
\end{proof}

    Using this lemma, we now prove Theorem~\ref{thm1}.

\begin{proof}[Proof of Theorem~\ref{thm1}]
    Fix $u_0\in L^2(0,1)$, and set
    \begin{equation*}
        \begin{aligned}
            G&:=\int_0^b g(x)\,dx,\qquad
            R:=\left(
            \frac{2L(b,N)G^\delta}{\beta}
            \right)^{1/(\delta+1)},
            M:=\max\left\{
            \int_0^b g(x)|u_0(x)|^2\,dx,R
            \right\}.
        \end{aligned}
    \end{equation*}
    Let $T>0$ and $f\in L^2((0,T)\times(0,1))$ be arbitrary, and let $u$ be the corresponding solution of \eqref{cntrl_eqn}. For any target $u_1\in L^2(0,1)$, the reverse triangle inequality gives
    \begin{equation}
    \label{Y(T)}
        \begin{split}
            \|u_1-u(T)\|_{L^2(0,1)}
            &=\left(
            \int_0^1|u_1(x)-u(T,x)|^2\,dx
            \right)^{1/2}\\
            &\geq\left(
            \int_{b/4}^{b/2}|u_1(x)-u(T,x)|^2\,dx
            \right)^{1/2}\\
            &\geq
            \|u_1\|_{L^2\left(\frac b4,\frac b2\right)}
            -\|u(T,\cdot)\|_{L^2\left(\frac b4,\frac b2\right)}.
        \end{split}
    \end{equation}
    Define
    \begin{equation*}
        E(t):=\int_0^b g(x)|u(t,x)|^2\,dx,
        \qquad 0\leq t\leq T.
    \end{equation*}
    By H\"older's inequality,
    \begin{equation*}
        \begin{split}
            E(t)
            &=\int_0^b
            \big(g(x)^{1/(\delta+1)}|u(t,x)|^2\big)
            g(x)^{\delta/(\delta+1)}\,dx\\
            &\leq
            \left(
            \int_0^b g(x)|u(t,x)|^{2\delta+2}\,dx
            \right)^{1/(\delta+1)}
            G^{\delta/(\delta+1)}
        \end{split}
    \end{equation*}
    for almost every $t\in(0,T)$. Hence \eqref{estimt1}
    yields
    \begin{equation*}
        E'(t)\leq
        2L(b,N)-\beta G^{-\delta}E(t)^{\delta+1}.
    \end{equation*}
    Since $M\geq E(0)$ and
    $2L(b,N)-\beta G^{-\delta}M^{\delta+1}\leq0$,
    scalar comparison gives
    \begin{equation*}
        E(t)\leq M,\qquad 0\leq t\leq T.
    \end{equation*}
    The estimate holds at every time by the continuity
    of $E$. Since
    \begin{equation*}
        x(b-x)\geq\frac{3b^2}{16}\geq\frac{b^2}{8},
        \qquad x\in[b/4,b/2],
    \end{equation*}
    we obtain
    \begin{equation}
    \label{estimate_u(T)}
        \begin{split}
            \|u(T,\cdot)\|_{L^2\left(\frac b4,\frac b2\right)}^2
            &=\int_{b/4}^{b/2}
            \big(x(b-x)\big)^{-N}
            \big(x(b-x)\big)^N u^2(T,x)\,dx\\
            &\leq
            \left(\frac{b^2}{8}\right)^{-N}
            \int_0^b g(x)|u(T,x)|^2\,dx\\
            &\leq
            \left(\frac{b^2}{8}\right)^{-N}M.
        \end{split}
    \end{equation}
    The right-hand side is independent of both $f$ and $T$.
    Let $\eta_{0}>0$, and choose
    $u_1\in C_c^\infty((b/4,b/2))$ such that
    \begin{equation}
    \label{assumption_uT}
        \|u_1\|_{L^2\left(\frac b4,\frac b2\right)}
        >
        \left(\frac{b^2}{8}\right)^{-N/2}M^{1/2}
        +\eta_{0}.
    \end{equation}
    Such a target exists: take a sufficiently large multiple
    of a nonzero smooth function compactly supported in
    $(b/4,b/2)$. Combining \eqref{Y(T)},
    \eqref{estimate_u(T)} and \eqref{assumption_uT},
    we obtain
    \begin{equation*}
        \|u_1-u(T)\|_{L^2(0,1)}>\eta_{0}.
    \end{equation*}
    Since $T$ and $f$ were arbitrary and $u_1$ was chosen
    independently of both, the conclusion follows.
\end{proof}

\begin{remark}
    We conclude this section with the following observations.
    \begin{itemize}
        \item[(i)] For each fixed $u_0\in L^2(0,1)$,
        any target satisfying \eqref{assumption_uT} remains
        at distance greater than $\eta_{0}$ from every
        reachable state, independently of the time allowed.
        Thus, increasing the terminal time does not remove
        this obstruction.

        \item[(ii)] For any steady state $u\in\mcS$,
        applying \eqref{estimt1} and using H\"older's
        inequality yields
        \begin{equation*}
            \int_0^b g(x)|u(x)|^2\,dx\le R,
        \end{equation*}
        where $R$ is defined in the proof of
        Theorem~\ref{thm1}. Consequently,
        \begin{equation*}
            \|u\|_{L^2(\frac b4,\frac b2)}
            \le \left(\frac{b^2}{8}\right)^{-N/2}R^{1/2}
            \le \left(\frac{b^2}{8}\right)^{-N/2}M^{1/2}.
        \end{equation*}
        Thus, no steady state satisfies
        \eqref{assumption_uT}, and hence the obstruction
        is consistent with Theorem~\ref{Main_thm}.
    \end{itemize}
\end{remark}

    Although approximate controllability between arbitrary states fails even when the terminal time is allowed to vary, a positive result can still hold for restricted classes of initial and target states.

\begin{remark}
\label{PAC}
    Theorem~\ref{thm1} gives a negative answer to approximate controllability between arbitrary states when the terminal time may depend on the initial state, the target and the prescribed accuracy. In particular, taking $u_0=0$, we obtain smooth target states that cannot be approximated at any time. Nevertheless, we obtain a positive result when $u_0$ and $u_1$ are steady states lying in the same connected component of $\mcS$, endowed with the
    $C^2$-topology; see Section~\ref{QSD}.
\end{remark}
\section{Steady States and Their Connected Components}\label{SSATCC}
\subsection{Characterization of steady states}
    The following lemmas provide families of nontrivial solutions of the ordinary differential equation in \eqref{steady_state} on the intervals $[0,b]$ and $[c,1]$, satisfying the conditions $u(0)=0$ and $u(1)=0$, respectively. By combining a solution on $[0,b]$ with a solution on $[c,1]$, one obtains nontrivial functions $u\in C^2\bigl([0,b]\cup[c,1]\bigr)$ that satisfy \eqref{steady_state} on $(0,b)\cup(c,1)$.

\begin{lem}
\label{lem:left_ivp}
    There exists $s_0>0$ such that, for each
        $s\in(-s_0,s_0)$, the initial value problem
    \begin{equation}
    \label{ivp_ss}
        \begin{cases}
            u_{xx}-\alpha u^{\delta} u_x+\beta u(1-u^{\delta})(u^{\delta}-\gamma)=0,\\[1mm]
            u(0)=0,\\[1mm]
            u_x(0)=s,
        \end{cases}
    \end{equation}
    admits a unique solution $u\in C^2([0,b])$. Moreover, if $s\neq 0$, then $u$ is nontrivial.
\end{lem}

\begin{lem}
\label{lem:right_ivp}
    There exists $\widehat{s}_0>0$ such that, for every $\widehat{s}\in(-\widehat{s}_0, \widehat{s}_0)$, the initial value problem
    \begin{equation}
    \label{ivp_right}
        \begin{cases}
            u_{xx}-\alpha u^{\delta} u_x+\beta u(1-u^{\delta})(u^{\delta}-\gamma)=0,\\[1mm]
            u(1)=0,\\[1mm]
            u_x(1)=\widehat{s},
        \end{cases}
    \end{equation}
    admits a unique solution  $u\in C^2([c,1])$. Moreover, if $\widehat{s}\neq0$, then $u$ is nontrivial.
\end{lem}
    The proofs of Lemmas~\ref{lem:left_ivp} and~\ref{lem:right_ivp} are given in the Appendix.

    To study the approximate controllability properties of \eqref{cntrl_eqn}, it is essential to characterize the set $\mathcal{S}$ of steady states. To this end, we seek functions $u\in C^2([0,1])$ satisfying \eqref{steady_state}. To be more specific, we first seek solutions $u \in C^2([0,b]\cup[c,1])$ of \eqref{steady_state} and subsequently extend them to functions in $C^2([0,1])$. A convenient way to construct a $C^2$-extension of $u$ on $[b,c]$ is through a quintic polynomial. Its coefficients are uniquely determined by the values of $u$, $u'$, and $u''$ at $b$ and $c$. However, this extension is not unique; there exist infinitely many $C^2$-functions on $[b,c]$ satisfying the same matching conditions. Moreover, if $u_1,u_2\in\mathcal{S}$ satisfy $u_1(x)=u_2(x)$ for $x\in [0,b]\cup[c,1],$ then $u_1$ and $u_2$ belong to the same connected component of $\mathcal{S}$. Indeed, since the two steady states coincide outside $[b,c]$, one can continuously deform $u_1$ into $u_2$ by varying only their values on $[b,c]$ while keeping the endpoint matching conditions at $b$ and $c$ unchanged.

    The following lemma characterizes the connected components of $\mcS$.
\begin{lem}[Characterization of connected components of $\mathcal S$]
\label{lem:connected_components_S}
    Let $u_0,u_1\in\mathcal S$. Then $u_0$ and $u_1$ belong to the same connected component of $\mathcal S$ if and only if, for every
    \begin{equation*}
        (s,\widehat{s})\in\Bigg\{\tau\cmat{u_0'(0)}{u_0'(1)}+(1-\tau)\cmat{u_1'(0)}{u_1'(1)}:\tau\in[0,1]\Bigg\},
    \end{equation*}
    the maximal solution of
    \begin{align*}
        \begin{cases}
            u_{xx}-\alpha u^{\delta} u_x+\beta u(1-u^{\delta})(u^{\delta}-\gamma)=0,\\
            u(0)=0,\qquad u_x(0)=s,
        \end{cases}
    \end{align*}
    is defined on $[0,b]$, and the maximal solution of
    \begin{align*}
        \begin{cases}
            u_{xx}-\alpha u^{\delta} u_x+\beta u(1-u^{\delta})(u^{\delta}-\gamma)=0,\\
            u(1)=0,\qquad u_x(1)=\widehat s,
        \end{cases}
    \end{align*}
    is defined on $[c,1]$.
\end{lem}
    The proof of Lemma~\ref{lem:connected_components_S} is similar to that of Lemma~3.1 in \cite{Shirshendu20} and is therefore omitted.
\section{Approximate Controllability Between Steady States via Quasi-Static Deformation}
\label{QSD}
   In what follows, we restrict our attention to the case $\delta\ge2$. The case $\delta=1$ follows by the same argument, with the sums over $j=2,\ldots,\delta$ being empty and the corresponding higher-order terms absent.
\subsection{Reduction of the problem}
    Let $\epsilon>0,\ \tau\in [0,1].$ Let $\tau\mapsto \overline{u}(\tau,\cdot)$ be a $C^1$ path in $\mcS,$ connecting $u_0$ and $u_1.$ Let us define $\tau\mapsto \overline f(\tau,\cdot)$ by 
    \begin{equation*}
        \overline f(\tau, x):= \alpha \overline{u}^\delta(\tau, x) \overline{u}_x(\tau, x)-\beta \overline{u}(\tau, x)(1-\overline{u}^\delta(\tau, x))(\overline{u}^\delta(\tau, x) -\gamma)-\overline{u}_{xx}(\tau, x),
    \end{equation*}
    for $\tau\in[0,1]$ and $x\in[0,1].$ Setting $\tau=\epsilon t,$ we have $\tau\in [0,1]$ if and only if $t\in[0,1/{\epsilon}].$ For $t\in[0,1/{\epsilon}]$ and $x\in[0,1],$ let us define
    \begin{equation}
    \label{Intro_z_h}
        \begin{cases}
            z(t, x) : =  u(t, x)- \overline{u}(\epsilon t, x),\\
            h(t, x) : =  f(t, x)- \overline{f}(\epsilon t, x).     
        \end{cases}
    \end{equation}
    Thanks to the time reparametrization introduced above, we can perform a pole-shifting procedure on the finite-dimensional linear system corresponding to the unstable part of the equation. A suitable feedback control is then constructed to stabilize the full system.

    For simplicity, we define $F:\mathbb{R}\to\mathbb{R}$ by
    \begin{equation}
    \label{Defn_of_F}
        F(v):= \beta v(1-v^\delta)(v^\delta-\gamma),\quad v\in \mathbb{R}.
    \end{equation}
    A straightforward computation shows that
    \begin{equation}
    \label{Defn_F_prime}
        F^{\prime}(v)=\beta\big[(1-v^{\delta}) (v^{\delta} -\gamma)+\delta v^{\delta}(1-2v^{\delta} +\gamma)\big],\quad v\in \mathbb{R}.
    \end{equation}
    Taylor's formula with integral remainder around $\overline{v} \in \mathbb{R}$ yields, for every $\overline{h} \in \mathbb{R}$,
    \begin{equation}
    \label{Taylor_exp_F}
        F(\overline{v}+\overline{h}) - F(\overline{v}) - F'(\overline{v})\overline{h}= \overline{h}^{2}\int_{0}^{1}(1-s)\,F''(\overline{v}+s\overline{h})\,ds.
    \end{equation}
    Using \eqref{Taylor_exp_F}, we obtain the estimate
    \begin{equation}
    \label{Taylor_remainder_bound}
        \big|F(\overline{v}+\overline{h}) - F(\overline{v}) - F'(\overline{v})\overline{h}\big| 
         \;\lesssim\; \big(1+|\overline{v}|^{2\delta-1}\big)|\overline{h}|^{2} + |\overline{h}|^{2\delta+1},
    \end{equation}
    where the implicit constants depend only on $\delta$ and on the coefficients of $F$.

    Using \eqref{cntrl_eqn}, \eqref{steady_state}, and \eqref{Intro_z_h}, we obtain that, $z$ satisfies the following equation in $(0,T)\times (0,1)$ with $T=1/\epsilon,$
    \begin{align}
    \label{z_eqn}
        \begin{cases}
            z_t =z_{xx}-\alpha \overline{u}^{\delta} z_x-\alpha\delta \overline{u}^{\delta-1}\overline{u}_x z +F^{\prime}(\overline{u})z-\alpha\big(\delta \overline{u}^{\delta-1} z+ \sum_{j=2}^{\delta}\binom{\delta}{j} z^{j}\overline{u}^{\delta-j}\big) z_x \\
            \qquad\qquad-\alpha\big(\sum_{j=2}^{\delta}\binom{\delta}{j} z^{j}\overline{u}^{\delta-j}\big)\overline{u}_x + \big( F(z+\overline{u})-F(\overline{u})- F^{\prime}(\overline{u})z\big) \\
            \qquad\qquad +\chi_{\omega} h- \epsilon \overline{u}_{\tau},\;\;\; {\text{in}} \;\; (0,T)\times(0,1),\\
            z(t,0)=z(t,1)=0,\;\;\;\; t\in(0,T),\\
            z(0,x)=0,\;\;\;\; x\in (0,1).
        \end{cases}
    \end{align}
    We introduce a one-parameter family of unbounded linear operators for $0\leq \tau \leq 1$ $\big(\mcA(\tau), \mcD(\mcA(\tau))\big)$ in $L^2(0,1)$ defined by
    \begin{equation}
    \label{Defn_A_tau}
        \mcD(\mcA(\tau)):=H^1_0 (0,1)\cap H^2(0,1),\quad \mcA(\tau)(z):=  z_{xx}+p(\tau, \cdot) z_x+q(\tau, \cdot)z,
    \end{equation}
    with 
    \begin{equation*}
        p(\tau,\cdot):=-\alpha \overline{u}^{\delta}(\tau, \cdot),\quad q(\tau,\cdot):= F^{\prime}(\overline{u}(\tau, \cdot))-\alpha\delta \overline{u}^{\delta-1}(\tau, \cdot)\overline{u}_x(\tau, \cdot) ,
    \end{equation*}
    and
    \begin{equation*}
        F^{\prime}(\overline{u}(\tau, \cdot))= \beta\big[\big(1-\overline{u}^{\delta}(\tau, \cdot)\big) \big(\overline{u}^{\delta}(\tau, \cdot) -\gamma\big)+\delta \overline{u}^{\delta}(\tau, \cdot)\big(1-2\overline{u}^{\delta}(\tau, \cdot) +\gamma\big)\big].
    \end{equation*}
\begin{remark}
\label{rem:p_q_u_regularity}
    Since $\overline{u}\in C^1\big([0,1];C^2([0,1])\big)$ and $F$ is a polynomial, we have
    \begin{equation*}
        p\in C^1\big([0,1];C^2([0,1])\big),\qquad q\in C^1\big([0,1];C^1([0,1])\big).
    \end{equation*}
    In particular, $\overline{u},\overline{u}_x,\overline{u}_\tau,p,q,p_\tau$, and $q_\tau$ are uniformly bounded on $[0,1]\times[0,1]$.
\end{remark}
    
    Note that the domains $\mcD(\mcA(\tau))$ of $\mcA(\tau)$ are the same for all $\tau \in[0,1].$ Henceforth we denote $\mcD(\mcA(\tau))$ by $\mcD$ for all $\tau\in[0,1].$ Let $\mathcal{B}:L^2(0,1)\rightarrow L^2(0,1)$ be defined by $\mathcal{B}(h):=\chi_{\omega}h$. Then \eqref{z_eqn} can be rewritten as
    \begin{align}
    \label{z_eqn_in_operator_form}
        \begin{cases}
            z_t(t, \cdot)&=\mcA(\epsilon t)z(t, \cdot)+\mathcal{B}(h)(t)+\mcR(\epsilon, t, \cdot), \quad t\in [0,1/\epsilon]\\
            z(0)&=0,
        \end{cases}
    \end{align}
    where
    \begin{align}
    \label{Defn_R}
        \mcR(\epsilon, t, \cdot):=& -\alpha\bigg(\delta \overline{u}^{\delta-1} z+ \sum_{j=2}^{\delta}\binom{\delta}{j} z^{j}\overline{u}^{\delta-j}\bigg) z_x
        -\alpha\bigg(\sum_{j=2}^{\delta}\binom{\delta}{j} z^{j}\overline{u}^{\delta-j}\bigg)\overline{u}_x \\  
        & \quad +\big( F(z+\overline{u})-F(\overline{u})- F^{\prime}(\overline{u})z\big) -\epsilon \overline{u}_{\tau}(\epsilon t). \notag
    \end{align}
 
    To establish Theorem~\ref{Main_thm}, we construct a suitable control function and a Lyapunov functional that stabilizes the system \eqref{z_eqn}. It should be emphasized that the stabilization here is not in the usual sense (see~\cite{coron2007}, Chapter 7, Section 7.1), since $t\in [0,1/\epsilon]$ rather than $t\in [0,\infty)$. 

\subsection{Spectral analysis of the operator \texorpdfstring{$\mcA(\tau)$}{A(tau)}}
\label{Spectral_Analysis_A}
    The proof of Theorem~\ref{Main_thm} relies on the spectral analysis of the operator defined in \eqref{Defn_A_tau}. For $\tau\in[0,1]$, we study the eigenvalues of the operator $\mcA(\tau)$. To this end, we rewrite $\mcA(\tau)$ in Sturm-Liouville form as   
    \begin{equation*}
        \mathcal{A}(\tau)(z) = \frac{1}{\rho(\tau,\cdot)}\big((\rho(\tau,\cdot) z_x)_x + \rho(\tau,\cdot) q(\tau,\cdot) z\big),
    \end{equation*}
    where
    \begin{equation*}
        \rho(\tau,x):=\exp\!\left(\int_0^x p(\tau,s)\,ds\right), \quad x\in[0,1].
    \end{equation*}
    For each $\tau \in [0,1]$, define the Hilbert space $\mcH[\tau] := L^2\bigl( (0,1),\ \rho(\tau,x)\, dx \bigr),$ equipped with the inner product and the corresponding norm
    \begin{equation*}
        \langle z, w \rangle_{\mcH[\tau]} := \int_0^1 z(x) {w(x)}\, \rho(\tau,x)\, dx,\quad \|z\|_{\mcH[\tau]} := \left( \int_0^1 |z(x)|^2 \rho(\tau,x)\, dx \right)^{1/2}.
    \end{equation*}
    We refer to $\mcH[\tau]$ as the weighted $L^2$-space associated with the weight $\rho(\tau,\cdot)$. Similarly, we define the corresponding weighted Sobolev spaces
    \begin{equation*}
        \mcH^1_0[\tau] := H^1_0\bigl( (0,1),\ \rho(\tau,x)\, dx \bigr),\quad \mcH^2[\tau] := H^2\bigl( (0,1),\ \rho(\tau,x)\, dx \bigr).
    \end{equation*}
    endowed with the weighted norms
    \begin{equation*}
        \|z\|_{\mcH^1_0[\tau]} := \left(\int_0^1 \bigl(|z(x)|^2 + |z_x(x)|^2\bigr)\rho(\tau,x)\,dx\right)^{1/2},
    \end{equation*}
    and
    \begin{equation*}
        \|z\|_{\mcH^2[\tau]} := \left(\int_0^1 \bigl(|z(x)|^2 + |z_x(x)|^2 + |z_{xx}(x)|^2\bigr)\rho(\tau,x)\,dx\right)^{1/2}.
    \end{equation*}
    As $p\in C^1\big([0,1]; C^2[0,1]\big)$, there exist constants $0 < m_1 \leq m_2 < \infty$, such that
    \begin{equation*}
        m_1 \leq \rho(\tau,x) \leq m_2 \quad \forall\, \tau\in[0,1],\ \forall\, x\in[0,1].
    \end{equation*}
    Consequently, the weighted norms are uniformly equivalent to the standard norms, i.e.,
    \begin{equation}
    \label{Norm_Equivalence}
        \|\cdot\|_{\mcH[\tau]} \sim \|\cdot\|_{L^2(0,1)}, \quad\|\cdot\|_{\mcH^1_0[\tau]} \sim \|\cdot\|_{H^1_0(0,1)}, \quad \|\cdot\|_{\mcH^2[\tau]} \sim \|\cdot\|_{H^2(0,1)},
    \end{equation}
    uniformly for all $\tau \in [0,1]$.

    We consider the following eigenvalue problem
    \begin{align}
    \label{Eigenvalue_Problem}
        \begin{cases}
            \big(\rho(\tau,x)\phi^{\prime}(x)\big)^{\prime}+\rho(\tau,x)q(\tau, x)\phi(x)=\lambda(\tau)\rho(\tau,x) \phi(x),\quad \text{in}\quad (0,1),\\
            \phi(0)=\phi(1)=0.
        \end{cases}
    \end{align}
    The following result is a consequence of the Sturm-Liouville theory (see \cite{Coddington1955}).
    
\begin{thm}
\label{Cons_of_SLT}
    The eigenvalues and the corresponding eigenfunctions of 
    \eqref{Eigenvalue_Problem} satisfy the following properties:
    \begin{enumerate}
        \item[(i)]\label{First} For each $\tau\in[0,1]$, the eigenvalue problem \eqref{Eigenvalue_Problem} has an infinite sequence of real eigenvalues which satisfy
        \begin{equation*}
            \lambda_1(\tau)>\lambda_2(\tau)>\cdots>\lambda_k(\tau)>\lambda_{k+1}(\tau)>\cdots \to -\infty \quad \text{as } k\to\infty.
        \end{equation*}
        \item[(ii)]\label{Second} For each $\tau\in[0,1]$ and $k\in\N$, the eigenvalue $\lambda_k(\tau)$ is simple. The corresponding eigenfunctions $\phi_k(\tau,\cdot)$ can be chosen so that, for each $\tau\in[0,1]$, the family $\{\phi_k(\tau,\cdot)\}_{k\in\N}$ forms an orthonormal basis of $\mcH[\tau]$. 
        \item[(iii)]\label{Third} For each $k\in\N,$ the map $\tau\mapsto \lambda_k(\tau)$ is $C^1([0,1]),$ and the map $\tau\mapsto\phi_k(\tau, \cdot)$ is $C^1([0,1]; L^2(0,1));$ see, e.g., \cite{Kato1966}.
        \item[(iv)] The operator $\mcA(\tau)$ is self-adjoint on $\mcH[\tau]$.
    \end{enumerate}
\end{thm}
    Using $(i)$ and $(iii)$, an application of Dini's theorem yields
    \begin{equation}
    \label{Dini_yields}
        \sup_{\tau\in[0,1]}\lambda_k(\tau)\to -\infty \quad \text{as } k\to\infty.
    \end{equation}
    In view of \eqref{Dini_yields}, the set 
    \begin{equation*}
        \Big\{k\in\mathbb{N}:\sup_{\tau\in[0,1]}\lambda_k(\tau)\ge 0\Big\}
    \end{equation*} 
    is finite. Hence we can define
    \begin{equation}
    \label{Defn_of_n}
    n:=
    \begin{cases}
        \displaystyle \max\Big\{k\in\mathbb N:\sup_{\tau\in[0,1]}\lambda_k(\tau)\ge0\Big\},& \text{if the set is nonempty},\\[2mm]
        0,& \text{otherwise}.
    \end{cases}
    \end{equation}

\begin{remark}
    If $n=0$, the operator $\mcA(\tau)$ is uniformly exponentially stable for $\tau\in[0,1]$, and hence no finite-dimensional feedback is required. In this case, the finite-dimensional unstable component is absent, and the subsequent stabilization argument can be carried out using the Lyapunov estimate constructed later for the stable part alone. Hence, in what follows, we consider the case $n\geq 1$.
\end{remark}

    Consequently, $\sup_{\tau\in[0,1]}\lambda_{n+1}(\tau)<0$. Set $\eta:=-\sup_{\tau\in[0,1]}\lambda_{n+1}(\tau)>0$. Then, for every $k>n$, $\sup_{\tau\in[0,1]}\lambda_k(\tau)\le -\eta<0$. Equivalently, for every $k>n$,
    \begin{equation*}
        \sup_{t\in[0,1/\epsilon]}\lambda_k(\epsilon t)\le -\eta<0.
    \end{equation*}

\subsection{Decomposition of the system into stable and unstable parts}
    Let 
    \begin{equation*}
        z\in L^2(0,T;\mathcal D)\cap H^1(0,T;L^2(0,1))
    \end{equation*}
    be a solution of \eqref{z_eqn_in_operator_form}. This regularity will later be established for the solution of the closed-loop system in Theorem~\ref{Thm_Soln_Trans_sys}. Since $H^1(0,T;L^2(0,1))\hookrightarrow C([0,T];L^2(0,1))$, it follows that, for every $t\in[0,T]$, $z(t,\cdot)$ admits the following eigenfunction expansion:
    \begin{equation*}
        z(t,\cdot) = \sum_{k=1}^\infty z^{\epsilon}_k(t)\,\phi_k(\epsilon t, \cdot),\,\,\text{where}\,\,\, z^{\epsilon}_k(t):= \int_0^1 z(t,x)\,\phi_k(\epsilon t,x)\,\rho(\epsilon t,x)\,dx.
    \end{equation*}
    For simplicity of notation, we suppress the dependence on $\epsilon$ and write $z_k^{\epsilon}(t)\equiv z_k(t)$. The same convention applies to other functions later on.
    
    For each $\tau\in[0,1]$, let $\Pi(\tau)\in\mathcal L(L^2(0,1))$ denote the orthogonal projection, with respect to the inner product of $\mcH[\tau]$, onto the subspace $\mathrm{span}\{\phi_k(\tau,\cdot)\}_{k=1}^n$. More precisely, for $\tau\in[0,1]$ and $z\in L^2(0,1)$,
    \begin{equation}
    \label{Pi_tau_z}
        \Pi(\tau)z=\sum_{k=1}^n\big\langle z, \phi_k(\tau,\cdot) \big\rangle_{\mcH[\tau]}\,\phi_k(\tau,\cdot).
    \end{equation}
    It is immediate from the definition of $\Pi(\tau)$ that, for each $\tau\in[0,1]$ and $z\in\mcD$,
    \begin{equation}
    \label{Commu_of_Pi_A}
        \Pi(\tau)\mcA(\tau)z=\mcA(\tau)\Pi(\tau)z.
    \end{equation}
    Note that the mapping $\tau\mapsto\Pi(\tau)$ is continuously differentiable. More precisely, we have the following result.
\begin{lem}
    Let $\Pi$ be defined in \eqref{Pi_tau_z}. Then, $\Pi(\cdot)\in C^1\big([0,1];\mcL(L^2(0,1))\big)$ and for each $\tau\in[0,1]$ and $z\in \mcH[\tau]$, we have
    \begin{align}
    \label{Pi_prime_tau_z}
        {}\qquad {}\Pi^{\prime}(\tau)z = \sum_{k=1}^n \bigg\langle z,\, \phi_k(\tau,\cdot)\bigg\rangle_{\mcH[\tau]} & \,\frac{\partial\phi_k}  {\partial\tau}  (\tau,\cdot) + \sum_{k=1}^n \bigg\langle z, \frac{\partial\phi_k}{\partial\tau}(\tau,\cdot)\bigg\rangle_{\mcH[\tau]}\,\phi_k(\tau,\cdot)   \notag  \\
        & + \sum_{k=1}^n
        \bigg\langle z\left(\int_0^x p_\tau(\tau,s)\,ds\right),\phi_k(\tau,\cdot)\bigg\rangle_{\mcH[\tau]}\phi_k(\tau,\cdot).
    \end{align}
\end{lem}
\begin{proof}
    By property~(iii) of Theorem~\ref{Cons_of_SLT}, $ \tau\mapsto\phi_k(\tau,\cdot)\in C^1([0,1];L^2(0,1)).$ Moreover, by Remark~\ref{rem:p_q_u_regularity},
    \begin{equation*}
        \rho_\tau(\tau,x)=\rho(\tau,x)\int_0^x p_\tau(\tau,s)\,ds, \,\,\text{and hence} \,\,\tau\mapsto\rho(\tau,\cdot)\phi_k(\tau,\cdot) \in C^1([0,1];L^2(0,1)).
    \end{equation*}
    Since
    \begin{equation*}
        \Pi(\tau)z=\sum_{k=1}^n\big\langle z,\rho(\tau,\cdot)\phi_k(\tau,\cdot)\big\rangle_{L^2}\phi_k(\tau,\cdot),
    \end{equation*}
    and, for $v,w\in L^2(0,1)$, $\big\|z\mapsto\langle z,v\rangle_{L^2}w\big\|_{\mcL(L^2(0,1))}= \|v\|_{L^2}\|w\|_{L^2},$
    the finite-rank representation above yields
    \begin{equation*}
        \Pi(\cdot)\in C^1\big([0,1];\mcL(L^2(0,1))\big).
    \end{equation*}
    Differentiating the preceding expression with respect to $\tau$ and using the formula for $\rho_\tau$ gives \eqref{Pi_prime_tau_z}.
\end{proof}
    Let us consider the $n$ dimensional subspace $\mathbf{V}^{\tau}_n:=\mathrm{span}\{\phi_k(\tau,\cdot)\}_{k=1}^n$ of $\mcH[\tau]$. Observe that the projected component $\Pi(\tau)z(t,\cdot)$ of $z(t,\cdot)$ belongs to $\mathbf{V}^{\tau}_n$. We now derive the system of equations satisfied by the components
    \begin{equation}
    \label{component_z_k}
        z_k(t)= \big\langle z(t,\cdot), \phi_k(\epsilon t,\cdot)\big\rangle_{\mcH[\epsilon t]}=\int_0^1 z(t,x)\,\phi_k(\epsilon t,x)\,\rho(\epsilon t,x)\,dx,\quad k=1,2,\ldots,n, 
    \end{equation}
    of $\Pi(\tau)z(t,\cdot)$.
\begin{lem}
    Let $n$ be as in \eqref{Defn_of_n}. Let $\epsilon>0$. For all $t\in[0,T]$ with $T=1/{\epsilon}$, let $z(t,\cdot)$ be the solution of \eqref{z_eqn_in_operator_form} corresponding to a control $h\in L^2((0,T)\times(0,1))$ which is of the form
    \begin{equation*}
        h(t,\cdot)=\sum_{k=1}^n h_k(t)\phi_k(\epsilon t,\cdot),
    \end{equation*}
    where $h_k\in L^2(0,T),$ $k=1,\dots,n$. We set $\mathbf h(t):= \bigl(h_1(t),\dots,h_n(t)\bigr)^\top\in\mathbb R^n.$ Moreover let 
    \begin{equation}
    \label{Pi_epsilon_t_z}
        \Pi(\epsilon t) z(t,\cdot)=\sum_{k=1}^n z_k(t)\,\phi_k(\epsilon t,\cdot),
    \end{equation}
    where $z_k(t)$, $k=1,2,\dots,n$, are given by \eqref{component_z_k}. Then $z_k(t)$ satisfy 
    \begin{equation}
    \label{ODE_satisfied_by_z_k}
        \frac{dz_k}{dt}(t) = \lambda_k(\epsilon t) z_k(t) + \sum_{j=1}^n h_j(t) b_{j k}(\epsilon t) + {r}_k(\epsilon,t),\quad k=1,\dots,n,
    \end{equation}
    where, for $j,k=1,\dots,n$, $b_{jk}(\epsilon t)$ are given by
    \begin{equation*}
        b_{jk}(\epsilon t)= b_{kj}(\epsilon t) =  \big\langle \chi_{\omega} \phi_k(\epsilon t,\cdot), \phi_j(\epsilon t,\cdot)\big\rangle_{\mcH[\epsilon t]} = \int_{\omega}\phi_k(\epsilon t,x) \phi_j(\epsilon t,x)\rho(\epsilon t,x) dx,
    \end{equation*}
    and
    \begin{equation*}
        {r}_k(\epsilon,t) = \left\langle \mathcal{R}^1(\epsilon,t,\cdot),\ \phi_k(\epsilon t,\cdot) \right\rangle_{\mcH[\epsilon t]}, \quad k=1,\dots,n,
    \end{equation*}
    and $\mcR^{1}(\epsilon, t, \cdot)$ is given by
    \begin{align}
    \label{Defn_of_R1}
        \mcR^{1}(\epsilon, t, \cdot) = \Pi(\epsilon t) \mcR(\epsilon, t, \cdot)
        &+\epsilon \sum_{k=1}^n \bigg\langle z(t,\cdot), \frac{\partial\phi_k}{\partial\tau}(\epsilon t,\cdot)\bigg\rangle_{\mcH[\epsilon t]}\,\phi_k(\epsilon t,\cdot) \notag \\
        & +\epsilon\sum_{k=1}^n\bigg\langle z(t,\cdot)
        \left(\int_0^x p_\tau(\epsilon t,s)\,ds\right),
        \phi_k(\epsilon t,\cdot)\bigg\rangle_{\mcH[\epsilon t]}
        \phi_k(\epsilon t,\cdot).
    \end{align}     
\end{lem}
\begin{proof}
    Differentiating \eqref{Pi_epsilon_t_z} with respect to $t$ and using \eqref{Pi_prime_tau_z}, \eqref{z_eqn_in_operator_form} and \ref{Commu_of_Pi_A}, we have
    \begin{align*}
        \sum_{k=1}^n \frac{dz_k}{dt}(t)\,&\phi_k(\epsilon t,\cdot)+\epsilon \sum_{k=1}^n z_k(t)\,\frac{\partial \phi_k}{\partial\tau}(\epsilon t,\cdot) \\
        &= \epsilon \Pi^{\prime}(\epsilon t)z(t,\cdot)+ \Pi(\epsilon t)z_t(t,\cdot)\\
        &=\epsilon \sum_{k=1}^n \big\langle z(t,\cdot), \phi_k(\epsilon t,\cdot)\big\rangle_{\mcH[\epsilon t]}\,\frac{\partial\phi_k}{\partial\tau}(\epsilon t,\cdot) +\epsilon \sum_{k=1}^n \bigg\langle z(t,\cdot), \frac{\partial\phi_k}{\partial\tau}(\epsilon t,\cdot)\bigg\rangle_{\mcH[\epsilon t]}\,\phi_k(\epsilon t,\cdot)\\
        &\qquad \qquad \, + \epsilon \sum_{k=1}^n \bigg\langle z\left(\int_0^x p_\tau(\tau,s)\,ds\right), \phi_k(\tau,\cdot) \bigg\rangle_{\mcH[\epsilon t]}  \phi_k(\tau,\cdot)   \\
        &\qquad \qquad \qquad \quad + \mcA(\epsilon t) \Pi(\epsilon t) z(t, \cdot)+  \Pi(\epsilon t) \mathcal{B}(h)(t)+ \Pi(\epsilon t) \mcR(\epsilon, t, \cdot).
    \end{align*}
    Thus,
    \begin{align*}
        \sum_{k=1}^n \frac{dz_k}{dt}(t)\phi_k(\epsilon t,\cdot)={}&
        \sum_{k=1}^n \lambda_k (\epsilon t) z_k(t) \phi_k(\epsilon t,\cdot)+\Pi(\epsilon t)B(h)(t)+\Pi(\epsilon t)R(\epsilon,t,\cdot)\\
        &+\epsilon\sum_{k=1}^n\left\langle z(t,\cdot),\frac{\partial\phi_k}{\partial\tau}(\epsilon t,\cdot)\right\rangle_{\mathcal H[\epsilon t]}\phi_k(\epsilon t,\cdot)\\
        &+\epsilon\sum_{k=1}^n\left\langle z(t,\cdot)\left(\int_0^x p_\tau(\epsilon t,s)\,ds\right),\phi_k(\epsilon t,\cdot)\right\rangle_{\mathcal H[\epsilon t]}\phi_k(\epsilon t,\cdot).
    \end{align*}
    Further in view of \eqref{Defn_of_R1}, we can write
    \begin{equation}
    \label{dz_k_dt_Phi_k}
        \sum_{k=1}^n \frac{dz_k}{dt}(t)\,\phi_k(\epsilon t,\cdot) = \sum_{k=1}^n \lambda_k (\epsilon t) z_k(t) \phi_k(\epsilon t,\cdot) +  \Pi(\epsilon t) \mathcal{B}(h)(t) + \mcR^{1}(\epsilon, t, \cdot),
    \end{equation}
    We now compute $\Pi(\epsilon t)\mathcal{B}(h)(t)$:
    \begin{align*}
        \Pi(\epsilon t) \mathcal{B}(h)(t) =\Pi(\epsilon t)\chi_{\omega}h(t,\cdot) &= \sum_{k=1}^n h_k(t) \Pi(\epsilon t) \big(\chi_{\omega} \phi_k(\epsilon t,\cdot)\big) \\
        & = \sum_{k=1}^n \sum_{j=1}^n h_k(t) b_{kj}(\epsilon t)\,\phi_j(\epsilon t,\cdot),
    \end{align*}
    where, for $k,j=1,2,\dots,n,$
    \begin{equation*}
        b_{kj}(\tau) :=  \big\langle \chi_{\omega} \phi_k(\tau,\cdot), \phi_j(\tau,\cdot)\big\rangle_{\mcH[\tau]} = \int_{\omega}\phi_k(\tau,x) \phi_j(\tau,x)\rho(\tau,x) dx = b_{jk}(\tau).
    \end{equation*}
    We introduce an operator $\mathcal{B}_1\in\mcal{L}(\mathbb{R}^n, \mathbf{V}^{\tau}_n)$ defined by
    \begin{equation}
    \label{B_1_Tau_h}
        \mathcal{B}_1(\tau)\mathbf{v} :=\sum_{k=1}^n \sum_{j=1}^n v_{k} b_{kj}(\tau)\,\phi_j(\tau,\cdot),
    \end{equation}
    for $\mathbf{v}=(v_1,v_2,\ldots,v_n)^{\top}\in \R^n.$ Identifying $\mathbf{V}^{\tau}_n$ with $\mathbb R^n$ through the basis $\{\phi_k(\tau,\cdot)\}_{k=1}^n$, the operator $\mathcal{B}_1(\tau)$ can be represented by the symmetric $n\times n$  matrix 
    \begin{align} 
    \label{B_1_Matrix} 
        \mathcal{B}_1(\tau)=
        \begin{pmatrix}
            b_{11}(\tau) & b_{21}(\tau) & \cdots & b_{n1}(\tau) \\
            b_{12}(\tau) & b_{22}(\tau) & \cdots & b_{n2}(\tau) \\
            \vdots & \vdots & \ddots & \vdots \\
            b_{1n}(\tau) & b_{2n}(\tau) & \cdots & b_{nn}(\tau)
        \end{pmatrix}.
    \end{align}
    Substituting \eqref{B_1_Tau_h} in the equation \eqref{dz_k_dt_Phi_k}, we have
    \begin{equation*}
        \sum_{k=1}^n \frac{dz_k}{dt}(t)\,\phi_k(\epsilon t,\cdot) = \sum_{k=1}^n \lambda_k (\epsilon t) z_k(t) \phi_k(\epsilon t,\cdot) +  \sum_{k=1}^n \sum_{j=1}^n h_k(t) b_{kj}(\epsilon t)\,\phi_j(\epsilon t,\cdot) + \mcR^{1}(\epsilon, t, \cdot).
    \end{equation*}
    Projecting the above system onto each $\phi_k(\epsilon t,\cdot)$ yields \eqref{ODE_satisfied_by_z_k}. Thus, the result follows.
\end{proof}
    Let us denote
    \begin{align*}
        X_1(t):= \begin{pmatrix} 
		            z_{1}(t)\\
		              z_{2}(t)\\
		            \vdots \\
		            z_{n}(t) 
	                \end{pmatrix},\
	      \mcA_{1}(\tau):= \begin{pmatrix} 
		\lambda_{1}(\tau) & 0 & \dots & 0\\
		  0 & \lambda_{2}(\tau) & \dots & 0\\
		  \vdots & \vdots & \ddots & \vdots \\
		  0 & 0 & \dots & \lambda_{n}(\tau)
	    \end{pmatrix},\
        \mathcal{R}_1(\epsilon,t):= \begin{pmatrix} 
		r_{1}(\epsilon,t)\\
		r_{2}(\epsilon,t)\\
		\vdots \\
		  r_{n}(\epsilon,t)\end{pmatrix}.
    \end{align*}
    Therefore, system \eqref{ODE_satisfied_by_z_k} can be written as
    \begin{equation}
    \label{ODE_X_1}
        \frac{dX_1}{dt}(t)=\mcA_{1}(\epsilon t) X_1(t)+\mathcal{B}_1(\epsilon t) \mathbf{h}(t)+\mcR_1(\epsilon,t).
    \end{equation}
    We denote $\mathcal{H}[\tau](\omega):=L^2(\omega,\rho(\tau,x)\,dx)$. The following lemma shows that the matrix $\mathcal{B}_1(\tau)$ is invertible.
   
\begin{lem}
    For each $\tau\in[0,1]$, the matrix $\mathcal{B}_1(\tau)$ defined in \eqref{B_1_Matrix} is invertible.
\end{lem}
\begin{proof}
    It suffices to show that the family $\{\phi_k(\tau,\cdot)\}_{k=1}^n$ is linearly independent in $\mathcal{H}[\tau](\omega)$, since $\mathcal B_1(\tau)$ is the corresponding Gram matrix. Suppose $\gamma_1,\dots,\gamma_n\in\mathbb R$ satisfy
    \begin{equation*}
        \sum_{k=1}^n \gamma_k\,\phi_k(\tau,x)=0,\qquad x\in\omega.
    \end{equation*}
    Define $\psi_0(\tau,\cdot):= \sum_{k=1}^n\gamma_k\phi_k(\tau,\cdot)$, and for $j=1,\dots,n-1$ set 
    \begin{equation*}
       \psi_j(\tau, \cdot):=\big(\mcA(\tau)-\lambda_j(\tau)\big)\psi_{j-1}(\tau, \cdot).
    \end{equation*}
    Since $\mcA(\tau)\phi_k(\tau,\cdot) =\lambda_k(\tau)\phi_k(\tau,\cdot)$, a straightforward induction gives
    \begin{equation}
    \label{Eq_Psi_j_Formula}
        \psi_j(\tau, \cdot)=\sum_{k=j+1}^n \gamma_k\left(\prod_{i=1}^{j}\big(\lambda_k(\tau)-\lambda_i(\tau)\big)\right)\phi_k(\tau,\cdot),\qquad j=0,1,\dots,n-1.
    \end{equation}
    In particular, taking $j=n-1$,
    \begin{equation}
    \label{Eq_Psi_n_Formula}
        \psi_{n-1}(\tau, \cdot)=\gamma_n\left(\prod_{i=1}^{n-1}\big(\lambda_n(\tau)-\lambda_i(\tau)\big)\right)\phi_n(\tau,\cdot).
    \end{equation}
    
    Since, for each fixed $\tau\in[0,1]$, the operator $\mcA(\tau)$ is a local second-order differential operator whose coefficients depend only on $x$, it maps functions vanishing identically on $\omega$ to functions vanishing identically on $\omega$. As $\psi_0(\tau, \cdot) \equiv0$ on $\omega$, it follows inductively that $\psi_j(\tau, \cdot) \equiv0$ on $\omega$ for every $j=1,\dots,n-1$; in particular $\psi_{n-1}(\tau, \cdot) \equiv0$ on $\omega$.

    By \eqref{Eq_Psi_n_Formula}, $\psi_{n-1}$ is a scalar multiple of the eigenfunction $\phi_n(\tau,\cdot)$, and therefore
    \begin{equation*}
        \mcA(\tau)\psi_{n-1}(\tau, \cdot) = \lambda_n(\tau)\psi_{n-1}(\tau, \cdot).
    \end{equation*}
    Since $\mcA(\tau)$ is the second-order differential operator defined in \eqref{Defn_A_tau}, this is equivalent to the linear second-order ODE
    \begin{equation*}
        \psi_{n-1}''(\tau, \cdot)+p(\tau,\cdot)\psi_{n-1}'(\tau, \cdot)+q(\tau,\cdot)\psi_{n-1}(\tau, \cdot)=\lambda_n(\tau)\psi_{n-1}(\tau, \cdot) \quad\text{on } (0,1).
    \end{equation*}
    Fix $x_0\in\omega$. Since $\psi_{n-1}(\tau, \cdot) \equiv0$ on the open interval $\omega$, we have $\psi_{n-1}(\tau, x_0)=\psi_{n-1}'(\tau, x_0)=0.$ As $p(\tau,\cdot),q(\tau,\cdot)$ are continuous, the Cauchy-Lipschitz uniqueness theorem for this linear second-order ODE forces $\psi_{n-1}(\tau, \cdot)\equiv0$ on $(0,1)$. Since $\phi_{n}(\tau,\cdot)\not\equiv0,$ and $\lambda_{n}(\tau)\neq\lambda_{i}(\tau)$ for $i=1,\ldots,n-1.$ Hence, \eqref{Eq_Psi_n_Formula} implies that $\gamma_{n}=0$. Substituting $\gamma_n=0$ into \eqref{Eq_Psi_j_Formula} at $j=n-2$ shows that the $k=n$ term drops out, so that
    \begin{equation*}
        \psi_{n-2}(\tau, \cdot) =\gamma_{n-1}\left(\prod_{i=1}^{n-2}\big(\lambda_{n-1}(\tau)-\lambda_i(\tau)\big)\right)\phi_{n-1}(\tau,\cdot),
    \end{equation*}
    which also vanishes identically on $\omega$. Repeating the above argument with $\psi_{n-2}(\tau, \cdot)$ in place of $\psi_{n-1}(\tau, \cdot)$ yields $\gamma_{n-1}=0$. Iterating this procedure downward through $j=n-3,\ldots,0$ gives $\gamma_k=0$ for every $k=1,\ldots,n$. Therefore, the family $\{\phi_k(\tau,\cdot)\}_{k=1}^n$ is linearly independent in $\mcH[\tau](\omega)$. Consequently, the corresponding Gram matrix $\mathcal B_1(\tau)$ is invertible.
\end{proof}

\begin{remark} 
    Since $\mathcal B_1(\tau)$ is a symmetric Gram matrix, it is positive semidefinite. Moreover, by the preceding result, $\mathcal B_1(\tau)$ is invertible. Hence, $\mathcal B_1(\tau)$ is positive definite for every $\tau\in[0,1]$. 
\end{remark}

    We now construct a constant diagonal feedback gain matrix that stabilizes the finite-dimensional system. Let $I_n$ denote the $n\times n$ identity matrix.
\begin{lem}
\label{Pole_placement}
    There exists $\kappa >0$ such that, for every $\tau\in[0,1],$ the matrix $\mathcal{A}_1(\tau)+\mathcal{B}_1(\tau)\mathcal{K}_1$ is symmetric negative definite, where $\mathcal{K}_1 = -\kappa I_n$ is the feedback gain matrix. Moreover, there exists a $C^1$ map $\tau\mapsto P(\tau)$ on $[0,1]$, where $P(\tau)$ is a real symmetric positive definite $n\times n$ matrix such that the following identity 
    \begin{equation}
    \label{Identity_Lyapunov}
        P(\tau)\big( \mathcal{A}_1(\tau)+\mathcal{B}_1(\tau)\mathcal{K}_1 \big)+ \big(\mathcal{A}_1(\tau)+\mathcal{B}_1(\tau)\mathcal{K}_1 \big)^*P(\tau)=-I_n,
    \end{equation}
    holds for all $\tau\in[0,1].$
\end{lem}
\begin{proof}
    Since $\mathcal B_1(\tau)$ is symmetric positive definite for every $\tau\in[0,1]$ and $\tau\mapsto\mathcal B_1(\tau)$ is continuous, compactness of $[0,1]$ yields
    \begin{equation*}
        \mu_0:=\min_{\tau\in[0,1]}\lambda_{\min}\bigl(\mathcal B_1(\tau)\bigr)>0.
    \end{equation*}
    Moreover, since $\mathcal A_1(\tau)$ is diagonal and depends continuously on $\tau$, we may set
    \begin{equation*}
        \mu_1:=\max_{\substack{\tau\in[0,1]\\
        1\leq i\leq n}}\lambda_i(\tau)<\infty.
    \end{equation*}
    Choose $\kappa>0$ such that $\kappa \mu_{0} > \mu_{1}.$ Then, for every $X\in\mathbb{R}^n\setminus\{0\}$ and $\tau\in[0,1]$,
    \begin{align*}
        \left\langle\bigl(\mathcal A_1(\tau)+\mathcal B_1(\tau)\mathcal K_1\bigr)X,X\right\rangle_2=\langle\mathcal A_1(\tau)X,X\rangle_2-\kappa\langle\mathcal B_1(\tau)X,X\rangle_2\leq-(\kappa\mu_0-\mu_1)\|X\|_2^2<0.
    \end{align*}
    Thus $\mathcal A_1(\tau)+\mathcal B_1(\tau)\mathcal K_1$ is symmetric negative definite, uniformly with respect to $\tau\in[0,1]$. Set $P(\tau):=-\frac12\bigl(\mathcal A_1(\tau)+\mathcal B_1(\tau)\mathcal K_1\bigr)^{-1}.$ Then $P(\tau)$ is symmetric positive definite. Moreover, since the mappings $\tau\mapsto\mathcal A_1(\tau)$ and
    $\tau\mapsto\mathcal B_1(\tau)$ are of class $C^1$, so is
    $\tau\mapsto P(\tau)$. The identity \eqref{Identity_Lyapunov} follows directly from the definition of $P(\tau)$.
\end{proof}
    The matrix $\mathcal{K}_1$, obtained in Lemma~\ref{Pole_placement}, helps us to construct the state feedback control coefficient vector $\mathbf{h}(t)$ of the form
    \begin{equation}
    \label{Feedback_law}
        \mathbf{h}(t)=\mathcal{K}_1 X_1(t)=-\kappa X_1(t),\quad t\in [0,1/\epsilon]
    \end{equation}
    and stabilizes the finite-dimensional linear control system $\frac{dX_1}{dt}(t)=\mcA_{1}(\epsilon t) X_1(t)+\mathcal{B}_1(\epsilon t) \mathbf{h}(t)$. The feedback control corresponding to \eqref{Feedback_law} is given by
    \begin{equation}
    \label{Feedback}
        h(t,\cdot)=\sum_{k=1}^n h_k(t)\phi_k(\epsilon t,\cdot) =-\kappa\sum_{k=1}^n z_k(t)\phi_k(\epsilon t,\cdot)=-\kappa\Pi(\epsilon t)z(t,\cdot).
    \end{equation}
    Subsequently, we will show that the feedback control \eqref{Feedback} stabilizes the full infinite-dimensional system \eqref{z_eqn_in_operator_form}, provided $\epsilon$ is chosen sufficiently small, which will assist us in formulating a suitable Lyapunov functional to achieve stabilization of that system.

\subsection{Existence, uniqueness, and regularity}
    Let $0<T\leq 1/\epsilon$. In this subsection, we study the existence, uniqueness, and regularity of solutions to the closed-loop system 
    \begin{align}
    \label{Close_loop_z_eqn}
        \begin{cases}
            z_t(t, \cdot)&=\mcA(\epsilon t)z(t, \cdot) - \kappa\chi_\omega \Pi(\epsilon t) z(t,\cdot)+\mcR(\epsilon,t,\cdot),\quad \text{in}\,\ (0,T),\\
            z(0)&=0,
        \end{cases}
    \end{align}
    This system~\eqref{Close_loop_z_eqn} is obtained from \eqref{z_eqn_in_operator_form} by substituting the feedback control \eqref{Feedback}.

    The regularity established here will be used in the next section to construct a suitable Lyapunov functional. For $z\in\mcH[\tau]$ let us define
    \begin{equation}
    \label{Defn_of_Xi}
        \xi(\tau)z:= - \kappa\chi_\omega \Pi(\tau) z,\quad \text{with}\quad z=\sum_{j=1}^{\infty} z_j(\tau)\phi_j(\tau).
    \end{equation}
    Since $\Pi(\tau)$ is the orthogonal projection in $\mcH[\tau]$, we have
    \begin{equation*}
        \|\xi(\tau)z\|_{\mathcal H[\tau]}^2=\kappa^2\int_\omega|\Pi(\tau)z(x)|^2\rho(\tau,x)\,dx\leq\kappa^2\|\Pi(\tau)z\|_{\mathcal H[\tau]}^2\leq\kappa^2\|z\|_{\mathcal H[\tau]}^2.
    \end{equation*}
    Hence, $\xi(\tau)\in\mathcal L(\mathcal H[\tau])$ and $\|\xi(\tau)\|_{\mathcal L(\mathcal H[\tau])}\leq\kappa$.
    Moreover, using \eqref{Norm_Equivalence},
    \begin{equation}\label{Regularity_of_xi}
        \sup_{\tau\in[0,1]}\|\xi(\tau)\|_{\mcL(L^2(0,1))}\leq C\kappa,
    \end{equation}
    for some constant $C>0$ independent of $\tau$. The closed-loop system \eqref{Close_loop_z_eqn} can be written in terms of $\xi$ as
    \begin{align}
    \label{Close_loop_z_eqn_in_Xi}
        \begin{cases}
            z_t(t, \cdot)=\mcA(\epsilon t)z(t, \cdot)+\xi(\epsilon t) z(t,\cdot)+\mcR(\epsilon,t,\cdot),\quad \text{in}\,\ (0,T),\\
            z(0)=0.
        \end{cases}
    \end{align}
    Using \eqref{Defn_A_tau}, above system \eqref{Close_loop_z_eqn_in_Xi} can be rewritten as
    \begin{align}
    \label{Close_loop_z_eqn_expanded}
        \begin{cases}
            z_t= z_{xx}-\alpha \overline{u}^{\delta}(\epsilon t,\cdot)\,z_x+\Big(F'(\overline{u}(\epsilon t,\cdot))-\alpha\delta \overline{u}^{\delta-1}(\epsilon t,\cdot)\overline{u}_x(\epsilon t,\cdot)\Big)z+\xi(\epsilon t)z&\\
            \qquad-\alpha\Big(\delta \overline{u}^{\delta-1} z+\sum_{j=2}^{\delta}\binom{\delta}{j}z^j\overline{u}^{\delta-j}\Big)z_x-\alpha\Big(\sum_{j=2}^{\delta}\binom{\delta}{j}z^j\overline{u}^{\delta-j}\Big)\overline{u}_x&\\
            \qquad+\Big(F(z+\overline{u})-F(\overline{u})-F'(\overline{u})z\Big)-\epsilon \overline{u}_\tau(\epsilon t,\cdot),  &\text{in }(0,T)\times (0,1),\\
            z(t,0)=z(t,1)=0,  &t\in (0,T),   \\  
            z(0, x)  = 0,  &x\in (0,1).  
        \end{cases}
    \end{align}
    
\begin{thm}
\label{Thm_Soln_Trans_sys}
    The system \eqref{Close_loop_z_eqn_in_Xi}, equivalently \eqref{Close_loop_z_eqn}, admits a unique solution satisfying
    \begin{equation}
    \label{Regularity_Thm_Soln_Trans_sys}
        z\in L^2\bigl(0,T;H^2(0,1)\cap H_0^1(0,1)\bigr)\cap H^1\bigl(0,T;L^2(0,1)\bigr)\cap C\bigl([0,T];H_0^1(0,1)\bigr).
    \end{equation}
\end{thm}
    The proof of Theorem~\ref{Thm_Soln_Trans_sys} follows a similar strategy to that of Theorem~3.18 in \cite{Shirshendu20} and is provided in the Appendix.
    
\subsection{Construction of Lyapunov functional} 
    In this subsection, we construct a Lyapunov functional for the stabilization of system \eqref{z_eqn_in_operator_form}. To this end, we first establish the following quadratic-form representation of $\mcA(\tau)$.
\begin{lem}
\label{lem:quadratic_form_spectral}
    For every $\tau\in[0,1]$, the quadratic-form domain associated with $-\mcA(\tau)$ is $\mcH_0^1[\tau]$, and
    \begin{equation}
    \label{quadratic_form_spectral}
        \int_0^1\rho(\tau,x)\Bigl(|z_x(x)|^2-q(\tau,x)|z(x)|^2\Bigr)\,dx=-\sum_{k=1}^{\infty}\lambda_k(\tau)\left|\big\langle z,\phi_k(\tau,\cdot)\big\rangle_{\mcH[\tau]}\right|^2
    \end{equation}
    for every $z\in\mcH_0^1[\tau]$.
\end{lem}
\begin{proof}
    By Remark~\ref{rem:p_q_u_regularity}, $q$ is uniformly bounded on $[0,1]\times[0,1]$. Set $\mu:=1+\|q\|_{L^\infty},$ so that $\mu-q(\tau,x)\geq1$ for all $(\tau,x)\in[0,1]\times[0,1].$ For $z\in\mcD(\mcA(\tau))$, integration by parts, together with the Dirichlet boundary condition, gives
    \begin{equation*}
        \big\langle(\mu I-\mcA(\tau))z,z\big\rangle_{\mcH[\tau]}=\int_0^1\rho(\tau,x)|z_x|^2\,dx+\int_0^1\rho(\tau,x)\bigl(\mu-q(\tau,x)\bigr)|z|^2\,dx.
    \end{equation*}
    Define, for $z\in\mcH_0^1[\tau]$,
    \begin{equation*}
        \mathfrak b_\tau(z,z):= \int_0^1\rho(\tau,x)|z_x|^2\,dx +\int_0^1\rho(\tau,x)\bigl(\mu-q(\tau,x)\bigr)|z|^2\,dx.
    \end{equation*}
    Since $1\leq \mu-q(\tau,x)\leq 1+2\|q\|_{L^\infty}$ we have $\|z\|_{\mcH_0^1[\tau]}^2\leq\mathfrak b_\tau(z,z)\leq\bigl(1+2\|q\|_{L^\infty}\bigr)\|z\|_{\mcH_0^1[\tau]}^2.$ Hence $\mathfrak b_\tau$ is a closed coercive quadratic form on $\mcH_0^1[\tau]$. By the representation theorem for closed forms, $\mcD\bigl((\mu I-\mcA(\tau))^{1/2}\bigr)=\mcH_0^1[\tau].$ Since the addition of the bounded operator $\mu I$ does not change the form domain, the quadratic-form domain associated with $-\mcA(\tau)$ is also $\mcH_0^1[\tau]$. Moreover, $(\mu I-\mcA(\tau))\phi_k(\tau,\cdot)=(\mu-\lambda_k(\tau))\phi_k(\tau,\cdot).$ Thus, by the spectral theorem, for every $z\in\mcH_0^1[\tau]$,
    \begin{equation*}
        \mathfrak b_\tau(z,z)=\sum_{k=1}^{\infty}\bigl(\mu-\lambda_k(\tau)\bigr)\left|\big\langle z, \phi_k(\tau,\cdot)\big\rangle_{\mcH[\tau]}\right|^2.
    \end{equation*}
    Since $\{\phi_k(\tau,\cdot)\}_{k\in\mathbb N}$ is an orthonormal basis of $\mcH[\tau]$, we have
    \begin{equation*}
        \sum_{k=1}^{\infty}\left|\big\langle z,\phi_k(\tau,\cdot)\big\rangle_{\mcH[\tau]}\right|^2=
        \|z\|_{\mcH[\tau]}^2= \int_0^1\rho(\tau,x)|z(x)|^2\,dx.
    \end{equation*}
    Therefore,
    \begin{equation*}
        \mathfrak b_\tau(z,z)=\mu\int_0^1\rho(\tau,x)|z|^2\,dx-\sum_{k=1}^{\infty}\lambda_k(\tau)\left|\big\langle z, \phi_k(\tau,\cdot)\big\rangle_{\mcH[\tau]}\right|^2.
    \end{equation*}
    Comparing this with the definition of $\mathfrak{b}_\tau$ gives \eqref{quadratic_form_spectral}.
\end{proof}
    Using Lemma~\ref{lem:quadratic_form_spectral}, we now define the Lyapunov functional. Let $c>0$, to be chosen later. For $t\in[0,1/\epsilon]$ and $z(t,\cdot)\in\mcH_0^1[\epsilon t]$, define
	\begin{align}
    \label{Lyapunov_Functional}
        V(t,z(t,\cdot)):={}&c\big\langle P(\epsilon t)X_1(t),X_1(t)\big\rangle_2\notag\\
        &+\frac12\int_0^1\rho(\epsilon t,x)\Big(|z_x(t,x)|^2-q(\epsilon t,x)|z(t,x)|^2\Big)\,dx,
    \end{align}
    where $z_k(t):=\big\langle z(t,\cdot),\phi_k(\epsilon t,\cdot)\big\rangle_{\mcH[\epsilon t]}$, for $k\in\mathbb N,$ and
    $X_1(t):=\begin{pmatrix} z_1(t)&z_2(t)&\cdots&z_n(t) \end{pmatrix}^{\!\top}\in\mathbb R^n.$ By Lemma~\ref{lem:quadratic_form_spectral}, we have
    \begin{equation}
    \label{2nd_Lyapunov_Functional}
        V(t,z(t,\cdot))=c\big\langle P(\epsilon t)X_1(t),X_1(t)\big\rangle_2-\frac12\sum_{k=1}^{\infty} \lambda_k(\epsilon t)z_k^2(t).
    \end{equation}
    As $\tau\mapsto P(\tau)$ is continuous on $[0,1]$ and $P(\tau)$ is a real positive definite symmetric matrix for every $\tau\in[0,1],$ there exist positive constants $c_{11},c_{12}$ independent of $\tau\in[0,1]$ such that
    \begin{equation}
    \label{c11_P_tau_c22_}
        c_{11}||X||_2^2\leq \langle P(\tau) X, X\rangle_2 \leq c_{12}||X||_2^2,\quad \forall\ X\in \mathbb{R}^{n}.  
    \end{equation}
    Thus, $\langle P(\epsilon t) X_1(t), X_1(t)\rangle_2\sim ||X_1(t)||_2^2$, uniformly for $t\in[0,1/\epsilon].$ By property $(iii)$ of Theorem~\ref{Cons_of_SLT}, the maps $\tau\mapsto\lambda_1(\tau),\ \tau\mapsto\lambda_n(\tau)$ are continuous. Therefore, they attain their extrema on $[0,1]$, and we set 
    \begin{equation*}
        \overline{m}:=\max_{\tau\in[0,1]}\lambda_1(\tau), \qquad m:=\min_{\tau\in[0,1]}\lambda_n(\tau).
    \end{equation*}
    Then for $t\in[0,1/\epsilon]$, we have $m \|X_1(t)\|^2_2  \leq \sum_{k=1}^n \lambda_k(\epsilon t) z_k^{2}(t) \leq \overline{m} \|X_1(t)\|^2_2.$ Hence,
    \begin{equation*}
        m\|X_1(t)\|^2_2+ \sum_{k=n+1}^\infty \lambda_k(\epsilon t) z_k^{2}(t) \leq\sum_{k=1}^\infty \lambda_k(\epsilon t) z_k^{2}(t) \leq \overline{m} \|X_1(t)\|^2_2 + \sum_{k=n+1}^\infty \lambda_k(\epsilon t) z_k^{2}(t),
    \end{equation*}
    and therefore
    \begin{align*}
        -\frac{1}{2}\overline{m} \|X_1(t)\|^2_2 - \frac{1}{2}\sum_{k=n+1}^\infty \lambda_k(\epsilon t) z_k^{2}(t)\leq & -\frac{1}{2}\sum_{k=1}^\infty \lambda_k(\epsilon t) z_k^{2}(t) \\
        \leq & -\frac{1}{2}m\|X_1(t)\|^2_2 -\frac{1}{2} \sum_{k=n+1}^\infty \lambda_k(\epsilon t) z_k^{2}(t).
    \end{align*}
    In view of \eqref{2nd_Lyapunov_Functional}, \eqref{c11_P_tau_c22_}, and the above estimates, we obtain
    \begin{equation*}
		\big(c_{11}c-\frac{1}{2}\overline{m}\big)||X_1(t)||^2_2- \frac{1}{2}\sum_{k=n+1}^\infty \lambda_k(\epsilon t) z_k^{2}(t)  \leq V(t,z(t,\cdot)), 
    \end{equation*}
    and
    \begin{equation*}
        V(t,z(t,\cdot))  \leq \Big(c_{12}c-\frac{1}{2}m\Big)||X_1(t)||^2_2- \frac{1}{2}\sum_{k=n+1}^\infty \lambda_k(\epsilon t) z_k^{2}(t). 
	\end{equation*}
    Thus, by choosing $c>0$ large enough in the definition of $V$, we ensure that $c_{11}c-\frac{\overline{m}}{2}>0$, and $c_{12}c-\frac{m}{2}>0.$ Hence
    \begin{equation}
    \label{3rd_Lyapunov_functional}
		   V(t,z(t,\cdot))\sim\|X_1(t)\|^2_2-\sum_{k=n+1}^\infty \lambda_k(\epsilon t) z_k^{2}(t),
	\end{equation}
    uniformly for $t\in[0,1/\epsilon]$. In particular, since $\lambda_k(\epsilon t)\le -\eta<0$ for all $k>n$, it follows that $V(t,z(t,\cdot))$ is positive definite.
    
\begin{lem}
\label{lem:V_equivalence}
    Let $V$ be defined by \eqref{Lyapunov_Functional}. Then, uniformly for $t\in[0,1/\epsilon]$, one has
    \begin{equation}
    \label{V_sim_z_H_1}
        V(t,z)\sim \|z\|_{\mcH_0^1[\epsilon t]}^2,\qquad\forall\, z\in\mcH_0^1[\epsilon t].
    \end{equation}
    Moreover, for every $z\in\mcD(\mcA(\epsilon t))$,
    \begin{equation}
    \label{V_lessim_A_X_1}
        V(t,z)\lesssim \|X_1\|_2^2+\|\mcA(\epsilon t)z\|_{\mcH[\epsilon t]}^2.
    \end{equation}
\end{lem}
\begin{proof}
    We first prove \eqref{V_sim_z_H_1}. Let $t\in[0,1/\epsilon]$ and $z\in\mcH_0^1[\epsilon t]$. By \eqref{Lyapunov_Functional}, \eqref{c11_P_tau_c22_}, and Parseval's identity,
    \begin{align*}
        V(t,z)&=c\big\langle P(\epsilon t)X_1,X_1\big\rangle_2+\frac12\int_0^1\rho(\epsilon t,x)\Big(|z_x(x)|^2-q(\epsilon t,x)|z(x)|^2\Big)\,dx\\
        &\lesssim\|X_1\|_2^2+\|z_x\|_{\mcH[\epsilon t]}^2+\|z\|_{\mcH[\epsilon t]}^2\\ &\lesssim\|z\|_{\mcH_0^1[\epsilon t]}^2,
    \end{align*}
    since $\|X_1\|_2^2=\sum_{k=1}^n z_k^2\leq\sum_{k=1}^{\infty}z_k^2=\|z\|_{\mcH[\epsilon t]}^2.$
    Thus,
    \begin{equation}
    \label{V_lessim_z_H}
        V(t,z)\lesssim\|z\|_{\mcH_0^1[\epsilon t]}^2.
    \end{equation}
    For the converse estimate, Lemma~\ref{lem:quadratic_form_spectral} gives
    \begin{equation*}
        \|z_x\|_{\mcH[\epsilon t]}^2=-\sum_{k=1}^{\infty} \lambda_k(\epsilon t)z_k^2+\int_0^1q(\epsilon t,x)\rho(\epsilon t,x)|z(x)|^2\,dx.
    \end{equation*}
    Hence,
    \begin{align*}
        \|z_x\|_{\mcH[\epsilon t]}^2+\|z\|_{\mcH[\epsilon t]}^2 &\leq -\sum_{k=1}^{\infty} \lambda_k(\epsilon t)z_k^2 + (\|q\|_{\infty}+1) \sum_{k=1}^{\infty}z_k^2\\
        &=-\sum_{k=1}^{n}\lambda_k(\epsilon t)z_k^2+(\|q\|_{\infty}+1)\sum_{k=1}^{n}z_k^2\\
        &\quad-\sum_{k=n+1}^{\infty}\lambda_k(\epsilon t)z_k^2+ (\|q\|_{\infty}+1)\sum_{k=n+1}^{\infty}z_k^2.
    \end{align*}
    By continuity of $\tau\mapsto\lambda_k(\tau)$ on $[0,1]$, there exists $M>0$ such that
    \begin{equation*}
        |\lambda_k(\tau)|\leq M,\qquad k=1,\ldots,n,\quad \tau\in[0,1].
    \end{equation*}
    Therefore,
    \begin{equation*}
        -\sum_{k=1}^{n}\lambda_k(\epsilon t)z_k^2+(\|q\|_{\infty}+1)\sum_{k=1}^{n}z_k^2\leq(M+\|q\|_{\infty}+1)\sum_{k=1}^{n}z_k^2.
    \end{equation*}
    Moreover, since
    \begin{equation*}
        \lambda_k(\epsilon t)\leq-\eta<0,\quad k>n,\,\,\,\text{we have}\,\,\, z_k^2\leq-\frac1\eta\lambda_k(\epsilon t)z_k^2,\quad k>n.
    \end{equation*}
    Consequently,
    \begin{equation*}
        -\sum_{k=n+1}^{\infty}\lambda_k(\epsilon t)z_k^2+(\|q\|_{\infty}+1)\sum_{k=n+1}^{\infty}z_k^2 \leq\left(1+\frac{\|q\|_{\infty}+1}{\eta}\right)\left(-\sum_{k=n+1}^{\infty}\lambda_k(\epsilon t)z_k^2\right).
    \end{equation*}
    It follows that $\|z\|_{\mcH_0^1[\epsilon t]}^2\lesssim\sum_{k=1}^{n}z_k^2-\sum_{k=n+1}^{\infty}\lambda_k(\epsilon t)z_k^2=\|X_1\|_2^2-\sum_{k=n+1}^{\infty}\lambda_k(\epsilon t)z_k^2.$ Thus, by \eqref{3rd_Lyapunov_functional},
    \begin{equation}
    \label{z_H_lessim_V}
        \|z\|_{\mcH_0^1[\epsilon t]}^2\lesssim V(t,z).
    \end{equation}
    Combining \eqref{V_lessim_z_H} and \eqref{z_H_lessim_V} yields \eqref{V_sim_z_H_1}. 
    
    We now prove \eqref{V_lessim_A_X_1}. Let $z\in\mcD(\mcA(\epsilon t))$. Since
    \begin{equation*}
        -\lambda_k(\epsilon t)\geq\eta,\quad k>n,\,\,\, \text{we have}\,\,\, -\lambda_k(\epsilon t)\leq\frac1\eta\lambda_k(\epsilon t)^2,\quad k>n.
    \end{equation*}
    Therefore,
    \begin{equation*}
        -\sum_{k=n+1}^{\infty}\lambda_k(\epsilon t) z_k^2 \leq \frac1\eta\sum_{k=n+1}^{\infty}\lambda_k(\epsilon t)^2 z_k^2\leq \frac1\eta \sum_{k=1}^{\infty} \lambda_k(\epsilon t)^2z_k^2=\frac1\eta\|\mcA(\epsilon t)z\|_{\mcH[\epsilon t]}^2.
    \end{equation*}
    Hence, \eqref{3rd_Lyapunov_functional} yields $V(t,z)\lesssim\|X_1\|_2^2+\|\mcA(\epsilon t) z\|_{\mcH[\epsilon t]}^2,$ which proves \eqref{V_lessim_A_X_1}.
\end{proof}
    Since the solution of the closed-loop system~\eqref{Close_loop_z_eqn} satisfies~\eqref{Regularity_Thm_Soln_Trans_sys}, we have $z(t,\cdot)\in H_0^1(0,1)$ for every $t\in[0,1/\epsilon],$ and $z(t,\cdot)\in\mcD$ for almost every $t\in(0,1/\epsilon).$ Consequently, \eqref{V_sim_z_H_1} holds for every $t\in[0,1/\epsilon]$, whereas \eqref{V_lessim_A_X_1} holds for almost every $t\in(0,1/\epsilon)$.

    Let $z(t,\cdot)$ denote the solution to the closed-loop system \eqref{Close_loop_z_eqn} on $[0,1/\epsilon]$. For simplicity, set
    \begin{align}
    \label{V_1_definition}
        V_1(t):={}&V(t,z(t,\cdot)) \notag\\
        ={}&c\big\langle P(\epsilon t) X_1(t), X_1(t)\big\rangle_2 +\frac12\int_0^1\rho(\epsilon t,x) \Big(|z_x(t,x)|^2-q(\epsilon t,x)|z(t,x)|^2\Big)\,dx .
    \end{align}
    Thus, $V_1(t)$ is well-defined for every $t\in[0,1/\epsilon]$. Moreover, since $z(t,\cdot)\in\mcD(\mcA(\epsilon t))$ for almost every $t\in(0,1/\epsilon)$, integration by parts gives
    \begin{equation}
    \label{V_1_operator_representation}
        V_1(t)=c\big\langle P(\epsilon t) X_1(t),X_1(t)\big\rangle_2-\frac12 \big\langle z(t,\cdot),\mcA(\epsilon t)z(t,\cdot) \big\rangle_{\mcH[\epsilon t]}
    \end{equation}
    for almost every $t\in(0,1/\epsilon)$.

  \begin{remark}\label{rem:Lyapunov_energy_level}
    The Lyapunov functional \eqref{Lyapunov_Functional} is defined on the
     quadratic-form domain $\mcH_0^1[\epsilon t]$. Hence,
    $z\in C([0,T];H_0^1(0,1))$ is sufficient for its pointwise definition,
   whereas $z(t,\cdot)\in \mcD(A(\epsilon t))$ is needed only for a.e.\ $t$
   when using the corresponding operator representation and deriving the
   differential identity. In contrast, in \cite{Shirshendu20}, the stronger
   operator-domain regularity of the closed-loop solution is established in
   Theorem~3.19 before introducing the Lyapunov functional.
\end{remark}
    
    We next show that $V_1\in W^{1,1}(0,1/\epsilon)$ and establish the differential inequality satisfied by $V_1$.
\begin{lem}
\label{Diff_ineq_lem}
    Let $z(t,\cdot)$ denote the solution to the closed-loop system \eqref{Close_loop_z_eqn} on $[0,1/\epsilon],$ and let $V_1$ be defined by \eqref{V_1_definition}. Then $V_1\in W^{1,1}(0,1/\epsilon),$ and hence $V_1$ is absolutely continuous on $[0,1/\epsilon]$. Moreover, there exist $\epsilon_0>0$ and $\widehat{p}>0,$ independent of $\epsilon,$ such that, for every $0<\epsilon \leq \epsilon_0,$ the function $V_1$ satisfies
    \begin{align}
    \label{d_dt_V1_t_lessim_1}
        \frac{d}{dt}V_1(t)+ \widehat{p} V_1(t)\lesssim  \Big(V_1(t)^2&+V_1(t)^{{\delta+1}} + V_1(t)^{{\delta}} + V_1(t)^{2\delta+1}\Big)+\epsilon^2   \notag \\
        & + \Big(V_1(t)^{\frac{3}{2}}+V_1(t)^{\frac{\delta+2}{2}} + V_1(t)^{\frac{\delta+1}{2}}  \Big),
    \end{align}  
    for almost every $t\in(0,1/\epsilon)$.
\end{lem}
\begin{proof}
    Set
    \begin{equation*}
        \mathcal{E}(t):=\frac{1}{2}\int_0^1 \rho(\epsilon t,x)
        \Big(|z_x(t,x)|^2-q(\epsilon t,x)|z(t,x)|^2\Big)\,dx.
    \end{equation*}
    Then, by \eqref{V_1_definition}, 
    \begin{equation*}
        V_1(t)=c\big\langle P(\epsilon t)X_1(t),X_1(t)\big\rangle_2+\mathcal{E}(t)
    \end{equation*}
    and by \eqref{Regularity_Thm_Soln_Trans_sys}, we have 
    \begin{equation*}
        z\in L^2(0,1/\epsilon;\mcD)\cap H^1(0,1/\epsilon;L^2(0,1)).
    \end{equation*}
    Moreover, since the eigenfunctions $\phi_k(\tau,\cdot)$ and the weight $\rho(\tau,\cdot)$ depend $C^1$-smoothly on $\tau$, it follows from
    \begin{equation*}
        z_k(t)=\big\langle z(t,\cdot),\phi_k(\epsilon t,\cdot)\big\rangle_{\mcH[\epsilon t]},\qquad k=1,\ldots,n,\,\,\text{that}\,\,\,  X_1\in H^1(0,1/\epsilon;\mathbb{R}^n).
    \end{equation*}
    We next compute the derivative of $\mathcal{E}$. To justify the computation under the above regularity, we use a standard time-mollification argument. Indeed, after extending $z$ in time and denoting by $z^\nu$ its time mollification, we have
    \begin{equation*}
        z^\nu\to z\quad\text{in } L^2(0,1/\epsilon;\mcD),\qquad z_t^\nu\to z_t\quad\text{in }L^2(0,1/\epsilon;L^2(0,1)).
    \end{equation*}
    For $z^\nu$, differentiation with respect to time and integration by parts in the spatial variable yield
    \begin{align*}
        \frac{d}{dt}\mathcal{E}^\nu(t)
        ={}&-\big\langle z_t^\nu(t,\cdot), \mcA(\epsilon t)z^\nu(t,\cdot)\big\rangle_{\mcH[\epsilon t]}\\
        &+\frac{\epsilon}{2}\int_0^1\rho_\tau(\epsilon t,x) |z_x^\nu(t,x)|^2\,dx-\frac{\epsilon}{2} \int_0^1(\rho q)_\tau(\epsilon t,x)|z^\nu(t,x)|^2\,dx.
    \end{align*}
    Passing to the limit as $\nu\to0$, we obtain $\mathcal{E}\in W^{1,1}(0,1/\epsilon),$ and, for almost every $t\in(0,1/\epsilon)$,
    \begin{align}
    \label{d_dt_energy_part}
        \frac{d}{dt}\mathcal{E}(t)={}&-\big\langle z_t(t,\cdot),\mcA(\epsilon t)z(t,\cdot)
        \big\rangle_{\mcH[\epsilon t]}\notag\\
        &+\frac{\epsilon}{2}\int_0^1\rho_\tau(\epsilon t,x) |z_x(t,x)|^2\,dx-\frac{\epsilon}{2}\int_0^1(\rho q)_\tau(\epsilon t,x)|z(t,x)|^2\,dx.
    \end{align}
    Since $P\in C^1([0,1];\mathbb{R}^{n\times n})$ and $X_1\in H^1(0,1/\epsilon;\mathbb{R}^n),$ it follows that 
    \begin{equation*}
        t\mapsto c\big\langle P(\epsilon t)X_1(t),X_1(t)\big\rangle_2\,\,\text{ belongs to}\,\,\, W^{1,1}(0,1/\epsilon).
    \end{equation*}
    Consequently, $ V_1\in W^{1,1}(0,1/\epsilon),$ and hence $V_1$ is absolutely continuous on $[0,1/\epsilon]$.
    
    Combining \eqref{d_dt_energy_part} with the definition of $V_1$, we obtain, for almost every $t\in(0,1/\epsilon)$,
    \begin{align}
    \label{d_dt_V1_t}
        \frac{d}{dt}V_1(t)={}&c\Big(2\big\langle P(\epsilon t)X_1(t),X_1'(t)\big\rangle_2+\epsilon\big\langle P' (\epsilon t)X_1(t),X_1(t)\big\rangle_2\Big)\notag\\
        &-\big\langle z_t(t,\cdot),\mcA(\epsilon t)z(t,\cdot)
        \big\rangle_{\mcH[\epsilon t]}\notag\\
        &+\frac{\epsilon}{2}\int_0^1\rho_\tau(\epsilon t,x) |z_x(t,x)|^2\,dx-\frac{\epsilon}{2}\int_0^1(\rho q)_\tau (\epsilon t,x)|z(t,x)|^2\,dx.
    \end{align}
    Using \eqref{Close_loop_z_eqn}, we have
    \begin{align*}
        \big\langle z_t(t,\cdot),\mcA(\epsilon t)z(t,\cdot)\big\rangle_{\mcH[\epsilon t]}={}&\|\mcA(\epsilon t)z(t,\cdot)\|_{\mcH[\epsilon t]}^2+\big\langle\xi(\epsilon t)z(t,\cdot),\mcA(\epsilon t) z(t,\cdot)\big\rangle_{\mcH[\epsilon t]}\\
        &+\big\langle\mcR(\epsilon,t,\cdot),\mcA (\epsilon t)  z(t,\cdot)\big\rangle_{\mcH[\epsilon t]}.
    \end{align*}
    From \eqref{ODE_X_1}, \eqref{Feedback_law}, and
    \eqref{Identity_Lyapunov}, we infer
    \begin{equation}
    \label{P_X1_X1_prime}
        2\big\langle P(\epsilon t)X_1(t),X_1'(t)\big\rangle_2
        =-\|X_1(t)\|_2^2+2\big\langle P(\epsilon t)X_1(t),
        \mcR_1(\epsilon,t)\big\rangle_2.
    \end{equation}
    Substituting the above identities into \eqref{d_dt_V1_t}, we obtain
    \begin{align}
    \label{d_dt_V1_t_equal}
        \frac{d}{dt}V_1(t)={}&-c\|X_1(t)\|_2^2-\|\mcA(\epsilon t) z(t,\cdot)\|_{\mcH[\epsilon t]}^2\notag\\
        &-\big\langle\xi(\epsilon t)z(t,\cdot),\mcA(\epsilon t)z(t,\cdot)\big\rangle_{\mcH[\epsilon t]}\notag\\
        &-\big\langle\mcR(\epsilon,t,\cdot),\mcA(\epsilon t) z(t,\cdot)\big\rangle_{\mcH[\epsilon t]}\notag\\
        &+2c\big\langle P(\epsilon t)X_1(t), \mcR_1(\epsilon,t)\big\rangle_2+c\epsilon\big\langle P'(\epsilon t)X_1(t),X_1(t)\big\rangle_2\notag\\
        &+\frac{\epsilon}{2}\int_0^1\rho_\tau(\epsilon t,x) |z_x(t,x)|^2\,dx-\frac{\epsilon}{2}\int_0^1(\rho q)_\tau(\epsilon t,x)|z(t,x)|^2\,dx.
    \end{align}
    We now estimate the terms on the right-hand side of
    \eqref{d_dt_V1_t_equal} one by one. By the definition of $\xi$ and the orthogonality of $\Pi(\epsilon t)$, we have
    \begin{equation*}
        \|\xi(\epsilon t)z(t,\cdot)\|_{\mcH[\epsilon t]}\leq\kappa\|\Pi(\epsilon t)z(t,\cdot)\|_{\mcH[\epsilon t]}=\kappa\|X_1(t)\|_2.
    \end{equation*}
    Hence, applying the Cauchy-Schwarz and Young inequalities, we obtain
    \begin{align}
    \label{xi_Az_estimate}
        \Big|\big\langle\xi(\epsilon t)z(t,\cdot), \mcA(\epsilon t)z(t,\cdot)\big\rangle_{\mcH[\epsilon t]}\Big|&\leq M_1\|X_1(t)\|_2^2+\frac18
        \|\mcA(\epsilon t)z(t,\cdot)\|_{\mcH[\epsilon t]}^2,
    \end{align}
    for some $M_1>0$.

    By \eqref{Defn_R} and the Cauchy-Schwarz inequality, we have
    \begin{align}
    \label{R_A_z_estimate}
        \Big|\big\langle\mcR(\epsilon,t,\cdot),\mcA(\epsilon t)z(t,\cdot)\big\rangle_{\mcH[\epsilon t]}\Big|&\leq
        \|\mcA(\epsilon t)z(t,\cdot)\|_{\mcH[\epsilon t]}
        \Bigg[\alpha\bigg\|\bigg(\delta\overline{u}^{\delta-1}z
        +\sum_{j=2}^{\delta}\binom{\delta}{j}z^j \overline{u}^{\delta-j} \bigg) z_x\bigg\|_{\mcH[\epsilon t]}\notag\\
        &\qquad+\alpha\bigg\|\bigg( \sum_{j=2}^{\delta}\binom{\delta}{j} z^j\overline{u}^{\delta-j}\bigg)\overline{u}_x \bigg\|_{\mcH[\epsilon t]}\notag\\
        &\qquad+\big\|F(z+\overline{u})-F(\overline{u})
        -F'(\overline{u})z\big\|_{\mcH[\epsilon t]}\Bigg]\notag\\
        &\qquad+\epsilon\Big| \big\langle \overline{u}_\tau(\epsilon t),\mcA(\epsilon t) z(t,\cdot)\big\rangle_{\mcH[\epsilon t]}\Big|.
    \end{align}
    We further estimate each term on the right-hand side of
    \eqref{R_A_z_estimate} separately. Using \eqref{Norm_Equivalence}, the continuous embedding $\mcH_0^1[\epsilon t]\hookrightarrow L^\infty(0,1)$ (uniform with respect to $t\in[0,1/\epsilon]$), and the uniform boundedness of $\overline{u}$, we obtain
    \begin{align*}
        \bigg\|\bigg(\delta\overline{u}^{\delta-1}z+  \sum_{j=2}^{\delta}\binom{\delta}{j}z^j\overline{u}^{\delta-j}
        \bigg)z_x\bigg\|_{\mcH[\epsilon t]}&\lesssim\bigg( \|z\|_{L^\infty}+\sum_{j=2}^{\delta}\|z\|_{L^\infty}^{\,j} \bigg)\|z_x\|_{\mcH[\epsilon t]}\\
        &\lesssim\|z\|_{\mcH_0^1[\epsilon t]}^2\bigg(1+\sum_{j=2}^{\delta}\|z\|_{\mcH_0^1[\epsilon t]}^{\,j-1}\bigg)\\
        &\lesssim\|z\|_{\mcH_0^1[\epsilon t]}^2\Big( 1+\|z\|_{\mcH_0^1[\epsilon t]}^{\delta-1}\Big)\\
        &\lesssim V_1(t)\Big(1+V_1(t)^{\frac{\delta-1}{2}}\Big).
    \end{align*}
    Using the uniform boundedness of $\overline{u}_x$, we have
    \begin{align*}
        \bigg\|\bigg(\sum_{j=2}^{\delta}\binom{\delta}{j}z^j \overline{u}^{\delta-j}\bigg)\overline{u}_x \bigg\|_{\mcH[\epsilon t]}&\lesssim \sum_{j=2}^{\delta}\|z^j\|_{\mcH[\epsilon t]}\lesssim
        \sum_{j=2}^{\delta} \|z\|_{L^\infty}^{j-1}\|z\|_{\mcH[\epsilon t]}\\
        &\lesssim\sum_{j=2}^{\delta}\|z\|_{\mcH_0^1[\epsilon t]}^j
        \lesssim\|z\|_{\mcH_0^1[\epsilon t]}^2
        \Big(1+\|z\|_{\mcH_0^1[\epsilon t]}^{\delta-2}\Big)\\
        &\lesssim V_1(t)\Big(1+V_1(t)^{\frac{\delta-2}{2}}\Big).
    \end{align*}
    Further, applying \eqref{Taylor_remainder_bound} pointwise with
    $\overline{v}=\overline{u}(\epsilon t,x)$ and
    $\overline{h}=z(t,x)$, we obtain
    \begin{align*}
        \big\|F(z+\overline{u})-F(\overline{u})-F'(\overline{u})z
        \big\|_{\mcH[\epsilon t]}&\lesssim\|z(t,\cdot)\|_\infty^2
        +\|z(t,\cdot)\|_\infty^{2\delta+1}\\
        &\lesssim V_1(t)+V_1(t)^{\frac{2\delta+1}{2}}.
    \end{align*}
    Using the uniform boundedness of $\overline{u}_\tau$, we obtain
    \begin{equation*}
        \epsilon \Big|\big\langle\overline{u}_\tau(\epsilon t),\mcA(\epsilon t)z(t,\cdot) \big\rangle_{\mcH[\epsilon t]}\Big|\lesssim\epsilon
        \|\mcA(\epsilon t)z(t,\cdot)\|_{\mcH[\epsilon t]}.
    \end{equation*}
    Thus,
    \begin{align}
    \label{R_A_z_estimate2}
        \Big|\big\langle\mcR(\epsilon,t,\cdot),
        \mcA(\epsilon t)z(t,\cdot)\big\rangle_{\mcH[\epsilon t]}\Big|
        \lesssim{}&\|\mcA(\epsilon t)z(t,\cdot)\|_{\mcH[\epsilon t]}
        \notag\\
        &\times\Bigg[V_1(t)\Big(1+V_1(t)^{\frac{\delta-1}{2}}
        +V_1(t)^{\frac{\delta-2}{2}}+V_1(t)^{\frac{2\delta-1}{2}}\Big)
        +\epsilon\Bigg].
    \end{align}
    By \eqref{Defn_of_R1}, Theorem~\ref{Third}, Remark~\ref{rem:p_q_u_regularity}, and the finite-dimensionality of $\mathbf{V}_n^\tau$, we have
    \begin{equation*}
        \|\mcR_1(\epsilon,t)\|_2
        \lesssim
        \|\mcR(\epsilon,t,\cdot)\|_{\mcH[\epsilon t]}
        +\epsilon\|z(t,\cdot)\|_{\mcH[\epsilon t]}.
    \end{equation*}
    Moreover, $\|X_1(t)\|_2^2\leq\|z(t,\cdot)\|_{\mcH[\epsilon t]}^2\lesssim V_1(t)$. Hence,
    \begin{align}
    \label{P_R1_estimate}
        \Big|2c\big\langle P(\epsilon t)X_1(t),
        \mcR_1(\epsilon,t)\big\rangle_2\Big| &\lesssim c\|X_1(t)\|_2\|\mcR_1(\epsilon,t)\|_2\notag\\
        &\lesssim c\sqrt{V_1(t)}\Bigg[V_1(t)\Big( 1+V_1(t)^{\frac{\delta-1}{2}}+V_1(t)^{\frac{\delta-2}{2}}
        +V_1(t)^{\frac{2\delta-1}{2}}\Big)\notag\\
        &\hspace{4.8cm}+\epsilon\big(1+\sqrt{V_1(t)}\big)\Bigg].
    \end{align}
    By Lemma~\ref{Pole_placement}, there exists a constant $C_2>0$
    such that $\|P'(\epsilon t)\|\leq C_2$ for all $t\in[0,1/\epsilon].$ Hence,
    \begin{equation}
    \label{P_prime__X1_estimate}
        \Big|c\epsilon\big\langle P'(\epsilon t) X_1(t),X_1(t)\big\rangle_2\Big|\lesssim c\epsilon\|X_1(t)\|_2^2\lesssim c\epsilon V_1(t).
    \end{equation}
    Finally, since $\rho_\tau$ and $(\rho q)_\tau$ are uniformly
    bounded, by \eqref{V_sim_z_H_1}, we obtain
    \begin{equation}
    \label{time_coefficient_estimate}
        \Bigg|\frac{\epsilon}{2}\int_0^1\rho_\tau(\epsilon t,x) |z_x(t,x)|^2\,dx-\frac{\epsilon}{2}\int_0^1(\rho q)_\tau (\epsilon t,x)|z(t,x)|^2\,dx\Bigg|\lesssim\epsilon \|z(t,\cdot)\|_{\mcH_0^1[\epsilon t]}^2\lesssim\epsilon V_1(t).
    \end{equation}
    Using \eqref{xi_Az_estimate}, \eqref{R_A_z_estimate2},
    \eqref{P_R1_estimate}, \eqref{P_prime__X1_estimate}, and
    \eqref{time_coefficient_estimate} in \eqref{d_dt_V1_t_equal},
    we obtain
    \begin{align}
    \label{d_dt_V1_t_lessim}
        \frac{d}{dt}V_1(t)&+\big(c-M_1\big)\|X_1(t)\|_2^2+\frac78
        \|\mcA(\epsilon t)z(t,\cdot)\|_{\mcH[\epsilon t]}^2\notag\\
        &\leq C\|\mcA(\epsilon t)z(t,\cdot)\|_{\mcH[\epsilon t]}
        \Big(V_1(t)+V_1(t)^{\frac{\delta+1}{2}}+V_1(t)^{\frac{\delta}{2}}+V_1(t)^{\frac{2\delta+1}{2}}\Big)\notag\\
        &\qquad+C\epsilon\|\mcA(\epsilon t)z(t,\cdot)\|_{\mcH[\epsilon t]}\notag\\
        &\qquad+Cc\Big(\epsilon\sqrt{V_1(t)}+V_1(t)^{\frac32}
        +V_1(t)^{\frac{\delta+2}{2}}+V_1(t)^{\frac{\delta+1}{2}}
        +V_1(t)^{\delta+1}\Big)\notag\\
        &\qquad+2C(c+1)\epsilon V_1(t),
    \end{align}
    where $C>0$ is independent of $\epsilon$. By Young's inequality, for any $\theta>0$, we have
    \begin{align*}
        \epsilon\|\mcA(\epsilon t)z(t,\cdot)\|_{\mcH[\epsilon t]}
        &\leq C_\theta\epsilon^2+\frac{\theta}{2}\|\mcA(\epsilon t) z(t,\cdot)\|_{\mcH[\epsilon t]}^2,\\
        \|\mcA(\epsilon t)z(t,\cdot)\|_{\mcH[\epsilon t]}\Big(& V_1(t)
        +V_1(t)^{\frac{\delta+1}{2}}+V_1(t)^{\frac{\delta}{2}} +V_1(t)^{\frac{2\delta+1}{2}}\Big)  \\
        &\leq\frac{\theta}{2}\|\mcA(\epsilon t)z(t,\cdot) \|_{\mcH[\epsilon t]}^2+4C_\theta\Big(V_1(t)^2 +V_1(t)^{\delta+1}+V_1(t)^\delta+V_1(t)^{2\delta+1}\Big),
    \end{align*}
    and $\epsilon\sqrt{V_1(t)}\leq\frac{\theta}{2}V_1(t)+C_\theta\epsilon^2.$ Employing the above estimates in \eqref{d_dt_V1_t_lessim} yields
    \begin{align}
    \label{d_dt_V1_t_lessim_}
        \frac{d}{dt}V_1(t)&+\big(c-M_1\big)\|X_1(t)\|_2^2
        +\Big(\frac78-C\theta\Big)\|\mcA(\epsilon t) z(t,\cdot)\|_{\mcH[\epsilon t]}^2\notag\\
        &\leq 4CC_\theta\Big(V_1(t)^2+V_1(t)^{\delta+1}+V_1(t)^\delta
        +V_1(t)^{2\delta+1}\Big)+CC_\theta(1+c)\epsilon^2\notag\\
        &\qquad+\frac12Cc\theta V_1(t)+Cc\Big(V_1(t)^{\frac32}
        +V_1(t)^{\frac{\delta+2}{2}}+V_1(t)^{\frac{\delta+1}{2}}
        +V_1(t)^{\delta+1}\Big)\notag\\
        &\qquad+2C(c+1)\epsilon V_1(t).
    \end{align}
    Choose $c>0$ sufficiently large so that $ c-M_1\geq1.$
    By \eqref{V_lessim_A_X_1}, there exists a constant $C_0>0$,
    independent of $\epsilon$, such that, for almost every
    $t\in(0,1/\epsilon)$,
    \begin{equation*}
        V_1(t)\leq C_0\Big(\|X_1(t)\|_2^2+\|\mcA(\epsilon t) z(t,\cdot)\|_{\mcH[\epsilon t]}^2\Big).
    \end{equation*}
    Choose $\theta>0$ sufficiently small so that $ \frac78-C\theta\geq\frac12.$ It follows that
    \begin{equation*}
        \big(c-M_1\big)\|X_1(t)\|_2^2+\Big(\frac78-C\theta\Big)
        \|\mcA(\epsilon t)z(t,\cdot)\|_{\mcH[\epsilon t]}^2\geq
        \frac{1}{2C_0}V_1(t).
    \end{equation*}
    Set $ p_0:=\frac{1}{2C_0}>0.$ Decreasing $\theta>0$, if necessary, we may further assume that $\frac12Cc\theta\leq\frac{p_0}{4}.$ Next, choose $\epsilon_0>0$ sufficiently small such that $ 2C(c+1)\epsilon_0\leq\frac{p_0}{4}.$ Therefore, for every $0<\epsilon\leq\epsilon_0$,
    \begin{equation*}
        p_0-\frac12Cc\theta-2C(c+1)\epsilon
        \geq\frac{p_0}{2}.\,\,\,\text{Setting}\,\,\, \widehat p:=\frac{p_0}{2}>0,
    \end{equation*}
    we deduce from \eqref{d_dt_V1_t_lessim_} that
    \begin{align*}
        \frac{d}{dt}V_1(t)+\widehat p\,V_1(t)\lesssim{}&V_1(t)^2+ V_1(t)^{\delta+1}+V_1(t)^\delta+V_1(t)^{2\delta+1} +\epsilon^2 \\
        &+V_1(t)^{\frac32}+V_1(t)^{\frac{\delta+2}{2}}+V_1(t)^{\frac{\delta+1}{2}},
    \end{align*}
    for almost every $t\in(0,1/\epsilon)$. Hence,
    \eqref{d_dt_V1_t_lessim_1} follows. This completes the proof.
\end{proof}

\begin{lem}
\label{smallness_0f_z}
    There exist $\epsilon_0^*\in(0,\epsilon_0]$ and a constant $C_3>0$ such that, for every $\epsilon\in(0,\epsilon_0^*]$, if $z(t,\cdot)$ is the solution of \eqref{Close_loop_z_eqn}, equivalently \eqref{z_eqn_in_operator_form} with feedback \eqref{Feedback}, then
    \begin{equation}
    \label{z_1_by_epsilon}
        \|z(t)\|_{H_0^1(0,1)}\le C_3\,\epsilon,\qquad \forall\, t\in\left[0,\frac{1}{\epsilon}\right],
    \end{equation}
    where $C_3$ is independent of $\epsilon$.
\end{lem}
\begin{proof}
    From Lemma~\ref{Diff_ineq_lem}, there exists a constant $C_4>0$ such that, for almost every $t\in(0,1/\epsilon)$,
    \begin{align}
    \label{d_dt_V1_t_leq_C}
        \frac{d}{dt}V_1(t) + \widehat{p} V_1(t) \leq C_4 \Big(V_1(t)^2 &+V_1(t)^{{\delta+1}} 
        + V_1(t)^{{\delta}} +V_1(t)^{\frac{3}{2}}\notag \\
        &+V_1(t)^{\frac{\delta+2}{2}} + V_1(t)^{\frac{\delta+1}{2}} +V_1(t)^{2\delta+1} + \epsilon^2 \Big)
    \end{align}
    Let $0<\nu<\min\left\{1,\left(\frac{\widehat p}{14C_4}\right)^2 \right\}$ be a constant. Set $\epsilon_0^*:=\min\left\{\epsilon_0,\nu^{3/4}\right\}.$ Fix $\epsilon\in(0,\epsilon_0^*]$. Using $V_1(0)=0$ and the continuity of the map $t \mapsto V_1(t)$, there exists $t_0\in(0,1/\epsilon]$ such that $V_1(t) < \nu$ for all $t \in [0, t_0)$. We claim that 
    \begin{equation}
    \label{V_1_less_nu}
        V_1(t) < \nu \quad \text{for all } t \in [0,1/\epsilon]
    \end{equation}
    Indeed, if not, since $t\mapsto V_1(t)$ is continuous, there exists $T^*\in (0,1/\epsilon]$ such that
    \begin{equation}
    \label{There_exists_T*_positive}
        V_1(t)<\nu,\quad \forall \ 0 \leq t < T^*,\ \text{and}\ V_1(T^*) = \nu.
    \end{equation}
    In view of \eqref{There_exists_T*_positive} and since $\delta\geq2$, for every $t\in[0,T^*),$
    \begin{equation*}
        \Big(V_1(t)^2+V_1(t)^{{\delta+1}} + V_1(t)^{{\delta}} + V_1(t)^{\frac{3}{2}}+V_1(t)^{\frac{\delta+2}{2}} + V_1(t)^{\frac{\delta+1}{2}} + V_1(t)^{2\delta+1}\Big)\leq 7\nu^{1/2}V_1(t).
    \end{equation*}
    Substituting the above estimate into \eqref{d_dt_V1_t_leq_C}, we obtain, for almost every $t\in(0,T^*)$, we have $\frac{d}{dt}V_1(t) + \widehat{p} V_1(t) \leq 7 C_{4}\nu^{1/2}V_1(t)+C_{4}\epsilon^2$ and it can be rewritten as
    \begin{equation*}
        \frac{d}{dt}V_1(t) + \big(\widehat{p}- 7C_{4}\nu^{1/2}\big) V_1(t) \leq C_{4} \epsilon^2.
    \end{equation*}
    As $0<\nu<\min \big\{1, (\frac{\widehat{p}}{14C_{4}})^2\big\}$, it follows that $7C_{4}\nu^{1/2} < \frac{\widehat{p}}{2}$. Then, for almost every $t\in(0,T^*),$ we get
    \begin{equation}
    \label{Simplified_diff_ineq}
        \frac{d}{dt}V_1(t) + \frac{\widehat{p}}{2} V_1(t) \leq C_{4} \epsilon^2.
    \end{equation}
    Integrating \eqref{Simplified_diff_ineq} over $[0,T^*]$ and using $V_1(0)=0$ yields
    \begin{equation*}
        V_1(T^*)\leq \frac{2C_{4}}{\widehat{p}}\epsilon^2\left(1-e^{-\frac{\widehat{p}}{2}T^*}\right) < \frac{2C_{4}}{\widehat{p}}\epsilon^2<\frac{\epsilon^2}{\sqrt{\nu}} \leq \nu.
    \end{equation*}
    Thus, we have $V_1(T^*)<\nu,$ which contradicts \eqref{There_exists_T*_positive}. Therefore \eqref{V_1_less_nu} holds. Consequently, \eqref{Simplified_diff_ineq} holds for almost every $t\in(0,1/\epsilon).$ Integrating it over $[0,t]$, we obtain
    \begin{equation}
    \label{V_1_less_epsilon2}
        V_1(t) \le \frac{1}{\sqrt{\nu}}\epsilon^2, \quad \forall\, t \in \left[0, \frac{1}{\epsilon}\right].
    \end{equation}
    Thus, \eqref{z_1_by_epsilon} follows from \eqref{V_sim_z_H_1}, \eqref{Norm_Equivalence}, and \eqref{V_1_less_epsilon2}.
\end{proof}
    We now prove Theorem~\ref{Main_thm}.
\begin{proof}[Proof of Theorem~\ref{Main_thm}]
    Let $0<\epsilon\le\epsilon_0^*$, and let $z$ be the solution of the closed-loop system~\eqref{Close_loop_z_eqn}, or equivalently of \eqref{Close_loop_z_eqn_in_Xi} with feedback~\eqref{Feedback}. By Lemma~\ref{smallness_0f_z}, we have
    \begin{equation*}
        \sup_{0\le t\le 1/\epsilon} \|z(t)\|_{H_0^1(0,1)} \le C_3\,\epsilon.
    \end{equation*}
    In view of \eqref{Intro_z_h} and \eqref{Feedback}, define the control for \eqref{cntrl_eqn} by
    \begin{equation*}
        f(t,\cdot) = \overline f(\epsilon t,\cdot)  - \kappa\Pi(\epsilon t)z(t,\cdot), \qquad t\in[0,1/\epsilon].
    \end{equation*}
    Then $f\in L^2((0, 1/\epsilon)\times(0,1)),$ and the function
    \begin{equation*}
        u(t,\cdot)=z(t,\cdot)+\overline u(\epsilon t,\cdot)
    \end{equation*}
    solves~\eqref{cntrl_eqn} with this control. Hence~\eqref{Uniform_tracking} follows immediately.
    
    Let $\widehat\delta>0$ be given, as in the statement of Theorem~\ref{Main_thm}. 
    Set $T_{\widehat\delta} := \max\left\{ \frac{C_3}{\widehat\delta},\ \frac{1}{\epsilon_0^*} \right\}.$ Let $T \ge T_{\widehat\delta}$ be arbitrary, and set $\epsilon := \dfrac1T$. Then
    \begin{equation*}
        0 < \epsilon \le \epsilon_0^* \qquad \text{and} \qquad C_3\,\epsilon = \frac{C_3}{T} \le \frac{C_3}{T_{\widehat\delta}} \le \widehat\delta.
    \end{equation*}
     Using $\overline{u}(\epsilon T) = \overline{u}(1) = u_1$, we obtain 
     \begin{equation*}
        \|u(T) - u_1\|_{H_0^1(0,1)}  \le \widehat\delta. 
     \end{equation*}
     This completes the proof.
\end{proof}

\begin{remark}
The control defined in the proof of
Theorem~\ref{Main_thm} enjoys higher time regularity. Indeed, by
the regularity of $\overline f$, $\Pi$, and the solution $z$ of the closed-loop system \eqref{Close_loop_z_eqn}, we obtain,
\begin{equation*}
    f\in H^1(0,T;L^2(0,1)).
\end{equation*}
Consequently, since multiplication by
$\chi_\omega$ is bounded on $L^2(0,1)$,
the actual control term satisfies
\begin{equation*}
 \chi_\omega f\in H^1(0,T;L^2(0,1)).
\end{equation*}
\end{remark}

\begin{remark}
    The time $T_{\widehat{\delta}}$ in
    Theorem~\ref{Main_thm} generally depends on the steady
    states $u_0,u_1\in\mathcal S$, through the constants
    arising in the stabilization estimates along the path
    connecting them.
\end{remark}
\section{Concluding Remarks}\label{CR}
    This section is devoted to some additional remarks on the controllability of the generalized Burgers-Huxley equation. We conclude by highlighting some open problems that naturally arise from the present work.
\subsection{Approximate Controllability under Neumann Data}
\mbox{}
    It is natural to ask whether the approximate controllability result established in Theorem~\ref{Main_thm} remains valid when the homogeneous Dirichlet boundary conditions in \eqref{cntrl_eqn} are replaced by homogeneous Neumann boundary conditions. In this case, the equation itself is unchanged, but the boundary conditions become $u_x(t,0)=u_x(t,1)=0$ for $t\in(0,T)$, together with the same initial condition $u(0,x)=u_0(x)$ for $x\in(0,1)$.

    The proof can be adapted to this setting by replacing the operator $\mathcal{A}(\tau)$ with its Neumann realization $\mathcal{A}_N(\tau)$, whose domain consists of functions in $H^2(0,1)$ satisfying the Neumann boundary conditions rather than the Dirichlet ones, and which acts by the same differential expression as before. The corresponding estimates for the Lyapunov functional continue to hold, with the relevant norm equivalence now taken with respect to $H^1(0,1)$ rather than $H_0^1(0,1)$. Thus, the arguments developed in this paper suggest that the approximate controllability result extends naturally to the Neumann setting as well.

\subsection{Connected Components of \texorpdfstring{$\mathcal{S}$}{S}}
    It remains unknown how many connected components the set of steady states of the GBHE possesses and how this depends on the parameters. For the heat equation, sufficient conditions on the nonlinearity ensuring the connectedness of the set of steady states are available (see Proposition~3.1 in~\cite{Coron2004}). Unfortunately, this result cannot be applied to the GBHE.

\subsection{A Necessary Condition for Approximate Reachability}

    Under the assumptions of Theorem~\ref{thm1},
    fix $u_0\in L^2(0,1)$, and let $M$ be as in its proof.
    If a target $u_1\in L^2(0,1)$ can be approximated
    from $u_0$ by admissible controls, with terminal times
    that may vary, then \eqref{estimate_u(T)} implies
    \begin{equation*}
        \|u_1\|_{L^2\left(\frac b4,\frac b2\right)}
        \leq
        \left(\frac{b^2}{8}\right)^{-N/2}M^{1/2}.
    \end{equation*}
    This is only a necessary condition. The estimate does
    not determine which targets satisfying it are
    approximately reachable.
\subsection{Approximate Controllability Between Restricted Classes of States}

    Theorem~\ref{thm1} shows that approximate controllability
    between arbitrary states fails even when the terminal
    time is allowed to depend on the data and the prescribed
    accuracy. In particular, starting from $u_0=0$,
    sufficiently large smooth targets supported in
    $(b/4,b/2)$ cannot be approximated at any time.
    On the other hand, Theorem~\ref{Main_thm} gives
    a positive result between steady states lying in the
    same connected component of $\mcS$. Our analysis
    does not settle controllability between steady states
    in different connected components.
\section*{Appendix}\label{Appendix}

    This appendix collects the proofs of 
Lemmas~\ref{lem:weak_time_continuity}, \ref{lem:left_ivp}, 
and~\ref{lem:right_ivp}, and of 
Theorem~\ref{Thm_Soln_Trans_sys}. 

\begin{proof}[Proof of Lemma~\ref{lem:weak_time_continuity}]
    Set
    \begin{equation*}
        \widetilde{p}:=2(\delta+1),\qquad {\widetilde{p}}':=\frac{2(\delta+1)}{2\delta+1},\qquad \widetilde{q}:=\frac{2\delta}{2\delta-1}>1.
    \end{equation*}
    Using $u^\delta u_x=\frac{1}{\delta+1}\partial_x(u^{\delta+1})$ and $\beta u(1-u^\delta)(u^\delta-\gamma)= \beta(1+\gamma)u^{\delta+1}- \beta\gamma u-\beta u^{2\delta+1},$ we write $u_t=H_1+H_2,$ where
    \begin{align*}
        H_1:=u_{xx}-\frac{\alpha}{\delta+1}\partial_x(u^{\delta+1})+\beta(1+\gamma)u^{\delta+1}-\beta\gamma u+\chi_\omega f,\,\,\,
        H_2:=-\beta u^{2\delta+1}.
    \end{align*}
    Since $u\in L^2(0,T;H_0^1(0,1))\cap L^{\widetilde{p}}((0,T)\times(0,1)),$ we have $u^{\delta+1}\in L^2((0,T)\times(0,1))$. Moreover, for $\varphi\in H_0^1(0,1)$,
    \begin{equation*}    
        \langle u_{xx},\varphi\rangle=-\int_0^1u_x\varphi_x\,dx,\qquad\langle\partial_x(u^{\delta+1}),\varphi\rangle=-\int_0^1u^{\delta+1}\varphi_x\,dx.
    \end{equation*}
    Hence
    \begin{equation*}
        H_1\in L^2(0,T;H^{-1}(0,1)).
    \end{equation*}
    Since $(2\delta+1)\widetilde{p}'=\widetilde{p}$, we also have $ H_2\in L^{\widetilde{p}'}((0,T)\times(0,1)).$ Furthermore, $H_0^1(0,1)\hookrightarrow L^{\widetilde{p}}(0,1)$, and therefore $ L^{\widetilde{p}'}(0,1)\hookrightarrow H^{-1}(0,1).$ Since $L^2(0,1)\hookrightarrow H^{-1}(0,1)$ and $u\in L^\infty(0,T;L^2(0,1))$, the assumption $ u_t\in L^{\widetilde{q}}(0,T;H^{-1}(0,1)) $ implies
    \begin{equation*}
        u\in W^{1,\widetilde{q}}(0,T;H^{-1}(0,1))\hookrightarrow C([0,T];H^{-1}(0,1)).
    \end{equation*}
    Let $U$ denote this representative. Then, for $0\le s\le t\le T$,
    \begin{equation*}
        U(t)-U(s)=\int_s^t u_t(r)\,dr\qquad\text{in }H^{-1}(0,1),
    \end{equation*}
    and therefore
    \begin{equation*}
        \|U(t)-U(s)\|_{H^{-1}}\le|t-s|^{1-\frac{1}{\widetilde{q}}}\|u_t\|_{L^{\widetilde{q}}(0,T;H^{-1})}.
    \end{equation*}
    We next show that $U\in C_w([0,T];L^2(0,1)).$ Let $ M:=\|u\|_{L^\infty(0,T;L^2)}.$ There exists a set $E\subset(0,T)$ of full measure such that $U(t)=u(t)\in L^2(0,1)$ and $\|U(t)\|_{L^2}\le M$ for every $t\in E$. Fix $t\in[0,T]$ and choose $s_j\in E$ with $s_j\to t$. A subsequence of $U(s_j)$ converges weakly in $L^2(0,1)$ to some $w\in L^2(0,1)$. Since $U(s_j)\to U(t)$ in $H^{-1}(0,1)$ and $L^2(0,1)\hookrightarrow H^{-1}(0,1)$ injectively, we obtain $w=U(t)$. Hence
    \begin{equation*}
        U(t)\in L^2(0,1),\qquad\|U(t)\|_{L^2}\le M\qquad\text{for every }t\in[0,T].
    \end{equation*}
    Now let $(t_j)$ be an arbitrary sequence in $[0,T]$ with $t_j\to t$. Then $\|U(t_j)\|_{L^2}\le M$ for every $j$. Thus, for $v\in L^2(0,1)$ and $\varphi\in H_0^1(0,1)$,
    \begin{equation*}
        |(U(t_j)-U(t),v)_{L^2}|\le{}\|U(t_j)-U(t)\|_{H^{-1}}\|\varphi\|_{H_0^1} +2M\|v-\varphi\|_{L^2}.
    \end{equation*}
    Using the density of $H_0^1(0,1)$ in $L^2(0,1)$ and then letting $j\to\infty$, we obtain $ U(t_j)\rightharpoonup U(t)$ in $L^2(0,1).$ Therefore $ U\in C_w([0,T];L^2(0,1)).$ We also identify the initial trace. For every $\varphi\in H_0^1(0,1)$,
    \begin{equation*}
        \langle U(t)-U(0),\varphi\rangle=\int_0^t\langle u_t(r),\varphi\rangle\,dr.
    \end{equation*}
    On the other hand, integrating the weak formulation over $(0,t)$ and using the prescribed initial datum gives
    \begin{equation*}
        \langle U(t),\varphi\rangle-(u_0,\varphi)_{L^2}=\int_0^t\langle u_t(r),\varphi\rangle\,dr.
    \end{equation*}
    Subtracting the two identities yields
    \begin{equation*}
        \langle U(0)-u_0,\varphi\rangle=0\qquad \text{for every }\varphi\in H_0^1(0,1).
    \end{equation*}
    Hence $U(0)=u_0$ in $H^{-1}(0,1)$. Since $U(0),u_0\in L^2(0,1)$ and $L^2(0,1)\hookrightarrow H^{-1}(0,1)$ is injective,
    \begin{equation*}
        U(0)=u_0\qquad\text{in }L^2(0,1).
    \end{equation*}
    It remains to prove strong continuity. Extend $U$ to $(-T,2T)$ by reflection:
        \begin{align*}
            \widetilde U(t)=
        \begin{cases}
            U(-t),&-T<t<0,\\
            U(t),&0\le t\le T,\\
            U(2T-t),&T<t<2T,
        \end{cases}
        \end{align*}
    and, for $i=1,2$, define
        \begin{align*}
            \widetilde H_i(t):=
        \begin{cases}
            -H_i(-t),&-T<t<0,\\
             H_i(t),&0<t<T,\\
            -H_i(2T-t),&T<t<2T.
        \end{cases}
        \end{align*}
    Since $U$ is continuous in $H^{-1}(0,1)$, no jump terms occur at $t=0,T$, and therefore $\widetilde U_t=\widetilde H_1+\widetilde H_2$ in the sense of distributions on $(-T,2T)$. Let $\rho\in C_c^\infty(-1,1)$ satisfy
    \begin{equation*}
        \rho\ge0,\qquad\int_{\mathbb R}\rho(r)\,dr=1,\,\,\,\text{and define}\,\,\, \rho_\varepsilon(t):=\frac1\varepsilon\rho\!\left(\frac{t}{\varepsilon}\right).
    \end{equation*}
    For $\varepsilon_m\downarrow0$, with $\varepsilon_m<T/2$, set
    \begin{equation*}
        u_m:=\rho_{\varepsilon_m}*\widetilde U,\qquad H_{1,m}:=\rho_{\varepsilon_m}*\widetilde H_1,\qquad H_{2,m}:=\rho_{\varepsilon_m}*\widetilde H_2.
    \end{equation*}
    Then, on $[0,T]$, $(u_m)_t=H_{1,m}+H_{2,m},$ and the approximate-identity property gives
    \begin{align*}
        u_m&\to u &&\text{in }L^2(0,T;H_0^1(0,1)) \cap L^{\widetilde{p}}((0,T)\times(0,1)),\\
        H_{1,m}&\to H_1 &&\text{in }L^2(0,T;H^{-1}(0,1)),\\
        H_{2,m}&\to H_2 &&\text{in }L^{\widetilde{p}'}((0,T)\times(0,1)).
    \end{align*}
    In particular,
    \begin{equation*}
        u_m\to u\qquad\text{in }L^2(0,T;L^2(0,1)).
    \end{equation*}
    For $m,n\in\mathbb N$, set
    \begin{equation*}
        z:=u_m-u_n,\qquad\mathfrak{h}_1:=H_{1,m}-H_{1,n},\qquad\mathfrak{h}_2:=H_{2,m}-H_{2,n}.
    \end{equation*}
    Then $z_t=\mathfrak{h}_1+\mathfrak{h}_2$, and since $z$ is smooth in time,
    \begin{equation*}
        \frac{d}{dt}\|z(t)\|_{L^2}^2=2\langle \mathfrak{h}_1(t),z(t)\rangle+2\int_0^1\mathfrak{h}_2(t,x)z(t,x)\,dx.
    \end{equation*}
    Integrating between $s$ and $t$ and using Cauchy-Schwarz and H\"older inequalities, we obtain
    \begin{align*}
        \|z(t)\|_{L^2}^2\le{}&\|z(s)\|_{L^2}^2 +2\|\mathfrak{h}_1\|_{L^2(0,T;H^{-1})} \|z\|_{L^2(0,T;H_0^1)}\\
        &+2\|\mathfrak{h}_2\|_{L^{\widetilde{p}'}((0,T)\times(0,1))} \|z\|_{L^{\widetilde{p}}((0,T)\times(0,1))}.
    \end{align*}
    Averaging with respect to $s\in(0,T)$ and taking the supremum over $t\in[0,T]$ yields
    \begin{align*}
        \sup_{t\in[0,T]}\|u_m(t)-u_n(t)\|_{L^2}^2\le{}&\frac1T \|u_m-u_n\|_{L^2(0,T;L^2)}^2\\
        &+2\|H_{1,m}-H_{1,n}\|_{L^2(0,T;H^{-1})}
        \|u_m-u_n\|_{L^2(0,T;H_0^1)}\\
        &+2\|H_{2,m}-H_{2,n}\|_{L^{\widetilde{p}'}((0,T)\times(0,1))}
        \|u_m-u_n\|_{L^{\widetilde{p}}((0,T)\times(0,1))}.
    \end{align*}
    Hence $(u_m)$ is Cauchy in $C([0,T];L^2(0,1))$, and therefore $ u_m\to\widehat u$ in $C([0,T];L^2(0,1))$ for some $\widehat u\in C([0,T];L^2(0,1))$. Since $u_m\to u$ also in $L^2(0,T;L^2(0,1))$, we have $\widehat u=u$ almost everywhere. Thus $\widehat u$ is a strongly continuous representative of the weak solution.

    Finally, $\widehat u$ and $U$ are both continuous from $[0,T]$ into $H^{-1}(0,1)$ and agree almost everywhere; hence they agree everywhere on $[0,T]$. In particular, $\widehat u(0)=U(0)=u_0.$ The continuous representative is unique, and we continue to denote it by $u$. Therefore $u\in C([0,T];L^2(0,1)).$
\end{proof}
    
\begin{proof}[Proof of Lemma~\ref{lem:left_ivp}]
    Set $y_1 = u$, $y_2 = u_x$, and $Y = (y_1, y_2)^\top$. The initial-value problem \eqref{ivp_ss} is then equivalent to the first-order system
    \begin{equation}
    \label{first_order_system}
        \begin{cases}
            Y' = G(Y), \\[1mm]
            Y(0) = \csmat{0}{s},
        \end{cases}
    \end{equation}
    where $G : \mathbb{R}^2 \to \mathbb{R}^2$ is given by $G(y_1,y_2):=({y_2},\,{\alpha\,y_1^\delta\,y_2-\beta\,y_1(1-y_1^\delta)(y_1^\delta-\gamma)})^{\top}.$ Since $G\in C^1(\mathbb{R}^2;\mathbb{R}^2)$, the Picard-Lindel\"{o}f theorem guarantees that for each $s \in \mathbb{R}$ there exists a unique maximal solution
    \begin{equation*}
        Y = (y_1,\, y_2)^\top\in C^1\bigl([0, l_{\max}(s));\, \mathbb{R}^2\bigr).
    \end{equation*}
    In particular, $u:=y_1\in C^2\bigl([0,l_{\max}(s))\bigr)$. Fix $r>0$ and set $B_r(0):=\{Y\in\mathbb{R}^2:|Y|\le r\}$. Since $G(0)=0$ and $G\in C^1$, the mean value theorem yields
    \begin{equation}
    \label{linear_growth}
        |G(Y)|\le C_r\,|Y|, \qquad \text{for all } Y \in B_r(0),
    \end{equation}
    where $C_r:=\sup_{Y\in B_r(0)}\|DG(Y)\|<\infty.$ Choose $s_0>0$ such that $s_0<r\,e^{-C_rb}<r$. Fix any $s\in\mathbb{R}$ with $|s|<s_0,$ so that $|Y(0)|=|s|<s_0<r$. We wish to show $l_{\max}(s)>b$. So assume for contradiction that
    \begin{equation}
    \label{contra_assumption}
        l_{\max}(s) \le b.
    \end{equation}
    Define $W(x):=|Y(x)|^2$, for all $x \in [0, l_{\max}(s)),$ and $\mathcal{\mathcal{G}}:=\big\{x\in[0,l_{\max}(s)):|Y(\xi)|<r$ for all $\xi\in[0,x]\big\}$. Since $|Y(0)|=|s|<r$, we have $0\in \mathcal{G}$, so $\mathcal{G}$ is nonempty and the supremum $x^*:=\sup\mathcal{G}$ is well defined. Moreover, by continuity of $Y$, there exists $\tilde{\varepsilon}>0$ such that $x^*\ge \tilde{\varepsilon}>0$, and clearly $x^*\le l_{\max}(s)$.
    \begin{equation}
    \label{Ys_bound}
         \hspace{-4.25cm}\text{claim:}\qquad\qquad\qquad|Y(x)| \le |s|\,e^{C_r x}<r,\qquad \text{for all } x\in [0, x^*).
    \end{equation}
    \textit{Proof of Claim.}
    Fix $x \in [0, x^*)$. By definition of $x^*$, we have $|Y(\xi)| < r$ for all $\xi \in [0,x]$, so \eqref{linear_growth} gives $|G(Y(\xi))| \le C_r |Y(\xi)|$ on $[0, x]$. Differentiating $W$ and applying the Cauchy-Schwarz inequality:
    \begin{equation*}
        W'(\xi)= 2\langle Y(\xi),\, Y'(\xi)\rangle= 2\langle Y(\xi),\, G(Y(\xi))\rangle\le 2|Y(\xi)|\,|G(Y(\xi))|\le 2C_r\, W(\xi).
    \end{equation*}
    Gr\"onwall's inequality on $[0,x]$ gives $W(x)\le W(0)\,e^{2C_r x}$, so $|Y(x)|\le |s|\,e^{C_r x}.$ Since $x<x^*\le l_{\max}(s) \le b$ (using \eqref{contra_assumption}), we have $e^{C_r x}<e^{C_r b}$, and therefore
    \begin{equation*}
        |Y(x)|\le |s|\,e^{C_r x}<|s|\, e^{C_r b}<s_0\, e^{C_r b}<r,
    \end{equation*}
    where the last inequality uses $s_0 < r\,e^{-C_r b}$. This proves the claim. \hfill$\square_{\mathrm{claim}}$
    \medskip
    
    We now rule out $x^* < l_{\max}(s)$. Suppose $x^* < l_{\max}(s)$. Passing to the limit $x \uparrow x^*$ in \eqref{Ys_bound} and using continuity of $Y$, we get $|Y(x^*)|\le|s|\,e^{C_rx^*}<r.$ Hence $Y(x^*)\in B_r(0)$. By continuity of $Y$, there exists $\tilde{\eta}>0$ such that $|Y(x)|<r$ for all $x\in [0,\,x^*+ \tilde{\eta}]$, contradicting the definition of $x^*$ as a supremum. Therefore $x^* = l_{\max}(s)$, and \eqref{Ys_bound} yields
    \begin{equation}
    \label{global_bound}
        |Y(x)|<r \qquad \text{for all } x\in[0, l_{\max}(s)).
    \end{equation}
    The bound \eqref{global_bound} shows that $Y$ remains in the compact set ${B_r(0)}$ on its entire maximal interval $[0, l_{\max}(s))$. Since $G \in C^1(\mathbb{R}^2;\mathbb{R}^2)$, the standard continuation theorem for ODEs states that if a maximal solution remains in a compact set, then it can be extended beyond $l_{\max}(s)$, contradicting maximality. Therefore the assumption \eqref{contra_assumption} is false, and $l_{\max}(s) > b.$ Consequently, $Y \in C^1\bigl([0,b];\,\mathbb{R}^2\bigr), u \in C^2\bigl([0,b]\bigr).$ Finally, for $0 < |s| < s_0$, the initial condition gives $y_2(0) = u_x(0) = s \neq 0$, whereas the trivial solution $u\equiv 0$ satisfies $u_x(0)=0$. Hence the solution $u$ is nontrivial.
\end{proof}
 
\begin{proof}[Proof of Lemma~\ref{lem:right_ivp}]
    The proof is analogous to that of Lemma~\ref{lem:left_ivp}. 
\end{proof}

\subsection{Proof of Theorem~\ref{Thm_Soln_Trans_sys}}
    The proof of Theorem~\ref{Thm_Soln_Trans_sys} relies on the Leray-Schauder fixed-point principle. Let us denote $Q: = (0,T) \times (0,1)$ and assume that $ 0 < \epsilon T \le 1 $. System \eqref{Close_loop_z_eqn_expanded} can be rewritten as
    \begin{align}
    \label{Closed_loop_N}
        \begin{cases}
            z_t-z_{xx}= \mathcal N(t, z),& (t,x)\in Q,\\
            z(t,0)=z(t,1)=0,& t\in(0,T),\\
            z(0,x)=0,& x\in(0,1),
        \end{cases}
    \end{align}
    where
    \begin{equation}
    \label{compact_form_nonlin}
        \mathcal N(t,z)=\displaystyle \sum_{j=0}^{\delta}a_j(t,x)z^jz_x+\xi(\epsilon t)z +\mcP(t,x,z)-\epsilon\overline u_\tau(\epsilon t,x).
    \end{equation}
    Here $a_0(t,x)=-\alpha\overline u^\delta(\epsilon t,x), a_1(t,x)=-\alpha\delta\overline u^{\delta-1}(\epsilon t,x),$ and, for $2\leq j\leq\delta$,
    \begin{equation*}
        a_j(t,x)=-\alpha\binom{\delta}{j}\overline u^{\delta-j}(\epsilon t,x),
    \end{equation*}
    while
    \begin{align*}
        \mcP(t,x,s)={}&F\bigl(s+\overline u(\epsilon t,x)\bigr) -F\bigl(\overline u(\epsilon t,x)\bigr)-\alpha\delta\overline u^{\delta-1}(\epsilon t,x)\overline u_x(\epsilon t,x)s\\
        &-\alpha\sum_{j=2}^{\delta}\binom{\delta}{j}s^j\overline u^{\delta-j}(\epsilon t,x)\overline u_x(\epsilon t,x).
    \end{align*}
    By Remark~\ref{rem:p_q_u_regularity} and~\ref{Regularity_of_xi}, there exists a constant $C>0$ such that
    \begin{equation}
    \label{coefficient_uniform_bound}
        \max_{0\leq j\leq\delta}\left(\|a_j\|_{L^\infty(Q)}+
        \|(a_j)_x\|_{L^\infty(Q)}\right)+\|\epsilon\overline u_\tau\|_{L^\infty(Q)}+\sup_{\tau\in[0,1]}\|\xi(\tau)\|_{\mathcal L(L^2(0,1))}\leq C.
    \end{equation}
    Furthermore, uniformly with respect to $(t,x)\in Q$ and $r,s\in\mathbb R$,
    \begin{align}
        |\mathcal P(t,x,s)|&\leq C\bigl(|s|+|s|^{2\delta+1}\bigr),\label{eq:P_growth_appendix}\\
        \mathcal P(t,x,s)|s|^{2\delta}s&\leq
        -\frac{\beta}{2}|s|^{4\delta+2}+C, \label{eq:P_dissipative_appendix}\\
        |\mathcal P(t,x,r)-\mathcal P(t,x,s)|&\leq C\bigl(1+|r|^{2\delta}+|s|^{2\delta}\bigr)|r-s|. \label{eq:P_lipschitz_appendix}
    \end{align}
    In order to apply the Leray-Schauder fixed-point theorem, we consider the homotopy (auxiliary) equation parameterized by $\lambda \in [0,1]$:
    \begin{align}
        \begin{cases}
            z_t-z_{xx}= \mathcal \lambda \mathcal{N}(t, z),
            & (t,x)\in Q,\\
            z(t,0)=z(t,1)=0,
            & t\in(0,T),\\
            z(0,x)=0,
            & x\in(0,1),
        \end{cases}
    \end{align}
    Define $\mathcal{X}_T := L^{2(2\delta + 1)}(0,T; H_0^1(0,1)).$ Let us consider the mapping $\Phi:\mathcal X_T\times[0,1]\to\mathcal X_T$ such that for $\widetilde z\in\mathcal X_T$ and $\lambda\in[0,1]$, $z=\Phi(\widetilde z,\lambda)$ if and only if $z$ is the solution to the linear problem
    \begin{align}
    \label{Phi_linear_problem}
        \begin{cases}
            z_t-z_{xx}=\lambda\mathcal N(t,\widetilde z),
            & (t,x)\in Q,\\
            z(t,0)=z(t,1)=0,& t\in(0,T),\\
            z(0,x)=0, & x\in(0,1).
        \end{cases}
    \end{align}
    First, we prove the following lemma.
\begin{lem}
\label{lem:Phi_well_defined_cont_compact}
    The mapping $\Phi:\mathcal X_T\times[0,1]\to\mathcal X_T$ is well-defined, continuous, and compact.
\end{lem}
\begin{proof}
    We begin by showing that $\Phi$ is well-defined. Let $\widetilde z\in\mathcal X_T.$  We first show that $\mathcal N(t,\widetilde z)\in L^2(Q)$. Using the continuous embedding $H_0^1(0,1)\hookrightarrow L^\infty(0,1)$ it is easy to see that for $0\leq j\leq\delta$, since $2(j+1)\leq2(\delta+1)\leq 2(2\delta + 1)$, we have
    \begin{equation*}
        \int_0^T\|\widetilde z^j\widetilde z_x\|_{L^2(0,1)}^2\,dt\leq C\int_0^T \|\widetilde z(t)\|_{H_0^1}^{2(j+1)}\,dt<\infty,
    \end{equation*}
    Similarly, for $1\leq k\leq2\delta+1$, since $2k\leq 2(2\delta + 1)$, we have $\int_0^T\|\widetilde z^k\|_{L^2(0,1)}^2\,dt<\infty.$ Together with \eqref{coefficient_uniform_bound}, this yields $\mathcal N(t,\widetilde z)\in L^2(Q).$ Therefore, by the $L^2$-regularity theory for the heat equation with homogeneous Dirichlet boundary condition, problem \eqref{Phi_linear_problem} admits a unique solution satisfying
    \begin{equation*}
        z\in L^2(0,T;H^2(0,1)\cap H_0^1(0,1))\cap C([0,T];H_0^1(0,1)),\qquad z_t\in L^2(Q).
    \end{equation*}
    Thus $\Phi(\widetilde z,\lambda)\in\mathcal X_T.$ Hence $\Phi$ is well-defined.

    Next, we prove that $\Phi$ is continuous. Let $\widetilde z_n\to \widetilde z \quad\text{in }\mathcal X_T$ and $ \lambda_n\to\lambda$ in $[0,1].$ Using
    \begin{equation*}
        |r^m-s^m|\le C_m\bigl(|r|^{m-1}+|s|^{m-1}\bigr)|r-s|, \qquad r,s\in\mathbb R,\quad m\in \mathbb{N},
    \end{equation*}
    \eqref{eq:P_lipschitz_appendix}, the embedding $H_0^1(0,1)\hookrightarrow L^\infty(0,1)$, and H\"older's inequality in time, we obtain, for every $R>0$,
    \begin{equation}
    \label{eq:N_local_lipschitz_appendix}
        \|\mathcal N(\cdot,v)-\mathcal N(\cdot,w)\|_{L^2(Q)}\leq C_R\|v-w\|_{X_T},
    \end{equation}
    whenever $\|v\|_{X_T}+\|w\|_{X_T}\leq R.$ Indeed, for the derivative terms, $v^jv_x-w^jw_x=(v^j-w^j)v_x+w^j(v_x-w_x),$ and hence $\|v^jv_x-w^jw_x\|_{L^2}\leq C\bigl(\|v\|_{H_0^1}^j +\|w\|_{H_0^1}^j\bigr)\|v-w\|_{H_0^1}.$ Estimate \eqref{eq:N_local_lipschitz_appendix} then follows from $4\delta+2\geq2(j+1)$ and $4\delta+2\geq2(2\delta+1)$. Using \eqref{eq:N_local_lipschitz_appendix}, we obtain
    \begin{equation*}
        \mathcal N(t,\widetilde z_n)\to\mathcal N(t,\widetilde z)\quad\text{in }L^2(Q).
    \end{equation*}
    $\lambda_n\to\lambda$, we get $\lambda_n\mathcal N(t,\widetilde z_n)\to \lambda\mathcal N(t,\widetilde z)$ in $L^2(Q).$ Now $z_n-z$ solves
    \begin{align*}
        \begin{cases}
            (z_n-z)_t-(z_n-z)_{xx}=\lambda_n\mathcal N(t,\widetilde z_n)-\lambda\mathcal N(t,\widetilde z),& (t,x)\in Q,\\
            z_n-z=0,& t\in(0,T),\ x=0,1,\\
            (z_n-z)(0,x)=0,& x\in(0,1).
        \end{cases}
    \end{align*}
    By the $L^2$-regularity estimate for the heat equation, 
    \begin{equation*}
        z_n-z\to 0\ \text{in}\ L^2(0,T;H^2(0,1)\cap H_0^1(0,1)) \cap H^1(0,T;L^2(0,1)).
    \end{equation*}
    In particular, $z_n\to z$ in $C([0,T];H_0^1(0,1)),$ which implies  $z_n\to z$ in $\mathcal X_T$. Hence $\Phi$ is continuous.

    We now finally show that $\Phi$ is compact. Let $\{(\widetilde z_n,\lambda_n)\}$ be a bounded sequence in $\mathcal X_T\times[0,1].$ Set $z_n:=\Phi(\widetilde z_n,\lambda_n).$ Since $ \{\widetilde{z}_n\} $ is bounded in $\mathcal{X}_T$, the estimates established earlier (in the proof of well-definedness and continuity) imply that $\{\mathcal N(t,\widetilde z_n)\}$ is bounded in $L^2(Q)$. Since $\lambda_n\in[0,1]$, it follows that $\{\lambda_n\mathcal N(t,\widetilde z_n)\}$ is also bounded in $L^2(Q)$. By the $L^2$-regularity estimate for the heat equation, the sequence $\{z_n\}$ is bounded in $L^2(0,T;H^2(0,1)\cap H_0^1(0,1))\cap H^1(0,T;L^2(0,1))$. Moreover, $\{z_n\}$ is bounded in $C([0,T];H_0^1(0,1))$.

    By the Aubin-Lions lemma, $L^2\bigl(0,T;H^2(0,1)\cap H_0^1(0,1)\bigr)\cap H^1\bigl(0,T;L^2(0,1)\bigr)$ is compactly embedded in $L^2\bigl(0,T;H_0^1(0,1)\bigr).$ Therefore, up to a subsequence,
    \begin{equation*}
        z_n\to z \quad\text{strongly in }L^2(0,T;H_0^1(0,1)).
    \end{equation*}
    Since $\{z_n\}$ is bounded in $C([0,T];H_0^1(0,1))$, there exists $K>0$ such that $\|z_n(t)\|_{H_0^1(0,1)}\le K$ for all $t\in[0,T]$
    and all $n$. Passing to a subsequence if necessary, we may also assume that
    \begin{equation*}
        z_n\rightharpoonup^\ast z \quad\text{in }L^\infty(0,T;H_0^1(0,1)),
    \end{equation*}
    and hence $\|z(t)\|_{H_0^1(0,1)}\le K$ for a.e. $t\in(0,T)$. Consequently, $\|z_n(t)-z(t)\|_{H_0^1(0,1)}\le 2K$ for a.e. $t\in(0,T)$. Therefore,
    \begin{equation*}
        \|z_n-z\|_{\mathcal X_T}^{2(2\delta + 1)}\le (2K)^{2(2\delta + 1)-2}\int_0^T \|z_n(t)-z(t)\|_{H_0^1(0,1)}^2\,dt.
    \end{equation*}
    Since $z_n\to z$ strongly in $L^2(0,T;H_0^1(0,1))$, the right-hand side tends to zero. Hence $z_n\to z$ strongly in $\mathcal X_T$. Thus every bounded sequence in the range of $\Phi$ has a convergent subsequence in $\mathcal X_T$. Therefore $\Phi$ is compact. Consequently, $\Phi:\mathcal X_T\times[0,1]\to\mathcal X_T$ is well-defined, continuous, and compact.
\end{proof}
    Next, we establish the following lemma.
\begin{lem}
\label{lem:uniform_bound_fixed_points}
Let $\widetilde{\Phi} := \left\{ z \in \mathcal{X}_T : \, z = \Phi(z,\lambda)\,\, \text{ for some } \lambda \in [0,1] \right\}.$
    Then there exists a constant $C_T>0$, independent of $\lambda\in[0,1]$ and of the corresponding fixed point $z$, such that
    \begin{equation*}
        \|z\|_{L^\infty(0,T;H_0^1(0,1))} +\|z\|_{L^2(0,T;H^2(0,1))}+\|z_t\|_{L^2(Q)}\le C_T.
    \end{equation*}
    Consequently, $\|z\|_{\mathcal X_T}\le C_T,$ and hence the set $\widetilde{\Phi}$ is uniformly bounded in $\mathcal X_T$.
\end{lem}
\begin{proof}
    Let $\lambda\in[0,1]$, and let $z\in\mathcal X_T$ satisfy $z=\Phi(z,\lambda).$ By the definition of $\Phi$ and Lemma~\ref{lem:Phi_well_defined_cont_compact}, the function $z$ satisfies
    \begin{equation}
    \label{homotopy_fixed_point_equation}
        \begin{cases}
            z_t-z_{xx}=\lambda\mathcal N(t,z),
            & (t,x)\in Q,\\
            z(t,0)=z(t,1)=0,& t\in(0,T),\\
            z(0,x)=0,& x\in(0,1),
        \end{cases}
    \end{equation}
    and has the regularity
    \begin{equation}
    \label{fixed_point_regularities}
            z\in{} L^2\bigl(0,T;H^2(0,1)\cap H_0^1(0,1)\bigr) \cap H^1\bigl(0,T;L^2(0,1)\bigr) \cap C\bigl([0,T];H_0^1(0,1)\bigr).
    \end{equation}
    Throughout the proof, $C>0$ denotes a generic constant independent of $\lambda$ and $z$, which may change from one line to another. We first justify the test functions used below. From \eqref{fixed_point_regularities}, we obtain $\|z\|_{L^\infty(Q)}\le C\|z\|_{L^\infty(0,T;H_0^1(0,1))}<\infty.$ Consequently,
    \begin{equation*}
        |z|^{2\delta}z\in L^2(0,T;H_0^1(0,1)),\quad \text{and}\quad \partial_x\bigl(|z|^{2\delta}z\bigr)=(2\delta+1)|z|^{2\delta}z_x.
    \end{equation*}
    Therefore, $|z|^{2\delta}z$ is an admissible test function in \eqref{homotopy_fixed_point_equation}. The Sobolev chain rule gives
    \begin{equation}
    \label{time_chain_rule}
        \int_0^1z_t|z|^{2\delta}z\,dx=\frac{1}{2(\delta+1)}\frac{d}{dt}\|z(t)\|_{L^{2(\delta+1)}(0,1)}^{2(\delta+1)}
    \end{equation}
    for almost every $t\in(0,T)$. Similarly, since $z_{xx},z_t\in L^2(Q)$, we may test \eqref{homotopy_fixed_point_equation} by $-z_{xx}$. The Hilbert-space chain rule yields
    \begin{equation}
    \label{H1_chain_rule}
        \int_0^1z_t(-z_{xx})\,dx=\frac12\frac{d}{dt}\|z_x(t)\|_{L^2(0,1)}^2
    \end{equation}
    for almost every $t\in(0,T)$. We now establish estimates independent of $\lambda$.Testing \eqref{homotopy_fixed_point_equation} by $|z|^{2\delta}z$, and using \eqref{time_chain_rule}, we obtain
    \begin{align*}
        &\frac{1}{2(\delta+1)}\frac{d}{dt}\|z(t)\|_{L^{2(\delta+1)}}^{2(\delta+1)}+(2\delta+1)\int_0^1|z|^{2\delta}|z_x|^2\,dx\\
        &\qquad =\lambda\int_0^1 \bigg(\sum_{j=0}^{\delta}a_jz^jz_x +\xi(\epsilon t)z +\mathcal P(t,x,z) -\epsilon\overline u_\tau\bigg) |z|^{2\delta}z\,dx .
    \end{align*}
    For $0\leq j\leq\delta$, integration by parts gives
    \begin{equation*}
        \int_0^1 a_jz^jz_x|z|^{2\delta}z\,dx = -\int_0^1 (a_j)_x \left( \int_0^z r^j|r|^{2\delta}r\,dr \right)dx,
    \end{equation*}
    since $z=0$ at $x=0,1$. Therefore, for every $\eta>0$,
    \begin{equation*}
        \left|\int_0^1 a_jz^jz_x|z|^{2\delta}z\,dx\right|\leq\eta \|z\|_{L^{4\delta+2}}^{4\delta+2}+C_\eta,
    \end{equation*}
    because $2\delta+j+2<4\delta+2.$ Moreover, by \eqref{eq:P_dissipative_appendix},
    \begin{equation*}
        \int_0^1\mathcal P(t,x,z)|z|^{2\delta}z\,dx\leq -\frac{\beta}{2}\|z\|_{L^{4\delta+2}}^{4\delta+2}+C.
    \end{equation*}
    Using \eqref{coefficient_uniform_bound} and Young's inequality, we obtain
    \begin{equation*}
        \left|\int_0^1(\xi(\epsilon t)z) |z|^{2\delta}z\,dx\right|\leq C\|z\|_{L^2} \|z\|_{L^{4\delta+2}}^{2\delta+1}\leq\eta\|z\|_{L^{4\delta+2}}^{4\delta+2}+ C_\eta\|z\|_{L^2}^2.
    \end{equation*}
    Similarly,
    \begin{equation*}
        \Big|\int_0^1 \epsilon\overline{u}_\tau(\epsilon t, x)|z|^{2\delta}z\,dx\Big|\leq \|\epsilon \overline{u}_\tau(\epsilon t,\cdot)\|_{L^2} \|z\|_{L^{4\delta+2}}^{2\delta+1} \leq \eta\|z\|_{L^{4\delta+2}}^{4\delta+2}+C_\eta.
    \end{equation*}
    Choosing $\eta>0$ sufficiently small, we find constants $c_1,c_2,C>0$, independent of $\lambda$, such that
    \begin{equation}
    \label{first_energy_inequality}
        \frac{d}{dt}\|z(t)\|_{L^{2(\delta+1)}}^{2(\delta+1)}+ c_1\int_0^1|z|^{2\delta}|z_x|^2\,dx+c_2\lambda \|z(t)\|_{L^{4\delta+2}}^{4\delta+2}\le C\left(1+ \|z(t)\|_{L^{2(\delta+1)}}^{2(\delta+1)}\right).
    \end{equation}
    Since $z(0)=0$, Gr\"onwall's inequality and integration of \eqref{first_energy_inequality} yield
    \begin{equation}
    \label{first_uniform_bound}
        \sup_{0\leq t\leq T} \|z(t)\|_{L^{2(\delta+1)}}^{2(\delta+1)}+\int_Q |z|^{2\delta}|z_x|^2\,dx\,dt+\lambda\int_0^T\|z(t)\|_{L^{4\delta+2}}^{4\delta+2}\,dt\leq C_T.
    \end{equation}
    Testing \eqref{homotopy_fixed_point_equation} by $-z_{xx}$, and using \eqref{H1_chain_rule}, we obtain
    \begin{align}
    \label{second_energy_identity}
        \frac12\frac{d}{dt}\|z_x(t)\|_{L^2}^2+\|z_{xx}(t)\|_{L^2}^2&=-\lambda\sum_{j=0}^{\delta} \int_0^1a_jz^jz_xz_{xx}\,dx-\lambda(\xi(\epsilon t)z,z_{xx})_{L^2}\nonumber\\
        &\qquad-\lambda(\mcP(t,\cdot,z),z_{xx})_{L^2} +\lambda(\epsilon \overline{u}_\tau(\epsilon t,\cdot),z_{xx})_{L^2}.
    \end{align}
    For $0\le j\le\delta$, Young's inequality gives
    \begin{equation*}
        \lambda\left|\int_0^1a_jz^jz_xz_{xx}\,dx\right|\leq \eta\|z_{xx}\|_{L^2}^2+C_\eta\left(\|z_x\|_{L^2}^2+\int_0^1|z|^{2\delta}|z_x|^2\,dx\right),
    \end{equation*}
    Using \eqref{coefficient_uniform_bound}, Young's inequality, $0\le\lambda\le1$, and Poincaré's inequality, we have
    \begin{equation*}
        \lambda\big|(\xi(\epsilon t)z,z_{xx})_{L^2}\big|\le C\lambda\|z\|_{L^2}\|z_{xx}\|_{L^2}\le \eta\|z_{xx}\|_{L^2}^2 +C_\eta\lambda^2\|z\|_{L^2}^2\le\eta\|z_{xx}\|_{L^2}^2+C_\eta\|z_x\|_{L^2}^2.
    \end{equation*}
    Moreover, $\lambda\big|(\mcP(t,\cdot,z),z_{xx})_{L^2}\big|\le \eta\|z_{xx}\|_{L^2}^2+C_\eta\lambda^2\|\mcP(t,\cdot,z)\|_{L^2}^2$. By~\eqref{eq:P_growth_appendix}, we have
    \begin{equation*}
        \|\mathcal P(t,\cdot,z)\|_{L^2}^2\leq C\left(\|z(t)\|_{L^2}^2+\|z(t)\|_{L^{4\delta+2}}^{4\delta+2}\right).
    \end{equation*}
    Using, \eqref{first_uniform_bound}, we obtain
    \begin{equation}
    \label{eq:P_integral_appendix}
        \lambda^2\int_0^T\|\mcP(t,\cdot,z)\|_{L^2}^2\,dt\leq C_T.
    \end{equation}
    Finally,
    \begin{equation*}
        \lambda\big|(-\epsilon\overline{u}_\tau(\epsilon t,\cdot),z_{xx})_{L^2}\big|\le\eta\|z_{xx}\|_{L^2}^2+
        C_\eta\lambda^2\|-\epsilon\overline{u}_\tau(\epsilon t,\cdot)\|_{L^2}^2\le\eta\|z_{xx}\|_{L^2}^2+C_\eta.
    \end{equation*}
    Choosing $\eta>0$ sufficiently small, we obtain
    \begin{equation*}
        \frac{d}{dt}\|z_x(t)\|_{L^2}^2+\|z_{xx}(t)\|_{L^2}^2 \leq{} C\|z_x(t)\|_{L^2}^2+C\int_0^1 |z|^{2\delta}|z_x|^2\,dx
        +C\lambda^2\|\mathcal P(t,\cdot,z)\|_{L^2}^2+C.
    \end{equation*}
    Since $z(0)=0$ in $H_0^1(0,1)$, we have $z_x(0)=0$ in $L^2(0,1),$ Gr\"{o}nwall's inequality together with \eqref{first_uniform_bound} and \eqref{eq:P_integral_appendix} yields
    \begin{equation}
    \label{H1_H2_uniform_estimate}
        \sup_{0\leq t\leq T}\|z(t)\|_{H_0^1(0,1)}^2+\int_0^T \|z(t)\|_{H^2(0,1)}^2\,dt\leq C_T.
    \end{equation}
    Here we have used Poincar\'e's inequality and
    \begin{equation*}
        \|v\|_{H^2(0,1)}\leq C\|v_{xx}\|_{L^2(0,1)},\qquad v\in H^2(0,1)\cap H_0^1(0,1).
    \end{equation*}
    Finally, using \eqref{homotopy_fixed_point_equation}, \eqref{eq:P_integral_appendix}, \eqref{H1_H2_uniform_estimate}, the embedding $H_0^1(0,1)\hookrightarrow L^\infty(0,1)$, and the uniform bounds in \eqref{coefficient_uniform_bound}, we obtain $\|z_t\|_{L^2(Q)}\leq C_T.$
    Consequently
    \begin{equation*}
        \|z\|_{L^\infty(0,T;H_0^1(0,1))}+ \|z\|_{L^2(0,T;H^2(0,1))}+\|z_t\|_{L^2(Q)}\le C_T,
    \end{equation*}
    where $C_T$ is independent of $\lambda\in[0,1]$ and of the corresponding fixed point $z$. Finally,
    \begin{equation*}
        \|z\|_{\mathcal X_T}^{2(2\delta + 1)}=\int_0^T \|z(t)\|_{H_0^1(0,1)}^{2(2\delta + 1)}\,dt\le T \Big( \sup_{0\le t\le T} \|z(t)\|_{H_0^1(0,1)} \Big)^{2(2\delta + 1)} \le C_T.
    \end{equation*}
    Therefore, $\|z\|_{\mathcal X_T}\le C_T.$ Thus, $\widetilde{\Phi}$ is bounded in $\mathcal X_T$.
\end{proof}
    We are now in a position to prove Theorem~\ref{Thm_Soln_Trans_sys}.
\begin{proof}[Proof of Theorem~\ref{Thm_Soln_Trans_sys}]
    By Lemma~\ref{lem:Phi_well_defined_cont_compact}, the mapping
    \begin{equation*}
        \Phi:\mathcal X_T\times[0,1]\longrightarrow\mathcal X_T
    \end{equation*}
    is continuous and compact. Moreover, setting $\lambda=0$ in \eqref{Phi_linear_problem} shows that $\Phi(\widetilde z,0)$ solves the homogeneous heat equation with zero initial and boundary data, so that $\Phi(\widetilde z,0)=0$ for every $\widetilde z\in\mathcal X_T.$ By Lemma~\ref{lem:uniform_bound_fixed_points}, the set $\widetilde{\Phi}$ is bounded in $\mathcal X_T$. Hence, by the Leray-Schauder fixed-point theorem, there exists $z\in\mathcal X_T$ such that $z=\Phi(z,1)$. Setting $\lambda=1$ and $\widetilde z=z$ in \eqref{Phi_linear_problem}, this fixed point is a solution of the original nonlinear closed-loop problem \eqref{Close_loop_z_eqn_expanded}. Furthermore, by the regularity established in Lemma~\ref{lem:Phi_well_defined_cont_compact},
    \begin{equation*}
        z\in L^2\bigl(0,T;H^2(0,1)\cap H_0^1(0,1)\bigr) \cap H^1\bigl(0,T;L^2(0,1)\bigr) \cap C\bigl([0,T];H_0^1(0,1)\bigr).
    \end{equation*}
    It remains to establish the uniqueness of the solution. Let $z_1,z_2\in\mathcal X_T$ be two fixed points of $\Phi(\cdot,1)$. By the definition of $\Phi$, both $z_1$ and $z_2$ solve the original nonlinear problem and satisfy
    \begin{equation*}
        z_i(t,0)=z_i(t,1)=0,\qquad z_i(0,x)=0,\qquad i=1,2.
    \end{equation*}
    Moreover, by Lemma~\ref{lem:Phi_well_defined_cont_compact},
    \begin{equation*}
        z_i\in L^2\bigl(0,T;H^2(0,1)\cap H_0^1(0,1)\bigr)\cap H^1\bigl(0,T;L^2(0,1)\bigr)\cap C\bigl([0,T];H_0^1(0,1)\bigr),
    \end{equation*}
    for $i=1,2$. In particular, $z_1,z_2\in L^\infty(Q)$. Set $ M:=\max\left\{\|z_1\|_{L^\infty(Q)},\|z_2\|_{L^\infty(Q)} \right\},$ and $w:=z_1-z_2.$ Then
        $$w(t,0)=w(t,1)=0,\qquad w(0,x)=0,$$
    and $w$ satisfies
    \begin{equation}
    \label{difference_equation}
        \begin{cases}
            w_t-w_{xx}=\displaystyle\sum_{j=0}^{\delta}a_j(t,x)
            \left(z_1^j(z_1)_x-z_2^j(z_2)_x\right) +\xi(\epsilon t)w\\\hspace{3.6cm}
            +\mcP(t,x,z_1)-\mcP(t,x,z_2),& (t,x)\in Q,\\
            w(t,0)=w(t,1)=0,& t\in(0,T),\\
            w(0,x)=0, & x\in(0,1).
        \end{cases}
    \end{equation}
    Taking the $L^2(0,1)$-inner product of \eqref{difference_equation} with $w$, we obtain, for almost every $t\in(0,T)$,
    \begin{align}
    \label{difference_energy_identity}
        \frac12\frac{d}{dt}\|w(t)\|_{L^2}^2+\|w_x(t)\|_{L^2}^2={}&\sum_{j=0}^{\delta}\int_0^1 a_j\left(z_1^j(z_1)_x-z_2^j(z_2)_x\right)w\,dx\nonumber\\
        & \qquad \quad +(\xi(\epsilon t)w,w)_{L^2}+(\mcP(t,\cdot,z_1)-\mcP(t,\cdot,z_2),w)_{L^2}.
    \end{align}
    For $0\le j\le\delta$, integration by parts gives
    \begin{align*}
        \int_0^1a_j\left(z_1^j(z_1)_x-z_2^j(z_2)_x\right)w\,dx =&-\frac{1}{j+1}\int_0^1(a_j)_x\left(z_1^{j+1}-z_2^{j+1}\right)w\,dx\\
        &\qquad \qquad-\frac{1}{j+1}\int_0^1a_j\left(z_1^{j+1}-z_2^{j+1}\right)w_x\,dx.
    \end{align*}
    Since $|z_1^{j+1}-z_2^{j+1}| \leq C_M|w|,$ it follows from \eqref{coefficient_uniform_bound} and Young's inequality that, for every $\eta>0$,
    \begin{equation*}
        \left|\int_0^1a_j\left(z_1^j(z_1)_x-z_2^j(z_2)_x\right)w\,dx\right|\leq\eta\|w_x\|_{L^2}^2+C_{M,\eta}\|w\|_{L^2}^2.
    \end{equation*}
    Summing over $j=0,\ldots,\delta$ and choosing $\eta>0$ sufficiently small, we obtain
    \begin{equation}
    \label{difference_convection_estimate}
        \left|\sum_{j=0}^{\delta}\int_0^1a_j\left(z_1^j(z_1)_x-z_2^j(z_2)_x\right)w\,dx\right|\leq \frac12\|w_x\|_{L^2}^2+C_{M} \|w\|_{L^2}^2.
    \end{equation}
    In view of \eqref{coefficient_uniform_bound},
    \begin{equation}
    \label{difference_feedback_estimate}
        |(\xi(\epsilon t)w,w)_{L^2}|\leq \|\xi(\epsilon t)w\|_{L^2}\|w\|_{L^2}\leq C \|w\|_{L^2}^2.
    \end{equation}
    Finally, by \eqref{eq:P_lipschitz_appendix} and the definition of $M$,
    \begin{equation*}
        |\mcP(t,x,s_1)-\mcP(t,x,s_2)|\leq C_{M} |s_1-s_2|,
    \end{equation*}
    and therefore,
    \begin{equation}
    \label{difference_reaction_estimate}
        \left|(\mcP(t,\cdot,z_1)-\mcP(t,\cdot,z_2),w)_{L^2}\right|\leq C_{M} \|w\|_{L^2}^2.
    \end{equation}
    Combining~\eqref{difference_energy_identity}, \eqref{difference_convection_estimate}, \eqref{difference_feedback_estimate}, and \eqref{difference_reaction_estimate}, we obtain
    \begin{equation*}
        \frac12\frac{d}{dt}\|w(t)\|_{L^2}^2+ \frac12\|w_x(t)\|_{L^2}^2 \leq C_{M} \|w(t)\|_{L^2}^2.
    \end{equation*}
    Since both $z_1$ and $z_2$ are fixed points of $\Phi(\cdot,1)$, their initial conditions are zero. Hence $w(0)=z_1(0)-z_2(0)=0$. Gr\"{o}nwall's inequality therefore yields
    \begin{equation*}
        \|w(t)\|_{L^2}^2=0,\qquad t\in[0,T].
    \end{equation*}
    Thus $w\equiv0$ in $Q$ and consequently $z_1=z_2$. This completes the proof of Theorem~\ref{Thm_Soln_Trans_sys}.
\end{proof}

% \bibliographystyle{abbrv}
  %\bibliography{xyz}

\clearpage
\thispagestyle{empty}

\noindent\textbf{Running title:}
Approximate Controllability of the GBHE

\bigskip

\noindent\textbf{Corresponding author:}\\
Mrinmay Biswas\\
Assistant Professor\\
Department of Mathematics\\
Indian Institute of Technology Kanpur\\
Kanpur--208016, India\\
Email: mbiswas@iitk.ac.in
  
\end{document}